\documentclass[]{imsart}

\RequirePackage{amsthm,amsmath,amsfonts,amssymb}
\RequirePackage[numbers]{natbib}
\RequirePackage[colorlinks,citecolor=blue,urlcolor=blue]{hyperref}
\RequirePackage{graphicx}

\usepackage{array}
\usepackage[boxed]{algorithm2e}
\usepackage{subfig}
\usepackage[dvipsnames]{xcolor}
\usepackage{colortbl}
\usepackage{makecell}

\startlocaldefs
\newtheorem{theorem}{Theorem}[section]
\newtheorem{lemma}[theorem]{Lemma}
\newtheorem{corollary}[theorem]{Corollary}
\newtheorem{proposition}[theorem]{Proposition}
\theoremstyle{remark}
\newtheorem{definition}[theorem]{Definition}
\newtheorem{rem}[theorem]{Remark}
\newtheorem{assumption}[theorem]{Assumption}

\newtheorem*{notation}{Notations}

\newcommand{\R}{\mathbb{R}}
\newcommand{\N}{\mathbb{N}}
\newcommand{\1}{\mathbf{1}}
\newcommand{\E}{\mathbb{E}}
\newcommand{\PP}{\mathbb{P}}
\newcommand{\e}{\mathrm{e}}
\newcommand{\dd}{\mathrm{d}}
\newcommand{\pb}{\mathcal{P}}
\newcommand{\pdf}{\mathcal{P}^\mathrm{ac}}

\endlocaldefs

\begin{document}

\begin{frontmatter}
\title{Collective Proposal Distributions for Nonlinear MCMC samplers: Mean-Field Theory and Fast Implementation}

\runtitle{Collective Proposal Distributions for Nonlinear MCMC samplers}

\begin{aug}
\author[A]{\fnms{Grégoire} \snm{Clarté}\ead[label=e1]{gregoire.clarte@helsinki.fi}},
\author[B,C]{\fnms{Antoine} \snm{Diez}\ead[label=e2,mark]{diez.antoinenicolas.4e@kyoto-u.ac.jp}}
\and
\author[D,E]{\fnms{Jean} \snm{Feydy}\ead[label=e3,mark]{jean.feydy@inria.fr}}
\address[A]{Department of Computer Science, University of Helsinki, FCAI,
\printead{e1}}

\address[B]{Department of Mathematics, Imperial College London, UK}

\address[C]{Institute for the Advanced Study of Human Biology (ASHBi), Kyoto University, Japan
\printead{e2}}

\address[D]{HeKA team, Inria Paris, F-75012 Paris, France,
\printead{e3}}

\address[E]{Inserm, Centre de Recherche des Cordeliers, Sorbonne Université, Université Paris Cité, F-75006 Paris, France}
\end{aug}

\begin{abstract}
Over the last decades, various ``non-linear'' MCMC methods have arisen. While appealing for their convergence speed and efficiency, their practical implementation and theoretical study remain challenging. In this paper, we introduce a non-linear generalization of the Metropolis-Hastings algorithm to a proposal that depends not only on the current state, but also on its \emph{law}. We propose to simulate this dynamics as the mean field limit of a system of interacting particles, that can in turn itself be understood as a generalisation of the Metropolis-Hastings algorithm to a population of particles. Under the double limit in number of iterations and number of particles we prove that this algorithm converges. Then, we propose an efficient GPU implementation and illustrate its performance on various examples. The method is particularly stable on multimodal examples and converges faster than the classical methods. 
\end{abstract}



\end{frontmatter}

\section{Introduction}

\subsection{Background}

Monte Carlo methods are designed to estimate the integral of a function of interest $\varphi$ (sometimes called \emph{observable}) under a probability measure $\pi$.
The integral $\int \varphi(x)\pi(\dd x)$ is approximated by a random estimator of the form
\[\frac{1}{n}\sum_{i=1}^n \varphi(X_i)~, \]
where the $X_i$ are $n$ independent and identically distributed (i.i.d.) random variables with law~$\pi$. The law of large numbers ensures the convergence as $n\to+\infty$ of this estimator. 
For complex cases, when i.i.d. $\pi$-distributed random variables cannot be generated exactly, a now classical procedure consists in constructing a Markov chain $(X_i)_i$ with stationary distribution $\pi$. Ergodic theory results then ensure that the estimator above still converges, even though the $X_i$ are not independent from each other.
The Metropolis-Hastings algorithm \cite{metropolis1949monte,metropolis1953equation,hastings1970monte} provides a simple construction for such a Markov chain that only requires to evaluate the probability density function $\pi$ up to a multiplicative constant. The constructed chain is a random walk biased by an accept-reject step. Its convergence properties have been thoroughly studied, for example in \cite{mengersen1996rates}.

This well-known procedure has become a building block for more advanced samplers, that are designed to overcome the known flaws of the Metropolis-Hastings algorithm: slow convergence, bad mixing properties for multimodal distributions \emph{etc.} Such extensions include for instance the Wang and Landau algorithm \cite{bornn2013adaptive}, regional MCMC algorithms \cite{craiu2009learn}, or non Markovian (adaptive) versions \cite{haario2001adaptive,haario1999adaptive,atchade2005,atchade2011adaptive} where the next proposed state depends on the whole history of the process. The more recent PDMP samplers \cite{fearnhead2018piecewise,vanetti2017piecewise} provide an alternative to the discrete time accept-reject scheme, replacing it by a continuous time non reversible Markov process with random jumps. Finally, more complex algorithms are based on the evaluation of the gradient of $\pi$, see for instance the Metropolis-adjusted Langevin \cite{besag1994mala} and the Hamiltonian Monte Carlo algorithms \cite{1987hmc}. All these methods can be seen as ``linear" as the next position only depends on the current position. 

A non-Markovian alternative to Metropolis-Hastings like methods is given by Importance Sampling algorithms. By drawing i.i.d. samples from an auxiliary distribution $q$, which is usually simple and called the importance distribution, we can build an estimator using the following identity:
\[\int \varphi(x)\pi(x)\dd x = \int\frac{\pi(x)}{q(x)} \varphi(x)q(x)\dd x \simeq \frac{1}{n}\sum_{i=1}^n w_i \varphi(X_i),\]
where the $X_i$ are i.i.d. with law $q$ and the $w_i\propto \pi(X_i)/q(X_i)$, are called the \emph{importance weights}. 
The choice of $q$ is critical, as bad choices can lead to a degeneracy of the importance weights. Iterative methods have been developed to sequentially update the choice of the importance distribution, and to update the $X_i$ now interpreted as particles that evolve along iterations. Among these algorithms, we can cite the Sequential Importance Sampling algorithm \cite{gordon1993novel}, the Population Monte Carlo (PMC) methods \cite{douc2007,cappe2004population,cappe2008adaptive} or the recent Safe Adaptive Importance Sampling (SAIS) algorithm \cite{delyon2019adaptive}. This paradigm is in particular well-suited to the study of filtering problems \cite{gordon1993novel}, leading to the development of Sequential Monte Carlo (SMC) methods \cite{del2006sequential,doucet2013sequential}. A review of population-based algorithms and of the SMC method can be found in \cite{jasra2007population}.

The SMC methodology has recently been used to design and study \emph{nonlinear} MCMC algorithms \cite{andrieu2011nonlinear}. This framework can be seen as a generalisation of some non-Markovian extensions of the Metropolis-Hastings algorithm (such as the ``resampling from the past'' procedure \cite{haario2001adaptive,atchade2005}) but also allows the use of a wider range of algorithmic techniques. Examples are given in \cite{andrieu2011nonlinear,andrieu2007nonlinear} and are often based on the simulation of auxiliary chains. In the present article, we show that an alternative procedure based on the simulation of a swarm of \emph{interacting} particles can also be used to approximate a nonlinear Markov chain.
This provides a multi-particle generalisation of the Metropolis-Hastings procedure. 

The duality between particle systems and non-linear Markov processes has first been underlined in statistical physics; on the mathematical side, it has been the subject of the pioneering works of \cite{mckean1966class,mckean1967propagation,kac1956foundations,dobrushin1979vlasov}. A key result is the \emph{propagation of chaos} property formalised by \cite{kac1956foundations}, which implies that under an initial \emph{chaoticity} assumption on the law of the particles, the empirical measure of the system at any further time converges towards a deterministic limit; this type of limit is called \emph{mean-field limit}. In a continuous time framework, this limit classically satisfies a nonlinear Partial Differential Equation (PDE)  \cite{sznitman1991topics,meleard1996asymptotic}. The original diffusion framework has been extended to jump and jump-diffusion processes in \cite{graham1992mckean,graham1992nonlinear}. We refer the interested reader to \cite{bellomo2017active,bellomo2019active} for recent reviews and surveys of the applications of such models. We also mention that this methodology has been used in the analysis of particle methods in filtering problems \cite{Del_Moral_2013,andrieu2010particle}.

\subsection{Objective and methods}

Let $\pi$ be a target measure on $E \subset \R^d$, known up to a multiplicative constant and which is assumed to have a density with respect to the Lebesgue measure. We denote $\pb(E)$ the set of probability measures on $E$.
The goal of the present article is to build a \emph{nonlinear} Markov chain $(\overline{X}_t)_t$ on $E$ that samples $\pi$ efficiently.
Given a sample $\overline{X}_t$ at iteration $t$,
we draw $\overline{X}_{t+1}$ according to
\[\overline{X}_{t+1}\sim K_{\mu_t}(\overline{X}_t,\dd y)~,\]
where the \emph{transition kernel} is defined by:
\begin{equation}\label{eq:cmckernel}
K_{\mu_t}(x,\dd y) := 
\underbrace{\vphantom{\int_{z\in E}}h(\alpha_{\mu_t}(x,y))\Theta_{\mu_t}(\dd y|x)}_{\text{accept}}
+
\underbrace{\Big[1-\int_{z\in E}h(\alpha_{\mu_t}(x,z))\Theta_{\mu_t}(\dd z|x)\Big]\delta_x(\dd y)}_{\text{reject}}
\end{equation}
and where for $t\in\N$, $\mu_t\in\pb(E)$ is the law of $\overline{X}_t$. 
In the discrete setting,
this method is implemented by Algorithm~\ref{algo:cmc}, detailed below. 
It relies on the following quantities:
\begin{itemize}
    \item The \textbf{proposal distribution}, a map 
    \[\Theta : E\times\pb(E)\longrightarrow\pdf_0(E),\]
    where $\pdf_0(E)$ the subset of non-vanishing absolutely continuous probability measures. For $x\in E$ and $\mu\in~\pb(E)$, its associated proposal probability density function is denoted by:
    \[\Theta(x,\mu)(y) \dd y\equiv \Theta_\mu(y|x)\dd y.\]
    Intuitively, the probability distribution
    $\Theta(x,\mu)$ can be understood as an approximation of the target $\pi$ 
    that our method uses to \emph{propose} new samples
    $y$ in a neighborhood of a point $x$,
    relying on the information that is provided
    by a probability measure~$\mu$. In the following $\mu$ will be the empirical distribution associated with a system of particles. 
    To ensure a fast convergence of our method, both in terms of computation time and number of iterations,
    the distribution $\Theta(x,\mu)$ should be both \textbf{easy to sample} and
    close to the target distribution $\pi$.
    In practice, the choice of a good proposal $\Theta$
    depends on the assumptions that can be made on the distribution~$\pi$.
    We will present several examples in Section \ref{sec:collectiveproposals}. A simple example to keep in mind and which will be detailed in Section \ref{sec:vanillacmc} is $\Theta_\mu(y|x) = K\star \mu(y)$ where $K$ is a random-walk kernel. 
    
    \item For $\mu\in\pb(E)$ and $x,y\in E$, the \textbf{acceptance ratio} is defined by: 
    \[\alpha_\mu(x,y) := \frac{\Theta_\mu(x|y)\pi(y)}{\Theta_\mu(y|x)\pi(x)}.\]
    This quantity expresses the relative appeals
    of the transition $x\rightarrow y$ for the ``model'' density $\Theta_\mu(\dd y|x)$ and the ground truth target $\pi(\dd y)$, as in classical Metropolis-Hastings methods. Noticeably the denominator is never null because of the non-vanishing assumption on $\Theta$.
    Crucially, it can be computed even when the law of $\pi$ is only
    known up to a multiplicative constant
    and allows our method to account for mis-matches
    between the proposal $\Theta$ and
    the distribution to sample $\pi$.
    Note that in practice, for the sake of numerical stability,
    the acceptance ratio is often manipulated through its logarithm:
    \[
    \underbrace{\log \alpha_\mu(x,y)}_{\text{``correction''}} := 
    \underbrace{\big[\log \pi(y) - \log\pi(x)\big]}_{\text{appeal of $x\rightarrow y$ for $\pi$}}
    - 
    \underbrace{\big[ \log \Theta_\mu(y|x) -  \log \Theta_\mu(x|y)}_{\text{appeal of $x\rightarrow y$ for $\Theta_\mu$}}
    \big] .\]
    In the following, we will assume that $\pi$ is bounded away from zero so that the acceptance ratio is always well-defined. 
    
    \item The \textbf{acceptance function} is a non-decreasing Lipschitz map of the form $h:[0,+\infty)\to~[0,1]$ which satisfies
    \begin{equation}\label{eq:acceptancerelation}\forall u\in[0,\infty),\quad uh(1/u)=h(u).\end{equation}
    A classical example is $h(u)=\min(1,u)$. Other examples in the literature can be found in \cite{agrawal2021optimal} and the references therein. In the following, we will occasionally use (for technical reasons) the ``lazy Metropolis-Hastings'' acceptance function $h(u)=\eta \min(1,u)$, $\eta<1$ introduced in \cite{latuszynski2013clts}. 
    
    As detailed in Algorithm~\ref{algo:cmc},
    we combine the acceptance ratio $\alpha_\mu(x,y)$ 
    and the acceptance function $h$ to \textbf{reject} 
    proposed samples $y$ that are much more appealing for $\Theta_\mu(\cdot|x)$
    than they are for $\pi$.
    
    This \textbf{necessary correction} ensures that our method samples
    the target $\pi$ instead of the simpler
    proposal distribution.
    On the other hand, it can also slow down the method
    if the proposed samples $y$ keep being rejected.
    Efficient proposal distributions should keep the acceptance ratio
    high enough to ensure a renewal of the population of samples
    $X^i_t$ at every iteration and thus provide good mixing properties.
    
\end{itemize}

\paragraph{Non-linearity.}
We say that the transition kernel is nonlinear due to its dependency on the law of the chain that it generates. When the proposal distribution does not depend on~$\mu_t$, the kernel is \emph{linear} and we obtain the general form of the classical Metropolis-Hastings kernel.

\paragraph{Interest of the method}
The main interest of the method appears for complex distributions, that is multimodal distributions. From an MCMC point of view, the use of a non-linear sampler removes the mixing problem which is one of the main drawbacks of the Metropolis-Hastings algorithm. For nonlinear samplers, exponential convergence is ensured as soon as the distribution carries weight in every mode, without having to explore and exit each of the modes. On the contrary, the Metropolis-Hastings algorithms may remain stuck in a single mode even with long runs. Nonlinear methods allow to learn efficiently the relative weight of the distribution modes, which is unavailable even for several independent runs of Metropolis-Hastings algorithm with different initialisations, that is the most direct parallel version of Metropolis-Hastings. The numerical experiments will confirm that nonlinear methods perform better for both exploration and convergence.

\paragraph{Contributions.}
We follow \cite{andrieu2011nonlinear} and split our analysis into two steps:

\begin{enumerate}
    \item We show that our non-linear kernel admits $\pi$ as a stationary distribution and study its asymptotic properties. 
    \item We present a practical implementation based on the simulation of a system of interacting particles that enables the simulation of this kernel for different choices of the proposal distribution.
\end{enumerate}

\begin{algorithm}[t]
    \SetAlgoLined
    \KwIn{An initial population of particles $(X^1_0,\ldots,X^N_0)\in E^N$, \\
    a maximum time $T\in\N$, a proposal distribution $\Theta$ \\
    and an acceptance function $h$}\vspace{.1cm}
    \KwOut{A sample $\big(X^i_t\big)_{1 \leq i \leq N ; \ 1 \leq t \leq T}$}\vspace{.1cm}
    \For {$t=0$ \textup{\textbf{to}} $T-1$}{\vspace{.05cm}
    \For{$i=1$ \textup{\textbf{to}} $N$}{\vspace{.05cm}
        \textbf{(Proposal)} Draw $Y^i_t\sim \Theta_{\hat{\mu}_t^N}(\,\cdot\,|X^i_t)$ a proposal for the new state of particle $i$\;
        
        \textbf{(Acceptation)} Compute $\alpha_{\hat{\mu}^N_t}(X^i_t,Y^i_t) = \frac{\Theta_{\hat{\mu}^N_t}(X^i_t|Y^i_t)\pi(Y^i_t)}{\Theta_{\hat{\mu}^N_t}(Y^i_t|X^i_t)\pi(X^i_t)}$\;
        
        Draw $U^i_t\sim\mathcal{U}([0,1])$\; \vspace{.2cm}
        
        \eIf{$U^i_t\leq h\big(\alpha_{\hat{\mu}^N_t}(X^i_t,Y^i_t)\big)$}{\vspace{.05cm}
            Set $X^i_{t+1}=Y^i_t$ \tcp*[r]{Accept, probability $h(\alpha_{\hat{\mu}^N_t}(X^i_t,Y^i_t))$.}
            }{
            Set $X^i_{t+1}=X^i_t$  \tcp*[r]{Reject, likely if $\alpha_{\hat{\mu}^N_t}(X^i_t,Y^i_t) \simeq 0 $.}
        }
    }}
    \vspace*{.1cm}
\caption{Collective Monte Carlo (CMC)}
\label{algo:cmc}
\end{algorithm}

\paragraph{Analytical study.}
Starting from an initial distribution $\mu_0\in\pb(E)$, the law $\mu_t$ of the nonlinear chain at the $t$-th iteration satisfies
    \[\mu_{t+1}=\mathcal{T}[\mu_t]\]
    where $\mathcal{T}:\pb(E)\to\pb(E)$ is the \textbf{transition operator} defined by duality in the space of measures by:
    \begin{align}\label{eq:transitionoperator}
    \langle \varphi, \mathcal{T}[\mu]\rangle &:= \int_{E} \varphi(x)\mathcal{T}[\mu](\dd x) = \iint_{E\times E} \varphi(y)K_\mu(x,\dd y)\mu(\dd x),
    \end{align}
    for any continuous bounded test function $\varphi\in C_b(E)$. 
    Thanks to the \textbf{detailed balance condition} (also called \emph{micro-reversibility} in the context of statistical mechanics \citep{tolman1979principles}), for all $x,y\in E$ and $\mu\in\pb(E)$: 
    \begin{equation}\label{eq:reversibility}\pi(x)\Theta_\mu(y|x)h(\alpha_\mu(x,y))=\pi(y)\Theta_\mu(x|y)h(\alpha_\mu(y,x)),\end{equation}
    the transition operator can be rewritten:
    \[\mathcal{T}[\mu](\dd x)= \mu(\dd x) + \int_{E} \pi(x)W_\mu(x\to y){\left(\frac{\mu(\dd y)}{\pi(y)}\dd x-\frac{\mu(\dd x)}{\pi(x)}\dd y\right)},\]
    with $W_\mu(x\to y) := \Theta_{\mu}(y|x)h(\alpha_{\mu}(x,y))$, from which it can be easily seen that $\mathcal{T}[\pi]=\pi$.

    We are going to develop an analytical framework in which the convergence of the sequence of iterations of the transition operator can be analysed. Using entropy methods, we prove the exponential convergence towards $\pi$ for a large class of proposal distributions. We show that in an asymptotic regime to be detailed, the rate of convergence depends only on how close from the target is the initial condition. As a byproduct, in the linear Metropolis-Hastings case, we obtain a convergence result similar to the one obtained in~\cite{diaconis2011geometric}.

\paragraph{Efficient implementation.}
   It is not possible in general to sample directly $\overline{X}_t$ from a nonlinear kernel because the law $\mu_t$ is not available. We therefore rely on a \textbf{mean-field} particle method to approximate such samples. Starting from a swarm of $N$ particles $X^1_t,\ldots,X^N_t\in E$ at the iteration $t$, we construct the next iteration by sampling independently for $i\in\{1,\ldots,N\}$:
    \[X^i_{t+1}\sim K_{\hat{\mu}^N_t}(X^i_t,\dd y),\]
    where 
    \[\hat{\mu}^N_t:=\frac{1}{N}\sum_{i=1}^N \delta_{X^i_t}\in\pb(E),\]
    is the (random) empirical measure of the system of particles which is used as a proxy of the distribution $\mu_t$. 
    We show that as $N$ goes to infinity and for each $t\in\N$, $\hat{\mu}^N_t$ converges towards a deterministic limit which is the $t$-th iterate of the nonlinear operator $\mathcal{T}$ starting from $\mu_0$. Moreover we show that the $N$ particles are asymptotically, in $N$, independent thus forming an approximation of a system of~$N$ independent nonlinear Markov chains with transition kernel \eqref{eq:cmckernel}. 

    A drawback of this approach is its high computational cost, that may scale in $\mathcal{O}(N^2)$ or $\mathcal{O}(N^3)$ for some choices of the proposal $\Theta$. To overcome this difficulty, we propose an implementation based on GPU, more precisely on the techniques developed in the KeOps library \cite{keops} by the third author. 

\paragraph{Outline.}
Section \ref{sec:framework} is devoted to the convergence analysis of Algorithm \ref{algo:cmc} for a general class of proposal distributions. The mean-field limit and the long-time asymptotic properties are studied respectively in Section \ref{sec:meanfield} and Section \ref{sec:convergence}. Several variants of the main algorithm are presented in Section \ref{sec:collectiveproposals}. The GPU implementation of the different algorithms is detailed in Section \ref{sec:gpu}. Applications to various problems are presented in Section \ref{sec:experiments}. In the appendix, we present the complementary proofs (Appendix \ref{sec:proofpoc}) as well as variations on the results of the paper (Appendix \ref{appendix:continuoustime}). We also add some complementary remarks on the links between our method and other classical methods (Appendix \ref{sec:related}) and other and additional numerical results (Appendix \ref{appendix:CMCIS} and Appendix \ref{appendix:Cauchy}).

Throughout this article, we assume that $\pi$ satisfies the following assumption. 

\begin{assumption}\label{assum:pi} The support of $\pi$, denoted by $E$, is a compact subset of $\R^d$. The target distribution $\pi$ is Lipschitz continuous and $\pi$ does not vanish on $E$ : 
\[m_0:=\inf_E \pi >0\quad\text{and}\quad M_0:=\sup_E \pi<+\infty.\]  
\end{assumption}

\begin{notation} The following notations will be used throughout the article. 
\begin{itemize}
    \item $\pb(E)$ denotes the set of probability measures on $E$.
    \item $\pdf(E)$ denotes the set of probability measures on $E$ which are absolutely continuous with respect to the Lebesgue measure on $\R^d$. A probability measure in $\pdf(E)$ is identified with its associated probability density function: when $f\in\pdf(E)$ we write indifferently,
    \[f(\dd x)\equiv f(x)\dd x.\]
    \item $\pdf_0(E)\subset\pdf(E)$ denotes the subset of continuous probability density functions which do not vanish on $E$ (recall that $E$ is compact).
    \item A test function $\varphi\in C_b(E)$ is a continuous (bounded) function on $E$. For $\mu\in\pb(E)$ we write indifferently 
    \[\langle \varphi,\mu\rangle\equiv\int_E \varphi(x)\mu(\dd x).\]
    \item $X\sim \mu$ means that the law of the random variable $X\in E$ is $\mu\in\pb(E)$.
    \item $W^1(\mu,\nu)$ denotes the Wasserstein-1 distance between the two probability measures $\mu,\nu\in\pb(E)$, defined by $W^1(\mu,\nu):= \inf_{X\sim \mu, Y\sim\nu} \E|X-Y|$ (see \cite[Chapter 6]{villani_optimal_2009} for additional details on Wasserstein distances). 
\end{itemize}
\end{notation}

\section{Convergence analysis}\label{sec:framework}

Algorithm \ref{algo:cmc} gives a trajectorial approximation of the nonlinear Markov chain $(\overline{X}_t)_t$ with law $(\mu_t)_t$ defined by the transition kernel \eqref{eq:cmckernel}. In this section, we prove the convergence of this algorithm under general assumptions on the proposal distribution $\Theta$ described below. The proof of our main result is split into two steps each summarised in a theorem, first the mean-field limit when $N\to+\infty$ (Section \ref{sec:meanfield}) and then the long-time convergence towards $\pi$ (Section \ref{sec:convergence}). 

\subsection{Assumptions}\label{sec:assumptions}

For our theoretical results, we will need the following assumptions. The first three following assumptions are needed to prove the many-particle limit in Section \ref{sec:meanfield}.

\begin{assumption}[Boundedness]\label{assum:bound}
There exist two constants $\kappa^-,\kappa^+>0$ such that for all $\mu\in\mathcal{P}(E)$ and for all $x,y\in E$ :
\[\kappa_-\leq \Theta_\mu(y|x)\leq \kappa_+.\]
\end{assumption}

\begin{assumption}[$L^\infty$ Lipschitz]\label{assum:lipschitz}
The map $\Theta:E\times\pb(E)\to\pdf_0(E)$ is globally Lipschitz for the $L^\infty$-norm on $E$: there exists a constant $L>0$ such that for all $x,y,x',y'\in E$ and for all $(\mu,\nu)\in\mathcal{P}(E)^2$ :
\[\big|\Theta_\mu(y|x)-\Theta_\nu(y'|x')\big|\leq L\Big(W^1(\mu,\nu)+|x-x'|+|y-y'|\Big).\]
\end{assumption}

\begin{assumption}[$W^1$ non-expansive]\label{assum:nonexpansive}
The map $\Theta:E\times\pb(E)\to\pdf_0(E)$ is non-expansive for the Wasserstein-1 distance: for all $x,x'\in E$ and for all $(\mu,\nu)\in\mathcal{P}(E)^2$,
\[W^1\Big(\Theta_\mu(\dd y|x),\Theta_\nu(\dd y|x')\Big)\leq W^1(\mu,\nu)+|x-x'|.\]
\end{assumption}

All the proposal distributions presented in Section \ref{sec:collectiveproposals} are based on a convolution product with one or many fixed kernels. The smoothness and boundedness properties of the proposal distribution (Assumptions \ref{assum:bound}, \ref{assum:lipschitz}, \ref{assum:nonexpansive}) are thus inherited from the properties of these kernels. Moreover, we note that these assumptions on the proposal distribution and the compactness Assumption \ref{assum:pi} on the target distribution could be replaced by assumptions on the ratio between the two distributions as explained at the end of Appendix \ref{sec:proofpoc} (see in particular Remark \ref{rem:aboutassumptions}).

The next two assumptions are needed to prove the long-time convergence property in Section \ref{sec:convergence}.

\begin{assumption}\label{assum:hbounded}
There exists $\eta\in (0,1)$ such that 
\[\forall u\in [0,+\infty),\quad h(u)\leq \eta.\]
\end{assumption}

\begin{rem}
Assumption \ref{assum:hbounded} is satisfied for instance for $h(u)=\eta\min(1,u)$ (which is referred as the ``lazy Metropolis-Hastings'' acceptance function in \cite{latuszynski2013clts}). We make this assumption mostly for technical reasons in order to obtain in an easy manner an explicit convergence rate in Theorem \ref{thm:convergencenonlineardiscrete}. However, using compactness arguments, we can prove that the convergence of $(\mu_t)_t$ towards $\pi$ still holds without this assumption (see Corollary \ref{coro:continuousimpliesdiscrete}).
\end{rem}

The next assumption ensures that the proposal distribution is not too ``far'' from $\pi$.

\begin{assumption}[Monotonicity]\label{assum:monotonicity}
The proposal distribution $\Theta$ satisfies the following monotonicity property: there exists a non decreasing function
\[c^-:[0,1]\to(0,1],\]
such that for all $\mu\in\pdf_0(E)$, 
\[\inf_{(x,y)\in E^2} \frac{\Theta_\mu(y|x)}{\pi(y)}\geq c^-\left(\inf_{x\in E} \frac{\mu(x)}{\pi(x)}\right).\]
\end{assumption}

\begin{rem}
Note that under Assumptions \ref{assum:pi} and \ref{assum:bound}, Assumption \ref{assum:monotonicity} is always satisfied with a constant function $c^- \equiv \kappa_-/M_0$. Sharper results can be obtained in specific cases. Moreover, note that Assumptions \ref{assum:bound}, \ref{assum:lipschitz} and \ref{assum:nonexpansive} are not necessary to prove Theorem \ref{thm:convergencenonlineardiscrete}. 
\end{rem}

\begin{rem}\label{rem:convolutioncminus}
The monotonicity Assumption \ref{assum:monotonicity} is satisfied for all the ``convolution based'' methods such as Algorithms~\ref{algo:coker} and \ref{algo:marmok}, which will be introduced later, since for $m>0$, it holds that:  
\[[\forall y\in E,\quad m\pi(y)\leq \mu(y)]\quad \Longrightarrow\quad [\forall y\in E,\quad m K\star\pi(y)\leq K\star \mu(y)],\]
and therefore, if the left-hand side condition holds, dividing by $\pi(y)$ yields:
\[\forall y\in E,\quad\frac{K\star \mu(y)}{\pi(y)}\geq m \frac{K\star\pi(y)}{\pi(y)}.\]
On the right-hand side of the last inequality the ratio $K\star\pi(y)/\pi(y)$ depends only on $\pi$ and is bounded from below, at least for small interaction kernels $K$, since $K\star\pi$ converges uniformly towards $\pi$ as $K\to\delta_0$. In the degenerate case $K=\delta_0$, we obtain $c^{-}(m)=m$ for all $m>0$ (see Remark \ref{rem:optimalrate}).
\end{rem}

\subsection{Main result}

The following theorem is our main convergence result.

\begin{theorem}\label{thm:convergencecmc}
Let $\hat{\mu}^N_t = \frac{1}{N}\sum_{i=1}^N \delta_{X^i_t}$ be the random empirical distribution of the particle system constructed at the $t$-th iteration of Algorithm \ref{algo:cmc} with an i.i.d. $\mu_0\in\pdf_0(E)$ distributed initial condition. Under Assumptions \ref{assum:bound}, \ref{assum:lipschitz}, \ref{assum:nonexpansive}, \ref{assum:hbounded} and \ref{assum:monotonicity}, there exist $C_1, C_2, C_3>0$ and $\lambda\in(0,1)$  which depend only on $\mu_0$, $\pi$ and $E$ such that for all $t\in\N$,
\[\E W^1(\hat{\mu}^N_t,\pi) \leq C_1\beta(N)\e^{tC_2} + C_3(1-\lambda)^{t/2},\]
where
\begin{equation}\label{eq:glivenkocantelli}
\beta(N):=\left\{\begin{array}{ll}
CN^{-1/2} & \text{if}\,\,\,d=1 \\
CN^{-1/2}\log(N) & \text{if}\,\,\,d=2\\
CN^{-1/d} & \text{if}\,\,\,d>2
\end{array}
\right.,
\end{equation}
and $C>0$ is a constant which depends only on $E$ and $\pi$. In particular $\beta(N)\to0$ as $N\to~+~\infty$.
\end{theorem}

\begin{rem}
The convergence speed in $t$ of the second term on the right-hand side is geometric. It corresponds to the convergence speed of the limit nonlinear Markov chain stated in Theorem \ref{thm:convergencetopi}. Note that it is also the convergence speed of the classical Metropolis-Hastings algorithm. This result does not account for practical mixing issues that will be described later.
\end{rem}

\begin{proof}
This result is deduced from Theorem \ref{thm:pathwiseestimate} and \ref{thm:convergencenonlineardiscrete}, as it is a direct consequence of the triangle inequality
\[\E W^1(\hat{\mu}^N_t,\pi) \leq \E W^1(\hat{\mu}^N_t,\mu_t) + \E W^1({\mu}_t,\pi),\]
and the convergence results \eqref{eq:EW1tdiscrete} and \eqref{eq:mut-piTV}. In order to bound the second term on the right-hand side, we recall that on the compact set $E$, the total variation norm controls the Wasserstein-1 distance \cite[Theorem 6.15]{villani_optimal_2009}.
\end{proof}

By the Kantorovich characterisation of the Wasserstein distance \cite[Remark 6.5]{villani_optimal_2009}, this result ensures the convergence in expectation of any Lipschitz observable (and by density of any continuous observable) in the double limit $N\to+\infty$ and $t\to+\infty$ provided that $\beta(N)\e^{C_2 t}\to 0$. These results and in particular the link between $N$ and $t$ are mostly of theoretical nature and often suboptimal; in practice higher convergence rates may be obtained (see Section \ref{sec:experiments}).

Moreover, although Theorem \ref{thm:convergencecmc} states a geometric convergence result for the nonlinear samplers which is similar to classical convergence results for the classical (linear) Metropolis-Hastings algorithm, the behaviours of nonlinear and linear samplers can be very different in practise. Nonlinear samplers are typically much more efficient in the sense that they need less iterations to converge and although each iteration is more costly, the total computation time remains much shorter. In addition to the experiments shown in Section \ref{sec:experiments}, Remark \ref{rem:optimalrate} gives a theoretical result obtained as a consequence of Theorem \ref{thm:convergencecmc} which illustrates this difference.  

\subsection{Mean field approximation}\label{sec:meanfield}

In this section, we show that the system of particles defined by Algorithm~\ref{algo:cmc} satisfies the propagation of chaos property and that the limiting law at each iteration is the law of the nonlinear Markov chain with transition kernel satisfying \eqref{eq:cmckernel}. From now on, we assume that the proposal distribution $\Theta$ satisfies Assumptions \ref{assum:bound}, \ref{assum:lipschitz} and \ref{assum:nonexpansive} (see also the discussion in Remark~\ref{rem:aboutassumptions}).

The main result of this section is the following theorem. 

\begin{theorem}[Coupling bound]\label{thm:pathwiseestimate}
Let $\Theta$ be a proposal distribution which satisfies Assumptions \ref{assum:bound}, \ref{assum:lipschitz} and \ref{assum:nonexpansive}. Let $t\in\N$. There is a probability space on which are defined a system of $N$ i.i.d. nonlinear Markov chains $(\overline{X}^i_t)^{}_t$, $i\in\{1,\ldots,N\}$, defined by the transition kernel \eqref{eq:cmckernel} and a system of $N$ particles $(X^i_t)^{}_t$, $i\in\{1,\ldots,N\}$, which is equal in law to the $N$-particle system constructed by Algorithm \ref{algo:cmc}, such that
\begin{equation}\label{eq:couplingbound}\forall i\in\{1,\ldots,N\},\quad\E|\overline{X}^i_t-X^i_t|\leq\beta(N)\e^{tC_\Theta},\end{equation}
where $C_\Theta>0$ is a constant which depends only on $\pi$ and $\Theta$ and where $\beta(N)$ is defined by \eqref{eq:glivenkocantelli} 
\end{theorem}

The proof of Theorem \ref{thm:pathwiseestimate} is based on coupling arguments inspired by \cite{sznitman1991topics} and adapted from \cite{diez2019propagation}. It can be found in Appendix \ref{sec:proofpoc}. This so-called coupling estimate classically implies the following properties (see \cite{sznitman1991topics}). 

\begin{corollary}[Mean-field limit and propagation of chaos]\label{thm:poc}
Let $\Theta$ be a proposal distribution which satisfies Assumptions \ref{assum:bound}, \ref{assum:lipschitz} and \ref{assum:nonexpansive}. Let $(X^i_0)_{i\in\{1,\ldots,N\}}$ be $N$ i.i.d. random variables with common law $\mu_0\in \pb(E)$ (\emph{chaoticity assumption}). Let $t\in\N$ and let $(X^i_t)_{i\in\{1,\ldots,N\}}$ be the $N$ particles constructed at the $t$-th iteration of Algorithm \ref{algo:cmc}. Let $\mu_t=\mathcal{T}^{(t)}[\mu_0]$ be the $t$-th iterate of the transition operator \eqref{eq:transitionoperator} starting from $\mu_0$, that is $\mu_t$ is the law of the nonlinear Markov chain defined by the transition kernel \eqref{eq:cmckernel} at iteration~$t$. Then the following properties hold true.
\begin{enumerate}
    \item The (random) empirical measure $\hat{\mu}^N_t = \frac{1}{N}\sum_{i=1}^N \delta_{X^i_t}$ satisfies
    \begin{equation}\label{eq:EW1tdiscrete}\E W^1(\hat{\mu}^N_t, \mu_t) \leq C_1\beta(N)\e^{tC_2},\end{equation}
    where $C_1,C_2>0$ are two absolute constants and $\beta(N)$ is defined by \eqref{eq:glivenkocantelli}.
    \item The sequence of random empirical measures $\hat{\mu}^N_t = \frac{1}{N}\sum_{i=1}^N \delta_{X^i_t}$, seen a sequence of $\pb(E)$-valued random variables, converges in law towards the deterministic limit :  
\[\hat{\mu}^N_t\underset{N\to+\infty}{\longrightarrow}\mu_t.\]
We recall that the space of probability measures is endowed with the topology of the weak convergence, i.e. convergence against bounded continuous test functions. The convergence in law of a sequence of random measures thus means the weak convergence of their laws in the space $\pb(\pb(E))$, i.e. since $\mu_t$ is deterministic, it holds that $\mathrm{Law}(\hat{\mu}^N_t)\in\pb(\pb(E)) \to \delta_{\mu_t} \in \pb(\pb(E))$ where $\delta_{\mu_t}$ is the Dirac mass at the point $\mu_t\in \pb(E)$ and the convergence is the weak convergence in $\pb(\pb(E))$. This is also equivalent to say that for all test function $\Phi\in C_b(\pb(E))$, $\E[\Phi(\hat{\mu}^N_t)]\to \Phi(\mu_t)$.
\item For every $\ell$-tuple of continuous bounded functions $\varphi_1,\ldots\varphi_\ell$ on $E$, it holds that: 
\begin{equation}\label{eq:kacpoc}\int_{E^\ell}\varphi_1(x_1)\ldots\varphi_\ell(x_\ell)\mu^{\ell,N}_t(\dd x_1,\ldots,\dd x_\ell)\underset{N\to+\infty}{\longrightarrow} \prod_{k=1}^\ell \langle \varphi_k, \mu_t\rangle,\end{equation}
where $\mu^{\ell,N}_t$ is the joint law at time $t$ of any subset of $\ell$ particles constructed by Algorithm \ref{algo:cmc} at iteration $t$. 
\end{enumerate}
\end{corollary}

\begin{proof}
The first property follows from the triangle inequality
\[\E W^1(\hat{\mu}^N_t,\mu_t) \leq \E W^1(\hat{\mu}^N_t,\overline{\mu}^N_t) + \E W^1(\overline{\mu}^N_t, \mu_t),\]
where $\overline{\mu}^N_t = \frac{1}{N}\sum_{i=1}^N \delta_{\overline{X}^i_t}$ is the empirical measure of the $N$ nonlinear Markov chains constructed in Theorem \ref{thm:pathwiseestimate}. The first term on the right-hand side is bounded by \eqref{eq:couplingbound} by definition of the Wasserstein distance and the exchangeability of the processes. The second term on the right-hand side is bounded by $\beta(N)$ by \cite[Theorem 1]{fournier2015rate}. The second property is a classical consequence of \eqref{eq:couplingbound}, see \cite[Section 1]{hauray_kacs_2014}. The third property is equivalent to the second property by \cite[Proposition 2.2]{sznitman1991topics}. 
\end{proof} 

The property \eqref{eq:kacpoc} corresponds to the original formulation of the propagation of chaos introduced by \cite{kac1956foundations}. From our perspective, it justifies the use of Algorithm~\ref{algo:cmc} and ensures that as the number of particles grows to infinity and despite the interactions between the particles, we asymptotically recover an i.i.d. sample. The final MCMC approximation of the expectation of an observable $\varphi\in C_b(E)$ is thus given at the $t$-th iteration by: 
\[\int_E \varphi(x)\pi(\dd x) \simeq \frac{1}{N}\sum_{i=1}^N \varphi(X^i_t).\]

\begin{rem}\label{rem:diffusionrwm} In \cite{jourdain_optimal_2015}, the authors prove the propagation of chaos towards a continuous-time (nonlinear) diffusion process for a classical random walk Metropolis-Hastings algorithm in the product space $E^N$ where the trajectory of each of the $N$ dimensions is interpreted as a particle, where the target distribution $\pi^{\otimes N}$ is tensorized and under a specific scaling limit in $N$ for the time and the size of the random-walk kernel. In this algorithm, each move is globally accepted or rejected for the $N$ particles whereas in Algorithm \ref{algo:cmc}, the acceptance step is individualized for each particle. One consequence is that, unlike the algorithm in \cite{jourdain_optimal_2015}, for a fixed number $N$ of particles, $\pi^{\otimes N}$ is in general \emph{not} a stationary distribution of the particle system defined by Algorithm \ref{algo:cmc}. 
\end{rem}

\subsection{Long-time Asymptotics}\label{sec:convergence}

In this section, we prove that $\pi$ is the unique stationary measure of the nonlinear Markov chain defined by the transition kernel \eqref{eq:cmckernel} and we give a quantitative long-time convergence result. 

\subsubsection{Main result}

Let $(\mu_t)_{t\in\N}$ be the sequence of laws of the nonlinear Markov chain defined by the transition kernel \eqref{eq:cmckernel}. It satisfies the recurrence relation 
\begin{equation}\label{eq:recmu}\mu_{t+1} = \mathcal{T}[\mu_t],\end{equation}
where we recall that given $\mu\in\pb(E)$, the transition operator $\mathcal{T}$ is defined by: 
\[\mathcal{T}[\mu](\dd x)= \mu(\dd x) + \int_{E} \pi(x)W_\mu(x\to y){\left(\frac{\mu(\dd y)}{\pi(y)}\dd x-\frac{\mu(\dd x)}{\pi(x)}\dd y\right)},\]
and
\[W_\mu(x\to y) := \Theta_{\mu}(y|x)h(\alpha_{\mu}(x,y)).\]
Note that if the initial condition has a density with respect to the Lebesgue measure, then $\mu_t$ has also a density with respect to the Lebesgue measure. In the following we make this assumption and we write with a slight abuse of notations $\mu_t(x)\dd x \equiv \mu_t(\dd x)$ for this density. The following elementary lemma shows that the ratio $\mu_t/\pi$ is controlled by the initial condition.  

\begin{lemma}\label{lemma:infsupdiscrete} For any $t\in\N$, let $\mu_t\in\pdf_0(E)$ be given by the recurrence relation \eqref{eq:recmu} with initial condition $\mu_0\in\pdf_0(E)$. Then 
\[\inf_{x\in E} \frac{\mu_t(x)}{\pi(x)}\geq \inf_{x\in E} \frac{\mu_0(x)}{\pi(x)},\quad\sup_{x\in E} \frac{\mu_t(x)}{\pi(x)}\leq \sup_{x\in E} \frac{\mu_0(x)}{\pi(x)}.\]
\end{lemma}

\begin{proof}
Since $W_\mu(x\to y)\geq 0$ and $\int_E W_\mu(x\to y)\dd y \leq 1$, this comes directly from the relation 
\[\frac{\mathcal{T}[\mu](x)}{\pi(x)} = {\left(1-\int_E W_\mu(x\to y)\dd y\right)}\frac{\mu(x)}{\pi(x)}+\int_E W_\mu(x\to y)\frac{\mu(y)}{\pi(y)}\dd y,\]
for all $\mu\in\pdf_0(E)$ and $x\in E$.
\end{proof}

The main result of this section is the following convergence result which is a direct consequence of the results presented in the following sections and discussed below.  

\begin{theorem}\label{thm:convergencenonlineardiscrete}
Let $\Theta\in\pdf_0(E)$ be a proposal distribution which satisfies Assumption \ref{assum:monotonicity} and $h$ be an acceptance function which satisfies Assumption \ref{assum:hbounded}. Then there exist two (explicit) constants $C>0$ and $\lambda\in (0,1)$ which depend only on $\mu_0$, $\pi$ and $E$ such that 
\begin{equation}\label{eq:mut-piTV}
\|\mu_t-\pi\|_{\mathrm{TV}}\leq C(1-\lambda)^{t/2}
\end{equation}
\end{theorem}

\begin{proof}
Using the fact that the Total Variation norm is equal to the $L^1$ norm of the probability density functions and the Cauchy-Schwarz inequality, it holds that 
\begin{align*}\|\mu_t-\pi\|_{\mathrm{TV}} &=\int_E |\mu_t(x)-\pi(x)|\dd x = \int_E \sqrt{\pi(x)}\sqrt{\pi(x)}\left|\frac{\mu_t(x)}{\pi(x)}-1\right|\dd x \\
&\leq \sqrt{2\mathcal{H}[\mu_t|\pi]},\end{align*}
where $\mathcal{H}[\mu_t|\pi] := \int_E \pi(x)|\mu_t(x)/\pi(x) - 1|^2\dd x$ is the relative entropy which will be introduced later, see \eqref{eq:entropydissipationdef} with the function $\phi:s\mapsto \frac{1}{2}(s-1)^2$. The conclusion follows from Proposition \ref{prop:entropyexpdecaydiscrete} which proves that the relative entropy decays geometrically in $t$. 
\end{proof}

\begin{rem}
The last inequality between the TV norm and the square root of the relative entropy is a simple form of a Csisz\'{a}r-Kullback-Pinsker inequality, see \cite[Appendix~A]{jungel2016entropy} and \cite{bolley2005weighted}.
\end{rem}

\begin{rem}[Convergence rate of nonlinear samplers]\label{rem:optimalrate} In this proof, the convergence rate $\lambda$ is obtained by a crude estimate of the infimum of the jump rate $W_{f_t}(x\to y)$. We do not claim that this rate is optimal. In the degenerate case $\Theta_\mu(y|x)=\mu(y)$, the best rate obtained is equal to $h(1)$ by taking an initial condition arbitrarily close to $\pi$ (see Remark \ref{rem:convolutioncminus}). In Remark~\ref{rem:moderateinteraction}, it is shown that these proposal distributions can be obtained as the limit when $\sigma\to0$ of the proposal distributions $\Theta_\mu(y|x)=K_\sigma\star\mu(y)$ where $K_\sigma$ is a random-walk kernel with variance $\sigma^2$. These proposal distributions are the simplest nonlinear analog of the random-walk Metropolis-Hastings algorithm with kernel $K_\sigma$ (the proposal distribution in this case is $K_\sigma\star \delta_x (y)$). However in the nonlinear case, the convergence rate tends to a constant nonnegative value when $\sigma\to0$ while in the random-walk Metropolis-Hastings case the convergence rate tends to 0 (see Theorem \ref{thm:convergencemetropolis}). 
\end{rem}

In the linear case, that is when $\mathcal{T}$ is the linear transition operator of the classical random walk Metropolis-Hastings algorithm, the convergence of the sequence of iterates of $\mathcal{T}$ is studied in particular in \cite{mengersen1996rates} and more recently in \cite{diaconis2011geometric} using analytical spectral methods. It is not possible to follow this strategy in the nonlinear case. In order to motivate our strategy, let us notice that the recurrence relation \eqref{eq:recmu} can be interpreted as the explicit Euler discretization scheme 
\begin{equation}\label{eq:expliciteuler}\frac{\mu_{t+1}-\mu_t}{\Delta t} = \mathcal{T}[\mu_t]-\mu_t,\end{equation}
of the nonlinear Partial Differential Equation 
\begin{equation}\label{eq:nonlinearpde}
    \partial_t f_t = \mathcal{T}[f_t]-f_t
\end{equation}
with a constant time-step $\Delta t=1$. The PDE \eqref{eq:nonlinearpde} has a remarkable entropic structure which is detailed in Appendix \ref{appendix:continuoustime} and which allows to prove that $f_t\to\pi$ as $t\to+\infty$. Entropy methods are by now a classical tool to study the long-time properties of both linear and nonlinear PDEs, see \cite{jungel2016entropy, schmeiser2018entropy}. The proof of Theorem \ref{thm:convergencenonlineardiscrete} is based on the adaptation of these ideas to the present discrete-time setting. A more detailed discussion of the links between \eqref{eq:expliciteuler} and \eqref{eq:nonlinearpde} can be found in Appendix \ref{appendix:linksdiscretecontinuous}. Entropy methods have been used previously in a similar context in \cite{jourdain_optimal_2014} to prove the long-time convergence of a process obtained as the scaling limit of a particle-based Metropolis-Hastings algorithm~\cite{jourdain_optimal_2015} which, unlike the present case, is a continuous-time (nonlinear) diffusion process (see Remark~\ref{rem:diffusionrwm}). 

\subsubsection{Entropy and dissipation}

For a given convex function $\phi:[0,+\infty)\to[0,+\infty)$ such that $\phi(1)=0$, the relative entropy $\mathcal{H}[\mu|\pi]$ and dissipation $\mathcal{D}[\mu|\pi]$ of a probability density $\mu\in\pdf(E)$ with respect to $\pi$ are defined respectively by
\begin{equation}\label{eq:entropydissipationdef}\mathcal{H}[\mu|\pi] := \int_E \pi(x) \phi{\left(\frac{\mu(x)}{\pi(x)}\right)} \dd x,\quad \mathcal{D}[\mu|\pi] := -\int_E \phi'{\left(\frac{\mu(x)}{\pi(x)}\right)}\big(\mathcal{T}[\mu](x)-\mu(x)\big)\dd x.\end{equation}
Using the detailed balance property 
\[\pi(x)W_\mu(x\to y) = \pi(y)W_\mu(y\to x),\]
for all $x,y\in E$ and $\mu\in \pdf(E)$, it holds that 
\begin{multline*}\mathcal{D}[\mu|\pi] =\\ \frac{1}{2}\iint_{E\times E} \pi(x)W_\mu(x\to y){\left(\frac{\mu(y)}{\pi(y)}-\frac{\mu(x)}{\pi(x)}\right)}{\left(\phi'{\left(\frac{\mu(y)}{\pi(y)}\right)}-\phi'{\left(\frac{\mu(x)}{\pi(x)}\right)}\right)}\dd x\dd y,\end{multline*}
and thus $\mathcal{D}[\mu|\pi]\geq 0$ by convexity of $\phi$.

In the following, we focus on the case $\phi(s) = \frac{1}{2}(s-1)^2$ for which we can prove the following crucial lemma. Note also that in this case the relative entropy is equal to a weighted $L^2$ norm and thus dominates the Total Variation norm between probability density functions. 

\begin{lemma}\label{lemma:discreteentropynonincreasing}
Let $\phi(s) = \frac{1}{2}(s-1)^2$, then the sequence $(\mathcal{H}[\mu_t|\pi])_{t\in\N}$ is non-increasing. 
\end{lemma}

\begin{proof}
In this case, 

\begin{equation}\label{eq:entropysquare}\mathcal{H}[\mu|\pi] = \frac{1}{2}\int_E \pi(x){\left|\frac{\mu(x)}{\pi(x)}-1\right|^2}\dd x = \frac{1}{4}\iint_{E\times E} \pi(x)\pi(y){\left|\frac{\mu(x)}{\pi(x)}-\frac{\mu(y)}{\pi(y)}\right|^2}\dd x\dd y,\end{equation}
and 
\begin{equation}\label{eq:dissipationphipol2}\mathcal{D}[\mu|\pi] = \frac{1}{2}\iint_{E\times E} \pi(x)W_\mu(x\to y){\left|\frac{\mu(y)}{\pi(y)}-\frac{\mu(x)}{\pi(x)}\right|^2}\dd x\dd y.\end{equation}

Moreover, since $\phi(s)=\frac{1}{2}(s-1)^2$ is a polynomial of order 2, the exact Taylor expansion 
\[\phi(v) = \phi(u) + \phi'(u)(v-u) + \frac{1}{2}\phi''(u)(v-u)^2,\]
yields 
\begin{equation}\label{eq:entropydissipationtaylor}\mathcal{H}[\mu_{t+1}|\pi]-\mathcal{H}[\mu_t|\pi] = -\mathcal{D}[\mu_t|\pi] + \frac{1}{2}\int_E \pi(x){\left|\frac{\mu_{t+1}(x)}{\pi(x)}-\frac{\mu_t(x)}{\pi(x)}\right|^2}\dd x.\end{equation}
Using that $\mu_{t+1}=\mathcal{T}[\mu_t]$ and the definition of $\mathcal{T}$, it holds that 
\begin{multline*}\int_E \pi(x){\left|\frac{\mu_{t+1}(x)}{\pi(x)}-\frac{\mu_t(x)}{\pi(x)}\right|^2}\dd x \\ = \int_E \pi(x) {\left|\int_E W_{\mu_t}(x\to y){\left(\frac{\mu_t(y)}{\pi(y)}-\frac{\mu_t(x)}{\pi(x)}\right)}\dd y\right|^2}\dd x.\end{multline*}
By the Cauchy-Schwarz inequality, we get
\begin{multline*}\int_E \pi(x){\left|\frac{\mu_{t+1}(x)}{\pi(x)}-\frac{\mu_t(x)}{\pi(x)}\right|^2}\dd x\\\leq \iint_{E\times E} \pi(x)W_{\mu_t}(x\to y) {\left|\frac{\mu_{t}(y)}{\pi(y)}-\frac{\mu_t(x)}{\pi(x)}\right|^2}\int_{E}W_{\mu_t}(x\to z)\dd z \dd x\dd y. \end{multline*}
Reporting into \eqref{eq:entropydissipationtaylor} and using \eqref{eq:dissipationphipol2} yields
\begin{multline}\label{eq:entropydissipationsquarediscrete}
    \mathcal{H}[\mu_{t+1}|\pi] - \mathcal{H}[\mu_t|\pi] \\ 
    \leq -\frac{1}{2}\iint_{E\times E} \pi(x)W_{\mu_t}(x\to y){\left(1-\int_E W_{\mu_t}(x\to z)\dd z\right)}{\left|\frac{\mu_{t}(y)}{\pi(y)}-\frac{\mu_t(x)}{\pi(x)}\right|^2}\dd x\dd y.
\end{multline}
Since $W_{\mu_t}(x\to y)\geq0$ and $\int_{E} W_{\mu_t}(x\to z)\dd z\leq 1$, the right-hand side is non negative which concludes the proof. 
\end{proof}

\subsubsection{Exponential decay of the entropy}\label{sec:exponentialdecaydiscrete}

Under the Assumptions \ref{assum:hbounded} and \ref{assum:monotonicity}, it is possible to improve the result of Lemma \ref{lemma:discreteentropynonincreasing} and to prove a quantitative exponential decay result. Since the entropy is given by \eqref{eq:entropysquare}, owing to \eqref{eq:entropydissipationsquarediscrete}, the goal is to bound from below by a multiple of $\pi(y)$ the quantity 
\[W_{\mu_t}(x\to y){\left(1-\int_E W_{\mu_t}(x\to z)\dd z\right)},\]
for all $x,y\in E$ and uniformly in $t$.

The main consequence of Assumption \ref{assum:monotonicity} is the following lower bound.

\begin{lemma}\label{lemma:boundbelowW}
Let $h:[0,+\infty)\to[0,1]$ be a continuous non-decreasing acceptance function which satisfies the relation \eqref{eq:acceptancerelation}. Let $\Theta$ be a proposal distribution which satisfies Assumption \ref{assum:monotonicity}. Then for all $\mu\in\pdf_0(E)$ and all $x,y\in E$,
\[W_{\mu}(x\to y)\geq \lambda(\mu) \pi(y),\quad\lambda(\mu) := c^-{\left(\inf_{x\in E}\frac{\mu(x)}{\pi(x)}\right)}h(1)\in (0,1].\]
As a consequence, using the fact that $c^-$ is non decreasing and Lemma \ref{lemma:infsupdiscrete}, for any $t\in\N$, it holds that 
\[W_{\mu_t}(x\to y)\geq \lambda_0 \pi(y),\quad\lambda_0 := c^-{\left(\inf_{x\in E}\frac{\mu_0(x)}{\pi(x)}\right)}h(1).\]
\end{lemma}

\begin{proof}
Let us first prove that for any bounded interval $[a,b]$ with $a>0$, it holds that
\[\inf_{(x,y)\in[a,b]^2} yh\left(\frac{x}{y}\right)=ah(1).\]
Since $h$ is continuous non-decreasing, for each $y\in[a,b]$, the function $x\in[a,b]\mapsto yh(x/y)$ has a minimum in $x=a$. This shows that the minimum on the function $(x,y)\in[a,b]^2\mapsto yh(x/y)$ is attained on the segment $\{(a,y),y\in[a,b]\}$. Since $yh(a/y)=ah(y/a)$, the same reasoning shows that this minimum is attained when $y=a$. The conclusion follows. Then, using Assumption \ref{assum:monotonicity} and applying this result with $a = c^-{\left(\inf_{x\in E}\frac{\mu(x)}{\pi(x)}\right)}$ yields
\begin{align*}
    W_{\mu}(x\to y)=\Theta_{\mu}(y|x)h(\alpha_{\mu}(x,y))&=\frac{\Theta_{\mu}(y|x)}{\pi(y)}h\left(\frac{\Theta_{\mu}(x|y)\pi(y)}{\pi(x)\Theta_{\mu}(y|x)}\right)\pi(y)\\
    &\geq c^-{\left(\inf_{x\in E}\frac{\mu(x)}{\pi(x)}\right)}h(1)\pi(y).
\end{align*}
\end{proof}

We are now ready to prove that in this case, the relative entropy converges exponentially fast towards zero. 

\begin{proposition}\label{prop:entropyexpdecaydiscrete}
Under Assumption \ref{assum:hbounded} and Assumption \ref{assum:monotonicity}, there exists a constant ${\lambda}\in (0,1)$ such that for all $t\in\N$,
\[\mathcal{H}[\mu_{t}|\pi] \leq (1-{\lambda})^t \mathcal{H}[\mu_0|\pi].\]
\end{proposition}

\begin{proof}
From the relation \eqref{eq:entropydissipationsquarediscrete} and using the assumptions on $h$ and $\Theta$, it holds that
\[\mathcal{H}[\mu_{t+1}|\pi] - \mathcal{H}[\mu_t|\pi] \leq -2\lambda_0(1-\eta)\mathcal{H}[\mu_t|\pi].\]
Moreover by definition of $\lambda_0$, it holds that $\lambda_0\leq \eta$ and thus ${\lambda}:=2\lambda_0(1-\eta)<1$. We deduce that
\[\mathcal{H}[\mu_{t+1}|\pi] \leq (1-{\lambda})\mathcal{H}[\mu_t|\pi],\]
and the conclusion follows. 
\end{proof}

\section{Some Collective Proposal Distributions}\label{sec:collectiveproposals}

The proposal distribution can be fairly general and so far, we have not detailed how to choose it. Several choices of proposal distributions are gathered in this section. 

This proposal should use the maximum of information coming from the value of the target $\pi$, in order to increase the fitness of $\Theta_\mu$ to the true distribution. However, the proposal must also allow for some exploration of the parameter space, this problem is addressed for some of the proposals.

Here, we only intend to present some of the proposals possible, each one having pros and cons, on the theoretical and practical side.

In view of Theorem \ref{thm:poc}, we can see each of the proposal distributions presented in this section either as a specific choice of nonlinear kernel \eqref{eq:cmckernel} with its associated nonlinear process or as its particle approximation given by Algorithm~\ref{algo:cmc}. From this second perspective the proposal distribution can be seen as a specific interaction mechanism between the particles. More specifically, it depicts a specific procedure which can be interpreted as ``information sharing'' between the particles: given the positions of \emph{all} the $N$ particles at a given time, we aim at constructing the best interaction mechanism which will favour a specific aspect such as acceptance, convergence speed, exploration etc. By analogy with systems of swarming particles which exchange local information (here, the local value of the target distribution) to produce global patterns (here, a globally well distributed sample), we call this class of proposal distributions \textbf{collective}. The class of methods introduced will be referred as Collective Monte Carlo methods (CMC). On the contrary, the nonlinear kernels introduced in \cite{andrieu2007nonlinear} do not belong to this class as explained in Appendix~\ref{sec:anothernonlinear}. For each proposal we give an implementation which, starting from population of particles $(X^1,\ldots,X^N)\in E^N$, returns a proposal $Y$.

Although our theoretical results (Theorem \ref{thm:poc} and Theorem \ref{thm:convergencenonlineardiscrete}) are general enough to encompass almost all of the proposal distributions described here (see Section \ref{sec:assumptions}), the validity and numerical efficiency of each of them will be assessed in Section \ref{sec:experiments} on various examples of target distributions. 

\subsection{Metropolis-Hastings Proposal (PMH)}\label{sec:pmh}

\paragraph{Proposal distribution.}
The classical Metropolis-Hastings algorithm fits into our formalism, with
\[\Theta_\mu(\dd y|x) = q(y|x)\dd y,\]
where $q$ is a fixed random walk kernel which does not depend on $\mu$.  

\paragraph{Particle implementation.} In this case, Algorithm \ref{algo:cmc} reduces to the simulation of $N$ independent Metropolis-Hastings chains in parallel.

\subsection{Convolution Kernel Proposal (Vanilla CMC)}\label{sec:vanillacmc} 

\paragraph{Proposal distribution.}
Let us then consider the following proposal distribution given by the convolution: 
\begin{equation}\label{eq:vanillacmcproposal}\Theta_\mu(\dd y|x) = K\star\mu(y)\dd y := \left(\int_E K(y-z)\mu(\dd z)\right)\dd y,\end{equation}
where $K$ is a fixed \emph{interaction kernel}, that is a (smooth) radial function which tends to zero at infinity. Typical examples are $K(x) = \mathcal{N}(0,\sigma^2 Id)$ and $K(x) =\frac{1}{|B_\sigma(0)|}1_{B_\sigma(0)}(x)$,
where $\sigma>0$ is fixed and $B_\sigma(0)$ denotes the ball of radius $\sigma>0$ centred at 0 in $\R^d$. Note that the proposal distribution does not depend on the starting point $x$. It may happen that the proposed state falls outside $E$. In this case it will be rejected since $\pi$ is equal to zero outside $E$. One can therefore take equivalently $\Theta_\mu(\dd y|x) \propto K\star\mu(y)\1_{E}(y)\dd y$ (the same remark holds for the other collective proposal distributions).

\paragraph{Particle implementation.} At each time step $t$, each particle $i$ samples uniformly another particle $j$ and then draw a proposal $Y^i_{t}\sim K(X^j_t,\dd y)$ which can be seen as a ``mutation'' of $X^j_t$. This ``resampling with mutation'' procedure is somehow similar to a (genetic) Wright-Fisher model (see for instance \cite{etheridge2011some} for a review of genetic models). Since a ``mutation'' may or may not be accepted, it can be described as a \textit{biased Wright-Fisher model}. 

The collective aspect is twofold: first the proposal distribution allows large scale move on all the domain filled by the $N$ particles; then during the acceptance step, for a particle in $X$ and a proposal in $Y$ the acceptance ratio can be understood as a measure of discrepancy between the target ratio $\frac{\pi(Y)}{\pi(X)}$ and the observed ratio $\frac{K\star\hat{\mu}^N(Y)}{K\star\hat{\mu}^N(X)}$ between the (average) number of particles around $Y$ and the (average) number of particles around $X$. In the linear Metropolis-Hastings case with a symmetric random-walk kernel, the acceptance ratio only takes into account the target ratio. As a consequence, the acceptance probability of a proposal state depends not only on how ``good'' it is when looking at the values of $\pi$ but also on how many particles are (or are not) already around this proposal sate compared to the present state (and therefore on how accepting the proposal would improve the current global state of the system). 

\begin{rem}[Moderate interaction, part 1]\label{rem:moderateinteraction}
When $\sigma\to0$ we obtain the degenerate proposal distribution $\Theta_\mu(\dd y|x)=\mu(\dd y)$ (which does not satisfy the assumption $\Theta_\mu(\dd y|x)\in\pdf_0(E)$ in general). It would not make sense to take this proposal distribution at the particle level in Theorem \ref{thm:pathwiseestimate}. However, it makes sense to consider the case $\Theta_\mu(y|x)=~\mu(y)\dd y$ in the nonlinear kernel \eqref{eq:cmckernel} where $\mu\in\pdf_0(E)$. This degenerate proposal distribution still satisfies the assumptions of Theorem \ref{thm:convergencenonlineardiscrete} and could lead to a better rate of convergence (see Remark \ref{rem:optimalrate}). It is thus worth mentioning that this degenerate proposal distribution can also be obtained as the many-particle limit of a system of particles under an additional \textit{moderate interaction} assumption \cite{Oelschl_ger_1985,jourdain1998propagation,diez2019propagation}. See Remark~\ref{rem:moderateinteraction2} for additional details.
\end{rem}

\begin{center}
\begin{minipage}{1\linewidth}
\begin{algorithm}[H]
    \SetAlgoLined
    Draw uniformly $j\in\{1,\ldots,N\}$\;
    Draw $e\sim K$\;
    Set $Y=X^j+e$\;
\caption{Vanilla CMC proposal generation}
\label{algo:coker}
\end{algorithm}
\end{minipage}
\end{center}

\subsection{Markovian Mixture of Kernels Proposal (MoKA and MoKA-Markov)}\label{sec:moka}

\paragraph{Proposal distribution.}

A limitation of the Convolution Kernel Algorithm~\ref{algo:coker} is the fixed size of the interaction kernel. A remedy is given by the following collective proposal distribution which is a convolution with a mixture of kernels (with different sizes) with (potentially) nonlinear mixture weights: 
\[\Theta_\mu(\dd y|x) = \sum_{p=1}^P \alpha_p(\mu) K_p\star\mu(y)\dd y,\]
A possible choice for the weights is to take a solution of the following minimisation problem: 
\begin{equation}\label{eq:minprobmarmok}\min_{\alpha \in S^p} \int_E \phi\left(\frac{\frac{\sum_p \alpha_p K_p\star\mu(x)}{\int_E \sum_p \alpha_p K_p\star\mu(x')\mu(\dd x')}}{\frac{\pi(x)}{\int_E \pi(x')\mu(\dd x')}}\right)\frac{\pi(x)}{\int_E \pi(x')\mu(\dd x')}\mu(\dd x),\end{equation}
where $S_p$ denotes the $p$-simplex and where $\phi$ is convex non-negative such that $\phi(1)=0$. Typically $\phi(s)=s\log s-s+1$. In this case, it corresponds to minimising the Kullback-Leibler divergence between the probability distributions $\frac{\sum_p \alpha_p K_p\star\mu(x)\mu(\dd x)}{\int_E \sum_p \alpha_p K_p\star\mu(x')\mu(\dd x')}$ and $\frac{\pi(x)\mu(\dd x)}{\int_E \pi(x')\mu(\dd x')}$. Note that if $\mu$ is interpreted as the current knowledge (i.e. the distribution of the particles or of the associated nonlinear chain), then this minimisation problem can be understood as choosing the proposal distribution which is the closest to $\pi$ relatively to $\mu$. In our experiments, we found that choosing $\phi(s) = |1-s|$, leads to similar, slightly better results compared to $\phi(s) = s\log(s) - s +1$. Moreover, this choice is also numerically more stable so we chose to implement this version, that we call Markovian Mixture of Kernels (MoKA-Markov).

Another choice for the weights, which is \textbf{non-Markovian}, is to take $\alpha_p$ proportional to the geometric mean of the acceptance ratio of the particles which have chosen the kernel $p$ at the previous iteration. This method will be referred as Mixture of Kernels Adaptive CMC (MoKA). It shares similarities with the $D$-kernel algorithm of \cite{douc2007} and the arguments developed by the authors suggest that the two versions (Markovian and MoKA) may be asymptotically equivalent. The proof is left for future work.

\paragraph{Particle implementation.} Same as in Algorithm \ref{algo:coker} but with an additional step to choose a ``mutation kernel'' at each proposal step. The computation of the weights of the mixture can be done in a fully Markovian way at the beginning of each iteration before the proposal step or, in MoKA, are computed at the end of the iteration and used at the next iteration.  

In Section \ref{sec:experiments} we will show that this algorithm can favours initial exploration if initially the particles are in an area of low potential.

\begin{center}
\begin{minipage}{1\linewidth}
\begin{algorithm}[H]
    \SetAlgoLined
    Compute the weights $\alpha_1,\ldots,\alpha_P$ using \eqref{eq:minprobmarmok}\;
    Draw $p\in\{1,\ldots\,P\}$ with probability $(\alpha_1,\ldots,\alpha_P)$\;
    Draw uniformly $j\in\{1,\ldots,N\}$\;
    Draw $e\sim K_p$\;
    Set $Y=X^j+e$\;
\caption{Markovian Mixture of Kernels proposal generation}
\label{algo:marmok}
\end{algorithm}
\end{minipage}
\end{center}

\subsection{Kernelised Importance-by-Deconvolution Sampling (KIDS)}

\paragraph{Proposal distribution.}

Algorithms based on a simple convolution operator (Algorithms \ref{algo:coker} and \ref{algo:marmok}) keep a ``blind'' resampling step. In order to improve the convergence speed of such algorithms one may want to favour the selection of ``good'' states. For a fixed interaction kernel $K$, one can choose a proposal distribution of the form: 
\[\Theta_\mu(y|x) = K\star\nu_\mu(y),\]
where $\nu_{\mu}$ solves the following deconvolution problem with an absolute continuity constraint: 
\begin{equation}\label{eq:klkids}\min_{\nu\ll\mu}\int_E \log\left(\frac{K\star\nu(x)}{\pi(x)}\right)K\star\nu(x)\dd x.\end{equation}
That is, we are looking for a weight function $w\geq0$ which satisfies the constraint
\[\int_E w(x)\mu(dx)=1\]
and such that the measure defined by $\nu_\mu(A) = \int_A w(x)\mu(\dd x)$ minimises the KL divergence above. The function $w$ is the Radon-Nikodym derivative of $\nu$ with respect to $\mu$. In other words the proposal distribution focuses on the parts of the support of $\mu$ which are ``closer'' to $\pi$. Note also that this optimization procedure does not depend on the normalisation constant of $\pi$. Similarly to Algorithm \ref{algo:marmok}, the idea is to minimize a KL divergence between a set of proposal distributions and the distribution $\pi$. However, in the minimization problem \ref{eq:minprobmarmok} the goal was to optimize the choice of the convolution kernel $K$ among a given family while in the optimization problem \ref{eq:klkids}, the goal is to optimize the support of the distribution $\mu$ to focus on the parts which are ``closer'' to $\pi$. Finally, it should be noted that the two minimization procedures are not excluding and that both can be conducted at the same time.

\begin{rem}
Although the proposal distributions KIDS and MoKA-Markov perform well in practice, we did not manage to prove the regularity Assumption~\ref{assum:nonexpansive} because we do not know the regularity of the solutions of the minimization problems \eqref{eq:minprobmarmok} and \eqref{eq:klkids}. Since the weights $\alpha_p$ are bounded in \eqref{eq:minprobmarmok}, MoKA-Markov satisfies Assumption \ref{assum:lipschitz} but we did not manage to prove this assumption for KIDS for the same reason as before. However, Vanilla CMC satisfies all the assumptions since the proposal distribution is a simple convolution with a smooth kernel. 
\end{rem}

\paragraph{Particle implementation.}
In the case of an empirical measure $\hat{\mu}^N=\frac{1}{N}\delta_{X^i}$, the Radon-Nikodym weight function $w$ can simply be seen as a vector of $N$ weights $(w^1,\ldots,w^N)\in[0,1]^N$ such that $\sum_i w^i=1$ and the measure $\nu_{\hat{\mu}^N}$ is thus a weighted empirical measure: 
\begin{equation}\label{eq:weightedempiricalmeasure}\nu_{\hat{\mu}^N}:=\sum_{i=1}^N w^i \delta_{X^i}.\end{equation}
The deconvolution procedure gives more weight to the set of particles that are already ``well distributed'' according to $\pi$. These particles are thus more often chosen in the Wright-Fisher resampling step (see Algorithm \ref{algo:coker}). 

Note that although the weighted empirical measure proposal \eqref{eq:weightedempiricalmeasure} is very reminiscent of an Importance Sampling procedure, the computation of the weights here follows from a completely different idea. 

In practice we solve the deconvolution problem using the Richardson-Lucy algorithm \cite{Richardson:72,lucy1974iterative} (also known as the Expectation Maximisation algorithm). See for instance \cite[Section 5.3.2]{natterer2001mathematical} where it is proved that the iterative algorithm below converges towards a minimiser of the Kullback-Leibler divergence~\eqref{eq:klkids} in the case of an empirical measure $\mu$. Note that the computation of the weights (Richardson-Lucy loop) can be done before the resampling step. 

\begin{center}
\begin{minipage}{1\linewidth}
\begin{algorithm}[H]
    \SetAlgoLined
    Set $w^{(0)}_i=1$ for all $i\in\{1,\ldots,N\}$\;
    \For{$s=0$ to $S-1$}{
        For all $i\in\{1,\ldots,N\}$, update the weight by: $w_i^{(s+1)} = w_i^{(s)}\sum_{j=1}^N\frac{\pi(X^j_t)K(X^i_t-X^j_t)}{\sum_{k=1}^N w^{(s)}_k K(X^j_t-X^k_t)}$\;
        }
    Set $w_i=w_i^{(S)}$ for all $i\in\{1,\ldots,N\}$\;
    Normalize the weights $(w_1,\ldots,w_N)$\;
    Draw $j\in\{1,\ldots,N\}$ with probability $(w_1,\ldots,w_N)$\;
    Draw $e\sim K$\;
    Set $Y=X^j+e$\;
\caption{Kernel Importance-by-Deconvolution Sampling proposal generation}
\label{algo:kids}
\end{algorithm}
\end{minipage}
\end{center}

\subsection{Bhatnagar-Gross-Krook sampling (BGK)}

\paragraph{Proposal distribution.}
In Algorithms \ref{algo:coker}, \ref{algo:marmok} and \ref{algo:kids}, the proposal distribution is based on a (mixture of) symmetric kernels: this symmetry property is reflected in the proposal distribution and might not well represent the local properties of the target distribution. In dimension $d\geq2$, we can adopt a different strategy by sampling proposals from a multivariate Gaussian distribution with a covariance matrix that is computed locally. An example is given by the following proposal distribution:
\[\Theta_\mu(\dd y|x) = \left(\int_E G_{\widehat{\Sigma}_\mu(z)}(\widehat{m}_\mu(z)-y)\mu(\dd z))\right)\dd y,\]
where
\begin{equation}\label{eq:samplemean}\widehat{m}_\mu(z) = \frac{1}{\int_E K(z-z')\mu(\dd z')}\int_E K(z-z')z'\mu(\dd z'),\end{equation}
and
\begin{equation}\label{eq:samplecov}\widehat{\Sigma}_\mu(z) = \frac{1}{\int_E K(z-z')\mu(\dd z')}\int_E K(z-z') (z'-\widehat{m}_\mu(z))(z'-\widehat{m}_\mu(z))^\mathrm{T}\mu(\dd z'),\end{equation}
and where $K$ is a fixed interaction kernel. This proposal distribution and the associated transition operator are reminiscent of a Bhatnagar-Gross-Krook (BGK) type operator~\cite{bhatnagar1954model}.

In the particular case of $K(x,y)\equiv1$, we have a more simple algorithm which can be interpreted as a Markovian version of the Adaptive Metropolis-Hastings algorithm introduced by \cite{haario2001adaptive}. However such algorithm does not benefit from the appealing properties of local samplers. Indeed, when the target distribution is multimodal, we can take more advantageously as interaction kernel $K(x)\propto \1_{|x|<\sigma}$, where the threshold $\sigma$ allows the proposal to be adapted to the local mode. Note that moving across the modes is still possible thanks to the choice of another particle at the first step (see proposal algorithm below). The main issue is the choice of $\sigma$ that must be higher than the size of the modes but smaller than the distance between the modes.

\paragraph{Particle implementation.} Each particle, say $X^i_t$ samples another particle, say $X^j_t$. Then we compute the local mean and covariance around $X^j_t$ and we draw a proposal $Y^i_t$ for $X^i_t$ according to a Normal law with the locally computed parameters. As before it may be cheaper to compute and store all the local means and covariances before the resampling loop. 

\begin{center}
\begin{minipage}{1\linewidth}
\begin{algorithm}[H]
    \SetAlgoLined
    Draw $j\in\{1\,\ldots,N\}$ uniformly\; 
    Compute $\kappa=\sum_i K(X^i_t,X^j_t) $ \;
    Compute the local mean $\widehat{m}_{X^j}=\frac{1}{\kappa} \sum_{i}K(X^i_t,X^j_t)X^i_t$\;
    Compute the local covariance $\widehat{\Sigma}_{X^j}=\frac{1}{\kappa} \sum_{i} K(X^i_t,X^j_t)(X^i_t-\widehat{m}_{X^j_t})(X^i_t-\widehat{m}_{X^j_t})^\mathrm{T}$\;
    Draw $Y\sim \mathcal{N}(\widehat{m}_{X^j_t},\widehat{\Sigma}_{X^j_t})$\;
\caption{BGK proposal generation}
\label{algo:bgk}
\end{algorithm}
\end{minipage}
\end{center}

\section{GPU implementation}\label{sec:gpu}

A major bottleneck for the development of collective samplers based on interacting particles is the computational cost associated to the simulation of $N^2$ pair-wise interactions between particles at every time step. To overcome this issue, many methods have been proposed throughout the years, often using techniques borrowed from $N$-body simulations or molecular dynamics and developed since the 60's. Let us cite for instance the Verlet and cell-list methods for short-ranged interactions \cite{verlet1967computer,hockney2021computer,SIGURGEIRSSON2001766}, the super-particle and fast multipole methods \cite{barnes1986hierarchical,ROKHLIN1985187,GREENGARD1987325} for long-ranged interactions or the more recent Random Batch Method introduced by~\cite{JIN2020108877}. In the MCMC literature, such methods have been applied for instance to particle smoothing in \cite{klaas2006fast}.  

Most importantly, the last decade has seen the introduction of massively parallel GPU chips. Beyond the training of convolutional neural networks, GPUs can now be used to simulate generic particle systems at ever faster rates and, noticeably, without relying on numerical approximations, contrary to classical $N$-body simulation algorithms. One goal of the present paper is to discuss the consequences of this hardware revolution on Monte Carlo sampling, from the practitioner point of view. The use of GPU in this context traces back at least to~\cite{lee2010utility}. Although in this article, the authors have focused on traditional SMC and MCMC methods which are often embarrassingly parallelizable (with a linear complexity in $N$), it should be noted that, as such, GPU programming is not a trivial task. In particular, let us quote \cite{lee2010utility} on the CUDA language, which is the  development environment for the Nvidia GPU chips : \textit{``a programmer proficient in C should be able to code effectively in CUDA within a few weeks of dedicated study"}. In the present article, our goal is to advocate the use of very recent GPU routines developed for high-level languages (in particular Python) which allow the end-user to benefit from GPU acceleration without any particular knowledge or experience in the CUDA programming language. Using these techniques, our implementation offers a transparent Python interface for massively parallelized MCMC methods, including both the traditional methods (Metropolis-Hastings algorithm, SMC) which are linear in $N$ and the population based methods (CMC-related methods) with a quadratic footprint. For the latter methods, the main specific technical difficulties are briefly described below. 

From the evaluation of kernel densities to the computation of importance weights in the Richardson-Lucy loop (Algorithm \ref{algo:kids}), the bottleneck of the Collective Proposal framework presented in Section \ref{sec:collectiveproposals} is the computation of an off-grid discrete convolution of the form:
\begin{equation}
    a_i = \sum_{j=1}^N K(X_t^i,X_t^j)\, b_j \label{eq:kernel_product}
\end{equation}
for all $i$ between $1$ and $N$, with arbitrary weights $b_j$. All CMC-related methods have $\mathcal{O}(N^2)$ time complexity, with a constant that is directly related to the efficiency of the underlying implementation of the ``kernel sum'' \eqref{eq:kernel_product}.

In the machine learning literature, this fundamental operation is often understood as a matrix-vector product between an $N$-by-$N$ \emph{kernel matrix} $(K(X_t^i,X_t^j))$ and a vector $(b_j)$ of size $N$. Common implementations generally rely on linear algebra routines provided e.g. by the PyTorch library \cite{paszke2017automatic} and have a quadratic memory footprint: even on modern hardware, this prevents them from handling populations of more than $N = 10^4$ to $10^5$ samples without making approximations or relying on specific sub-sampling schemes \cite{yang2012nystrom}.

Fortunately though, over the last few years, efficient GPU routines have been developed to tackle computations in the mould of \eqref{eq:kernel_product} with maximum efficiency. These methods can be  accessed through the KeOps extension for
PyTorch \cite{pytorch}, NumPy \cite{numpy}, R \cite{r} and Matlab \cite{matlab}, that is developed by the third author \cite{keops,these_jean,keops_neurips} and freely available at
\url{https://www.kernel-operations.io}. In practice, this library supports arbitrary kernels on the GPU with a linear memory footprint, log-domain implementations for the sake of numerical stability and outperforms baseline PyTorch GPU implementations by one to two orders of magnitude.
Computing the convolution of \eqref{eq:kernel_product} with a cloud of $N = 10^6$ points in dimension $3$ takes no more than 1s on a modern chip. Pairwise interactions between populations of $N = 10^5$ samples may also be simulated in no more than 10ms, without making any approximation.

We run all the tests of Section \ref{sec:experiments} using a single gaming GPU, the Nvidia GeForce RTX 2080 Ti on the GPU cluster of the Department of Mathematics at Imperial College London. With $10^4$ to $10^6$ particles handled at any given time, our simulations run in a handful of seconds at most, with performances that enable the real-time sampling of large populations of \emph{interacting} samples. Our code is available on the third author's GitHub page. 

\section{Numerical experiments}\label{sec:experiments}

We now run experiments on several target distributions in low and moderately high dimension. 

Our code and its documentation are available online at 

\begin{center}
\url{http://www.kernel-operations.io/monaco}
\end{center}
and on the GitHub page of the third author. All our experiments can be run on Google Colaboratory (\url{colab.research.google.com}) with a GPU within a few seconds to a few minutes, depending on the method and number of independent runs.

\subsection{Methodology}

\subsubsection{Methods in competition}

Among the variants of our methods, we show the results for
\begin{itemize}
    \item The Vanilla CMC Algorithm \ref{algo:coker} (named CMC in the following) with a fixed interaction kernel $K$ which is taken equal to the indicator of the ball centered at zero with a given radius $R$. The quadratic cost of this algorithm is only due to the computation of discrete convolutions. 
    \item The Markovian version of the MoKA algorithm \ref{algo:marmok} with several interaction kernels and adaptive weights (named MoKA Markov in the following). In addition to the computation of the discrete convolutions, solving the optimization problem \eqref{eq:minprobmarmok} slightly increases the computational cost.
    \item The non-Markovian version of MoKA outline in Section \ref{sec:moka} (named MoKA in the following). Thanks to the simplified computation procedure of the mixture weights compared to MoKA Markov, this algorithm has the same computation cost as CMC.
    \item The hybrid algorithm which features both the adaptive interaction kernels of MoKA (non-Markovian) and an adaptive weighting of the particles obtained with Algorithm \ref{algo:kids}. It will be named MoKA KIDS in the following. It is a significantly heavier method than CMC, as it relies on the Richardson-Lucy iterations to optimise the deconvolution weights. In the experiments shown below, this algorithm is roughly five times slower than CMC in terms of computation time (but needs less iterations to converge). 
\end{itemize}

Please note that we do not include the BGK and KIDS samplers in these experiments: although we believe that the ideas behind these methods are interesting enough to justify their presentation, we observe that they generally do not perform as well as the other CMC variants and leave them aside for the sake of clarity.

For the competitors, we consider the three following algorithms. 
\begin{itemize}
    \item As a baseline, we include a parallel implementation of Metropolis-Hastings (named PMH in the following) with a number of parallel runs that is equal to the number of particles in CMC and its variants. 
    \item The Safe Adaptive Importance sampling (SAIS) algorithm from \cite{delyon2019adaptive}, which is one of the state-of-the-art importance sampling based methods. It is detailed in Section \ref{sec:linksis}. We note that unlike all the other Markovian methods, the memory and computational cost of SAIS significantly increases with the number of iterations since the number of particles is not constant -- to the best of our knowledge, no procedure has yet been proposed to remove particles with time. We have implemented the so-called mini-batch version as proposed in \cite{delyon2019adaptive} but without the sub-sampling procedure since the computation time is not a problem with our fast GPU implementation. The parameters (in particular the kernel sequence and mixing weights) are chosen according to the recommendation of the authors as detailed in Section \ref{sec:linksis}. 
    \item A simple (non-adaptive) Sequential Monte Carlo algorithm as described in Section \ref{sec:smc} with a linear tempering scheme for the potentials and a simple Metropolis-Hastings mutation kernel. 
\end{itemize}

We emphasize the fact that in our implementation, all the methods, including PMH and SMC which have a linear complexity in $N$, are parallelized on a GPU using PyTorch \cite{pytorch} and KeOps \cite{keops}. All the methods are therefore comparable in terms of computation time. 

\subsubsection{Initialization and target distributions}

We have chosen various target distributions in order to illustrate the main difficulties and to outline the strengths and weaknesses of the various methods in competition. As an illustration we start with a banana-shaped distribution in low dimension (Section \ref{sec:banana}) then we consider several variants of the classical Gaussian mixture example introduced by \cite{cappe2008adaptive} in moderately high dimension (up to dimension 13), see Sections \ref{sec:cappe} and \ref{sec:manygaussians}. 

Another example of target distribution is included in Appendix \ref{appendix:Cauchy} and we include results for the estimation of the normalizing constant of some of the target distributions below in Appendix~\ref{appendix:CMCIS}. Experiments in non-Euclidean spaces such as the Poincaré hyperbolic plane and the group of 3D rotation matrices are available online in our documentation.

We always clip distributions on the unit (hyper-)cube $[0,1]^d$. For all the methods we have considered two types of initialization.
\begin{itemize}
    \item A simple initialization where all the $N$ particles are sampled uniformly and independently in the unit hyper-cube. 
    \item A difficult initialization with $N$ particles independently distributed according to 
    \begin{equation}\label{eq:cornerinitial}
    X^i_0 = (0.9,\ldots,0.9)^T+0.1U^i_0
    \end{equation}
    where $U^i_0\sim\mathcal{U}([0,1]^d)$, that is all the particles are located in a small corner of the domain. 
\end{itemize}

\subsubsection{Energy Distance}

We compare the results in term of Energy distance \cite{rizzo2016energy} between a true sample of the target distribution and the population of particles at each step. 

The Energy Distance (ED) between two probability distributions $\mu,\nu$ in $\R^d$ can be defined by several equivalent formulas. Perhaps, the simplest form is 
\[\mathcal{E}(\mu,\nu) = \E|X-Y| - \frac{1}{2}\E|X-X'| - \frac{1}{2}\E|Y-Y'|, \]
where $X,X'\sim \mu$ and $Y,Y\sim \nu$ are independent. In the machine learning literature, it is often understood as a kernel norm comparable to other popular distances such as the Wasserstein distance \cite{feydy2020geometric,ziel2021energy}. This interpretation comes from the formulation 
\[\mathcal{E}(\mu,\nu) = \frac{1}{2} \langle \mu-\nu, K\star(\mu-\nu)\rangle,\]
where $K(x,y) = -|x-y|$ and where we recall that the convolution operator is defined by $K\star\mu(x) := \int_{\R^d} K(x,y)\mu(\dd y)$. Among all the possible kernels, this particular choice is the simplest one which ensures that the distance between two Dirac masses is equal to the distance between the points. Finally, by writing this latter expression in the Fourier space, the ED can also be understood as a weighted $L^2$ distance between the characteristic functions:
\[\mathcal{E}(\mu,\nu) = C_d\int_{\R^d} \frac{|\varphi_\mu(z)-\varphi_\nu(z)|^2}{|z|^{d+1}}\dd z,\]
for a constant $C_d>0$ and where $\varphi_\mu(z) = \E_{X\sim \mu}[\e^{iz\cdot X}]$ is the characteristic function of the probability measure $\mu$. Note also that in dimension 1, the Energy Distance is equal to the Cram\'er's squared $L^2$ distance between the cumulative distribution functions $F_\mu$ and $F_\nu$ of $\mu$ and $\nu$ :
\[\mathcal{E}(\mu,\nu) = \int_\R |F_\mu(z)-F_\nu(z)|^2\dd z.\]
It can be shown that the ED satisfies all the properties of a metric and in particular $\mathcal{E}(\mu,\nu)=0$ if and only if $\mu=\nu$. 

The Energy Distance can serve as a baseline to test the null hypothesis that two random vectors $\mathcal{X} = (X_i)_{i\in\{1,\ldots,n\}}$, $X_i \in \R^d$ and $\mathcal{Y} = (Y_j)_{j\in\{1,\ldots,m\}}$, $Y_j \in \R^d$ are sampled from the same distribution. For this, we simply compute the ED between the associated empirical measures $\hat{\mu}_{\mathcal{X}} = \frac{1}{n}\sum_{i=1}^n \delta_{X^i}$ and $\hat{\mu}_{\mathcal{Y}} := \frac{1}{m}\sum_{j=1}^m \delta_{Y^j}$, namely:  
\begin{align*}
\mathcal{E}_{n,m}(X,Y) &:= \mathcal{E}(\hat{\mu}_{\mathcal{X}},\hat{\mu}_\mathcal{Y}) \\
&= \frac{1}{nm} \sum_{i,j} | X_i - Y_j | - \frac{1}{2n^2} \sum_{i,j} | X_i-X_j | - \frac{1}{2m^2} \sum_{i,j}| Y_i -Y_j |.\end{align*}

It is shown in \cite{szekely2004testing, rizzo2016energy} that this quantity can serve as a test statistic for equal distributions which implies in particular that $\mathcal{E}_{n.m}$ tends to 0 as $n,m\to+\infty$ when $\mathcal{X}$ and $\mathcal{Y}$ are i.i.d. samples.

In order to measure the quality of the various samplers, we will compare the ED between the output samples and a true sample of the target distribution, both of size $N$. We consider that the sampler successfully manage to sample the target distribution when the final samples fall within the $90\%$ prediction interval (computed with 1000 independent observations) around the average Energy Distance between two i.i.d. exact samples of size $N$. It means that the sample cannot be distinguished from a true sample.

\subsection{Banana-shaped distribution}\label{sec:banana}

As a first illustrative example in dimension 2, we consider a mixture of three (equally weighted) Gaussian distributions and a banana-shaped distribution, as introduced in \cite{haario1999adaptive,delyon2019adaptive}. The banana-shaped distribution is a classical test case for sampling algorithms. It is often considered difficult to sample from this distribution due to its curved, oblong shape. In dimension~2, however, it is easy to find an appropriate range of parameters such that all the methods including PMH, SAIS, SMC or a standard rejection sampler achieve very good results within a computation time which does not exceed a few seconds. Thus we mainly use this introductory example as a qualitative illustration of the behaviour of the collective algorithms CMC, MoKA, MoKA Markov and MoKA KIDS that we simply compare to PMH. In dimension 2, it is indeed possible to directly visualize the results for these algorithms, as shown in in Figure \ref{fig:banana_final} and in the supplementary videos available on the GitHub page of the project. More in-depth quantitative comparisons on much more difficult test cases in higher dimensions will be presented in the following sections. 

\begin{figure}[ht]
    \centering
    \subfloat[CMC without exploration]{\includegraphics[width=0.24\textwidth]{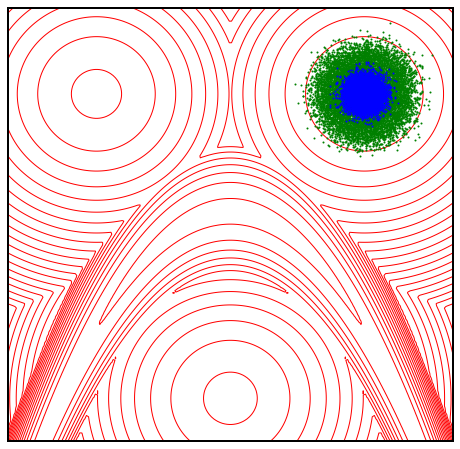}\label{subfig:banana_final_CMC_withoutexp}}
    \subfloat[PMH]{\includegraphics[width=0.24\textwidth]{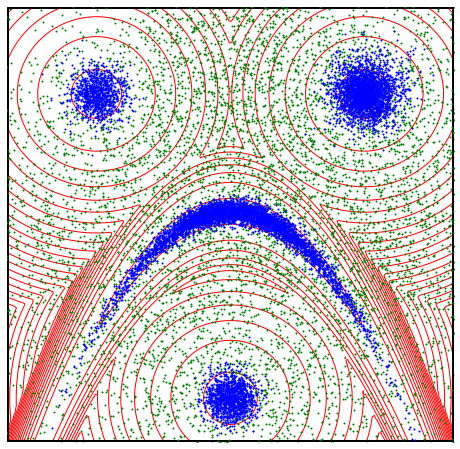}\label{subfig:banana_final_PMH}}
    \subfloat[CMC with exploration]{\includegraphics[width=0.24\textwidth]{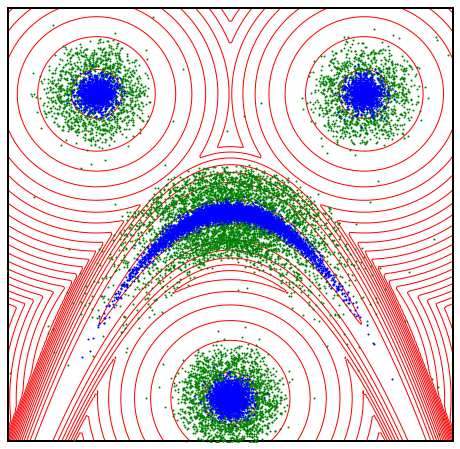}\label{subfig:banana_final_CMC}}
    \subfloat[MoKA Markov]{\includegraphics[width=0.24\textwidth]{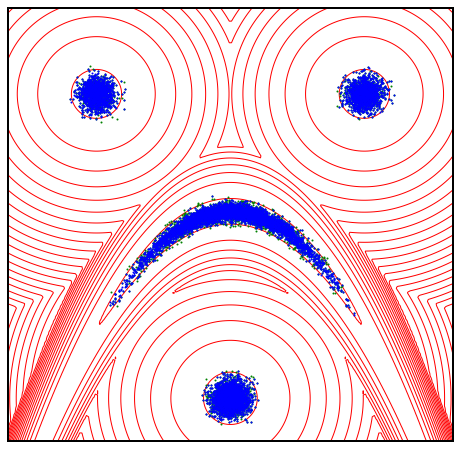}\label{subfig:banana_final_MoKAMarkov}}
    \caption{The level sets of the target distribution (red) and the final states of different algorithms after 50 iterations. The particles are displayed in blue and the rejected proposals (at the previous iteration) in green. There are $N=10^4$ particles. (a) Vanilla CMC without an exploration proposal and a too small proposal distribution. The particles concentrate on the top-right mode and cannot escape it because the proposals cannot reach the other modes. (b) PMH with a large proposal distribution reduces to a rejection sampler (c) CMC with a large exploration proposal. Although the target distribution is correctly sampled, many proposals in green are never accepted. (d) MoKA Markov with a mixture of small and large proposal distributions. The final state for MoKA and MoKA KIDS is similar. Thanks to the automatic choice of a small interaction kernel, most of the proposals are accepted.}
    \label{fig:banana_final}
\end{figure}

Starting from the initialization \eqref{eq:cornerinitial}, the particles are initially very close to one mode and if the radius of the interaction kernel is taken too small, then the particles will stay in this mode forever, as it can be seen in Figure \ref{subfig:banana_final_CMC_withoutexp}. It is thus necessary to add an exploration mechanism. For the three MoKA variants, this can be done automatically by simply adding a large proposal distribution within the proposal mixture. Alternatively, it is possible to consider a simple upgraded version of the Vanilla CMC algorithm \eqref{eq:vanillacmcproposal} by taking the proposal distribution 
\[\Theta_\mu(\dd y|x) = (1-\varepsilon)K\star\mu(y)\dd y +\varepsilon Q(\dd y|x), \]
where the exploration proposal $Q$ is for instance $Q(\dd y|x) \propto \exp(-\frac{|y-x|^2}{2\sigma^2})$ with $\sigma^2$ large. In the remaining of this section, we will always consider this proposal (and simply call it CMC) with $\varepsilon=0.01$, $\sigma = 0.5$ and the radius of the interaction kernel $K$ equal to 0.1. For the MoKA algorithms, we consider a mixture of proposal distributions of sizes 0.01, 0.03, 0.1 and 0.5. As a control system, we consider PMH with a mutation kernel of size 0.5. 

The results are visualized in dimension 2 in the supplementary videos available on the GitHub page of the project. We summarize below the main observations. 

\begin{itemize}
    \item In PMH (Figure \ref{subfig:banana_final_PMH}), due to the large mutation kernel, blind proposals are made all over the space and the algorithm behaves essentially like a rejection algorithm. Consequently, the acceptance rate is very low, about $3\%$. Note that PMH with a too small mutation kernel would lead to the same situation as shown in Figure \ref{subfig:banana_final_CMC_withoutexp}.
    \item In CMC, the particles move collectively like a coherent swarm. They initially massively concentrate on the first mode that they find (on the top-right corner). Thanks to the large exploration proposal, a few particles are able to find the other modes. These well-located particles then slowly attract the other ones and the correct balance between the different modes is progressively attained. At the end of the simulation, there are still many proposals which are not accepted because the interaction kernel is too large. The final acceptance rate is about $16\%$. 
    \item In MoKA Markov, a very large exploration proposal is initially selected (see Figure~\ref{subfig:banana_weights_MoKAMarkov}) and similarly to PMH, the particles are quickly scattered all over the space. However, unlike PMH, once particles are present in all the modes, a better adapted small interaction kernel is chosen (see Figure~\ref{subfig:banana_weights_MoKAMarkov}) and similarly to CMC, the balance between the modes is progressively but much more quickly attained. The final acceptance rate is about $90\%$. 
    \item In MoKA, a nontrivial mixture of medium and small proposal distributions is kept throughout the simulation (see Figure~\ref{subfig:banana_weights_MoKA}). However, we note that the larger ones are used more often at the beginning of the simulation, which demonstrates the ability of the algorithm to automatically explore the space. The final acceptance rate is about $80\%$. 
    \item In MoKA KIDS, the behaviour is similar to MoKA except that the badly-located particles are very quickly attracted by better-located ones so that the initial phase is much faster. The final acceptance rate is about $70\%$. 
\end{itemize}

\begin{figure}[ht]
    \centering
    \subfloat[MoKA]{\includegraphics[width=0.35\textwidth]{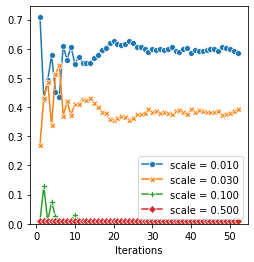}\label{subfig:banana_weights_MoKA}}
    \subfloat[MoKA Markov]{\includegraphics[width=0.35\textwidth]{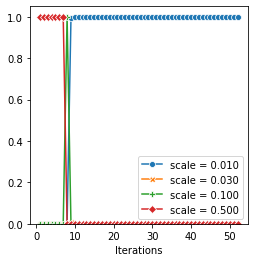}\label{subfig:banana_weights_MoKAMarkov}}
    \caption{Weights of the different proposal distribution in MoKA (a) and MoKA Markov (b). MoKA KIDS is not shown but is similar to MoKA.}
\end{figure}

\subsection{Gaussian mixtures}\label{sec:cappe} 

\subsubsection{Target distribution}

As a first challenging target distribution, we revisit the classical test case introduced in \cite{cappe2008adaptive} and consider a mixture of two Gaussian distributions in dimension $d=12$: 
\[\pi = w_1 \mathcal{N}(m + \mu_d,\sigma_d^2 I_d) + w_2 \mathcal{N}(m - \mu_d,\sigma_d^2 I_d),\]
where $m=(0.5,\ldots,0.5)^\mathrm{T}$ is the center of the domain $[0,1]^d$. The target distribution is parametrized $\mu_d\in \R^d$, $\sigma_d>0$ and the weights $w_1,w_2>0$. We will always consider the standard deviation
\[\sigma_d = \frac{1}{2} \sqrt{\frac{0.4}{d}}.\]
We will then consider two cases. 
\begin{enumerate}
    \item The example considered in the seminal article \cite{cappe2008adaptive} is given by
    \begin{equation}\label{eq:cappe_simple}w_1 = w_2 = \frac{1}{2},\quad \mu_d = \frac{1}{4\sqrt{d}}(-1,1,\ldots,1)^\mathrm{T}.\end{equation}
    \item A slightly more difficult example is given by
    \begin{equation}\label{eq:cappe_difficult}w_1 = 0.25,\quad w_2 = 0.75,\quad \mu_d = \frac{1}{8}(-1,1\ldots,1)^\mathrm{T}.\end{equation}
\end{enumerate}

The second example is more difficult because the distance between the two modes is larger and the two modes are not equally weighted. 

\subsubsection{Initialization and parameters}

All the algorithms are initialized according to the difficult case \eqref{eq:cornerinitial} where all the particles initially lie in a corner of the domain. Note also that the particles are initially closer to the first mode. 

The kernel size in CMC is taken equal to 0.25. For the three MoKA variants, the proposal distribution is made of a mixture of three kernels with sizes 0.25, 0.4 and 0.55. A detailed discussion on how to choose these parameters is postponed to Section \ref{sec:tuningCMC}. For PMH, SAIS and SMC, we have tested all the kernel sizes between 0.05 and 0.8 with step 0.05 and we show the results for the best ones. For the simpler example \eqref{eq:cappe_simple}, the best kernel size for PMH is 0.25, for SAIS it is 0.4 and for SMC it is 0.3. For the more difficult example \eqref{eq:cappe_difficult}, the best kernel size for PMH is 0.8, for SAIS it is 0.5 and for SMC it is 0.4. Regarding the other parameters, the ESS threshold in SMC is set to $0.75N$, there are 150 MH steps between two resampling steps and a linear tempering scheme with 25 steps is used from the initial distribution to the target distribution. For SAIS, the sequence of kernel bandwiths and mixture weights are taken according to the recommendations of the authors (see Appendix \ref{sec:linksis}). We take $N=10^5$ particles in CMC and its variants, PMH and SMC. The number of particles added to the system at each iteration in SAIS is such that the system has $N$ particles at the end of the burn-in phase.  

\subsubsection{Results}\label{sec:cappe_results_GPUtime}

For each method in competition, we have launched 50 independent simulations. For each simulation, we have allowed a fixed GPU computation time of 60 seconds (we recall that all the methods are parallelized and thus comparable in terms of computation time). The best results are shown for the different methods in competition in dimension $d=12$ in Figure \ref{fig:cappeexample} below. 

\begin{figure}[ht]
    \centering
    \subfloat[Simple]{\includegraphics[scale=0.22]{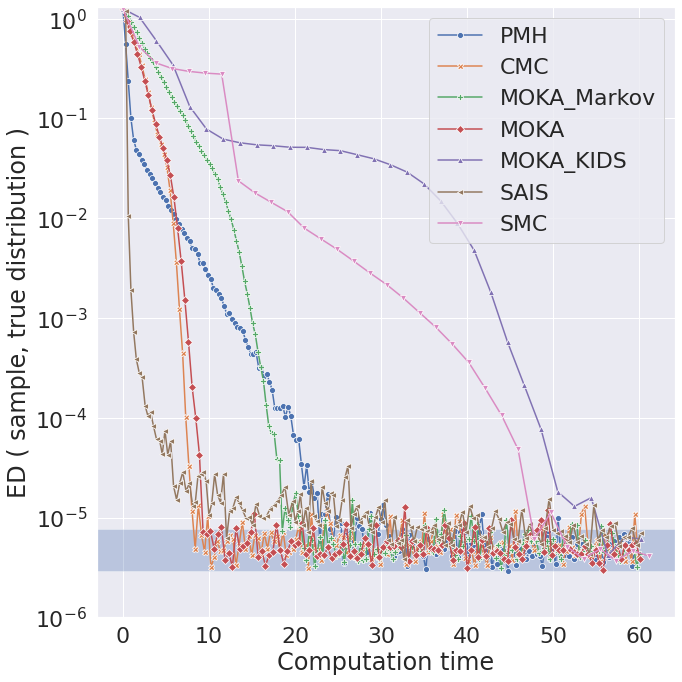}}
    \subfloat[Difficult]{\includegraphics[scale=0.22]{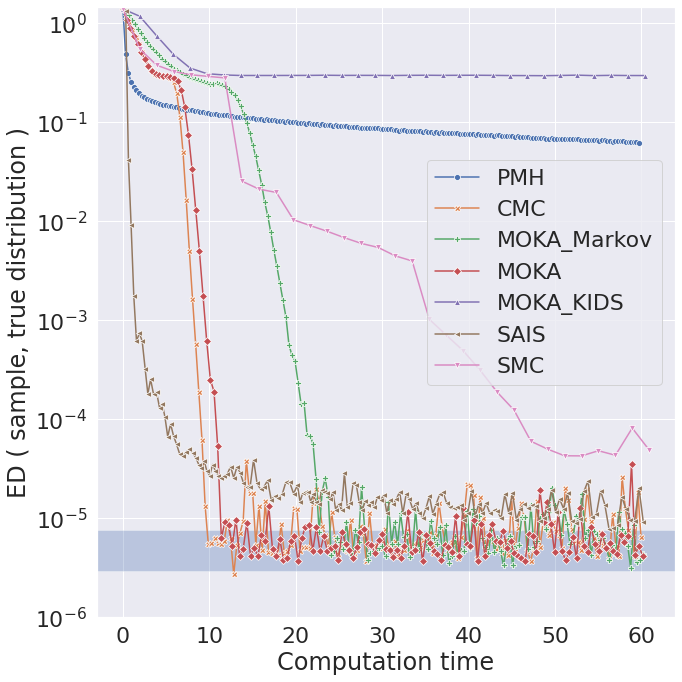}}
    \caption{The best runs among 50 independent runs for the different algorithms in the simple case (a) and the difficult case (b). The blue shaded region is the 90\% prediction interval of the energy distance between two true samples.}
    \label{fig:cappeexample}
\end{figure}

For the simpler example \eqref{eq:cappe_simple}, all the methods achieve excellent results, including PMH even though it is about three times slower than CMC. The adaptive methods MoKA Markov and MoKA KIDS also perform well but they suffer from a higher computational cost and they are thus slower than CMC on this example. MoKA has a computational cost comparable to the one of CMC and thus achieves equivalently good results. Note also that SMC is disadvantaged compared to the other methods in terms of computation time since due to the tempering scheme, this method effectively samples $\pi$ only during the few last iterations. SAIS is initially the fastest method but the convergence then slows down since the computational cost increases with the iterations. Moreover, on~50 independent experiments, SAIS failed to converge in about 30\% cases due to an early degeneracy of the particle weights which causes an over-concentration on only one mode. This is in contrast with all the other methods for which we have observed a very low variance in the results.

For the more difficult example, PMH fails to converge for any of the kernel sizes that we have tried. This situation can be expected because on the one hand, if the mutation kernel is small then the particle can never leave the first mode that they find and on the other hand if the mutation kernel is too large then the proposals are almost never accepted. This flaw is corrected in CMC and its MoKA variants by incorporating interactions between the particles and an adaptive choice of kernels. In particular, CMC, MoKA and MoKA Markov do not seem affected by the increased difficulty. Only MoKA KIDS fails to converge due to an early degeneracy of the particle weights which prevents them to leave the first mode. SAIS also has the same flaw as in the previous case. The performance of SMC are slightly deteriorated compared to the previous case but still achieves good results in most cases. 

\subsubsection{Comparison of the collective methods}
 
In this section we compare the behaviours of the collective algorithms CMC, MoKA, MoKA Markov and MoKA KIDS. In the results below we show the convergence in terms of number of iterations rather than in terms of computation time. This is to illustrate that the collective proposals are designed to achieve better efficiency at each iteration (at the cost of a possibly increased computation time). We include PMH in the comparison as a control system. In Figure \ref{fig:cappeexample_iter}, we show the results for the two examples \eqref{eq:cappe_simple} and \eqref{eq:cappe_difficult} during 100 iterations and for 10 independent runs. 

\begin{figure}[ht]
    \centering
    \subfloat[Simple]{\includegraphics[scale=0.22]{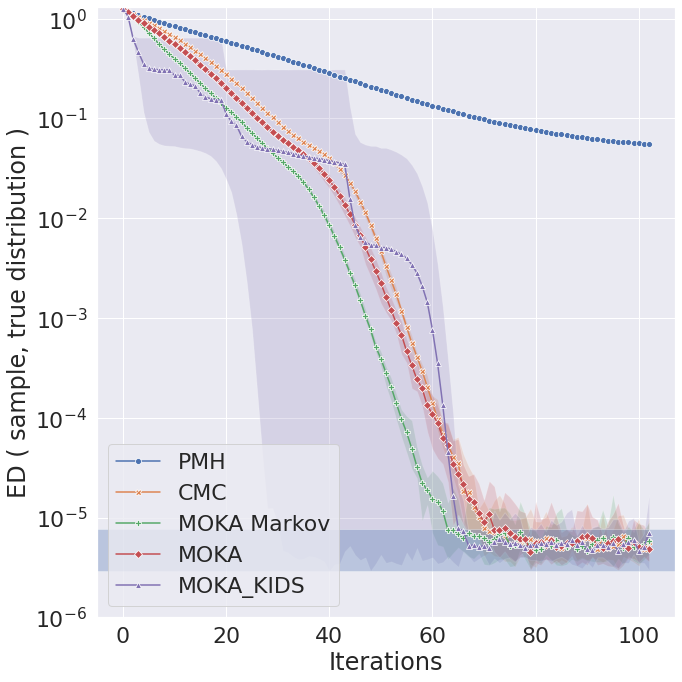}}
    \subfloat[Difficult]{\includegraphics[scale=0.22]{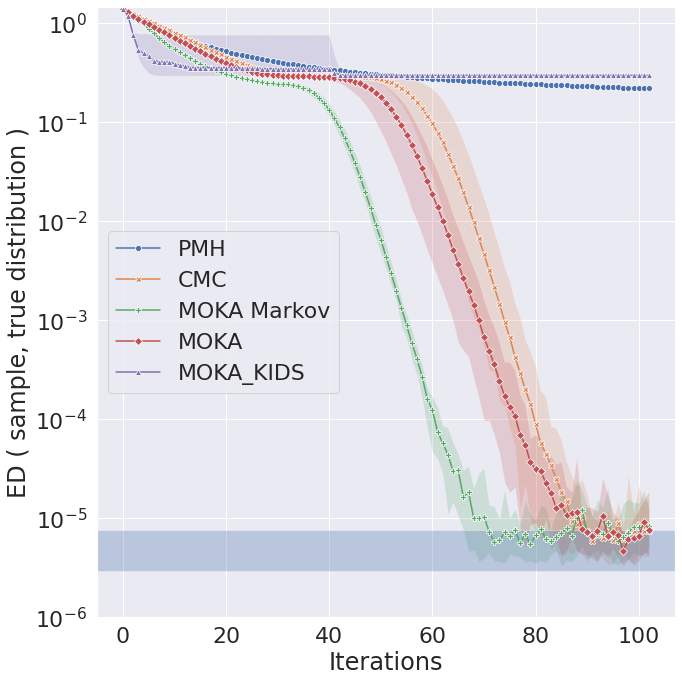}}
    \caption{Mean error for 100 iterations of the Collective algorithms and PMH and min-max envelope for 10 independent runs ($100\%$ prediction interval). The horizontal blue shaded region is the 90\% prediction interval of the energy distance between two true samples.}
    \label{fig:cappeexample_iter}
\end{figure}

A first observation is that, except for MoKA KIDS, the different methods typically show a very small variance in the results. Thanks to the propagation of chaos result and with a very large number of particles, the system is indeed expected to evolve according to the deterministic recurrence relation \eqref{eq:recmu}.  
It should also be noted that in all cases, the MoKA algorithms (except MoKA KIDS in the difficult case) need less iterations to converge. This shows that choosing adaptive kernels indeed increases the efficiency per iteration. Of course, it comes at the price of an increased computation cost. In this specific example and considering the results presented in Section \ref{sec:cappe_results_GPUtime}, this increased efficiency does not appear necessary but in the following Section \ref{sec:manygaussians},  we will investigate a situation where the MoKA algorithms perform better than CMC also in terms of computation time. Still regarding the efficiency, the collective methods including CMC are clearly better than PMH (in the simpler example, PMH needs about 2000 iterations to converge). This can be seen by comparing the acceptance rate in Figure \ref{fig:cappeexample_acceptweights}. In particular, all collective methods have an acceptance rate above 0.4 while the acceptance rate for PMH drops to almost zero in the difficult case. 

In Figure \ref{fig:cappeexample_acceptweights}, we can observe that the choice of the kernel weights is different in MoKA Markov compared to MoKA and MoKA KIDS. In the former case, the algorithm first chooses almost exclusively a large exploration proposal. Then, as the particles get closer to the target distribution, the proposal kernel becomes smaller. For the other methods, a nontrivial mixture of small and medium kernels is kept during the whole simulation. This behaviour is similar to the one observed on the banana-shaped target distribution in Section \ref{sec:banana}.

Finally, the convergence plot (Figure \ref{fig:cappeexample_iter}) of the collective methods are roughly piecewise linear indicating an exponential convergence. More precisely, for CMC, MoKA and MoKA Markov we can identify two straight lines and a plateau in the difficult case. The first phase (up to iteration 30$\sim$40) corresponds to the colonization of the first mode. The following plateauing behaviour in the difficult case corresponds to a phase where the particles are almost exclusively distributed in the first mode. But as more particles enter the second mode, the Energy Distance to the target distribution decreases and a last exponentially fast converging phase begins. The duration of this plateauing exploration phase can vary, this is why a larger variance is observed for the last phase in the difficult example \eqref{eq:cappe_difficult} compared to the simpler case \eqref{eq:cappe_simple}. MoKA KIDS is less stable but for the simple example, it sometimes needs three times less iterations to converge compared to the other methods. Moreover, unlike the other methods, on its convergence plot (Figure~\ref{fig:cappeexample_iter}), we can observe several plateauing phases interspersed with sudden drops. These drops corresponds to the times where a well-located particle gains most of the weight and suddenly attracts many other particles around itself. Similarly to the importance sampling based methods, this behaviour can be very efficient but also dangerous as it may impair the exploration of the space. This leads to a clear failure in the difficult example.

\begin{figure}[ht]
    \centering
    \subfloat[Acceptance]{\includegraphics[width=0.24\textwidth]{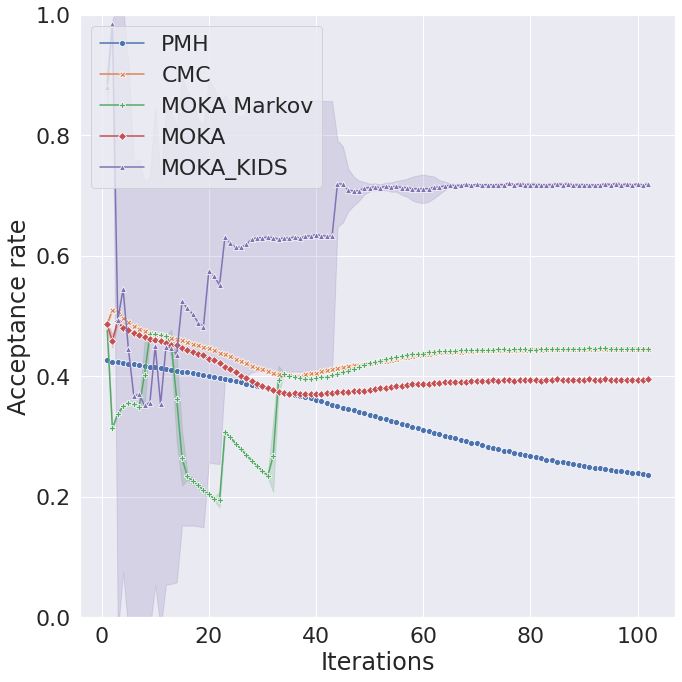}}
    \subfloat[MoKA Markov]{\includegraphics[width=0.24\textwidth]{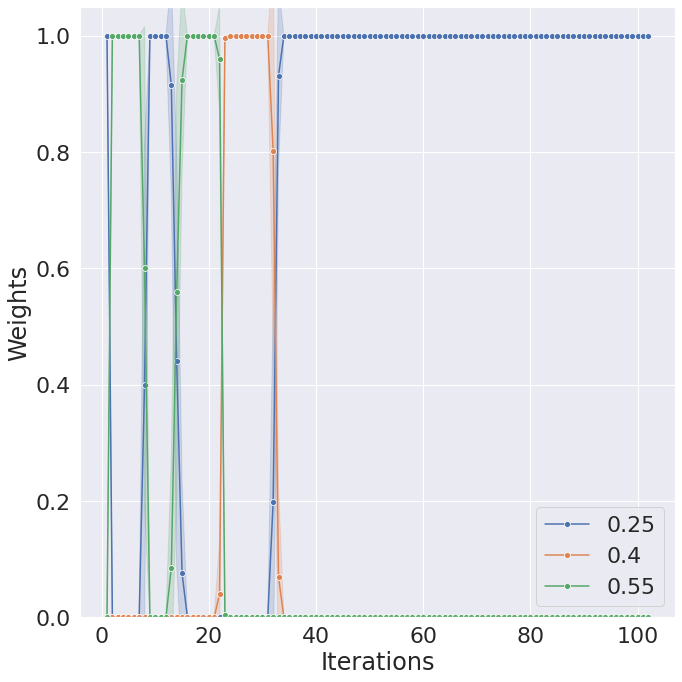}}
    \subfloat[MoKA]{\includegraphics[width=0.24\textwidth]{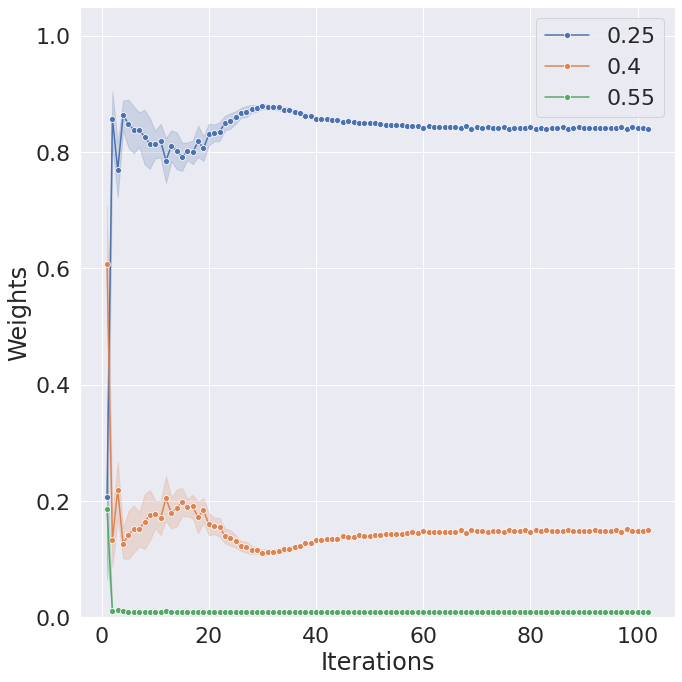}}
    \subfloat[MoKA KIDS]{\includegraphics[width=0.24\textwidth]{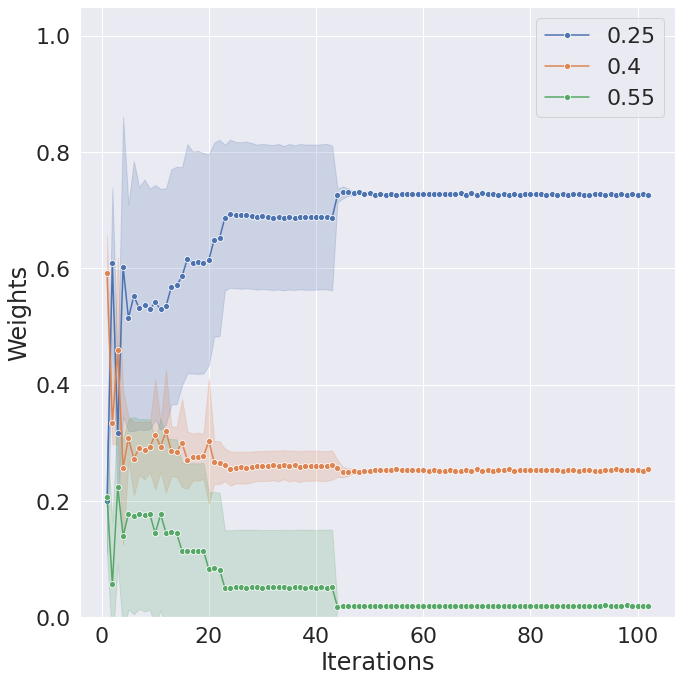}} \\ 
    \subfloat[Acceptance]{\includegraphics[width=0.24\textwidth]{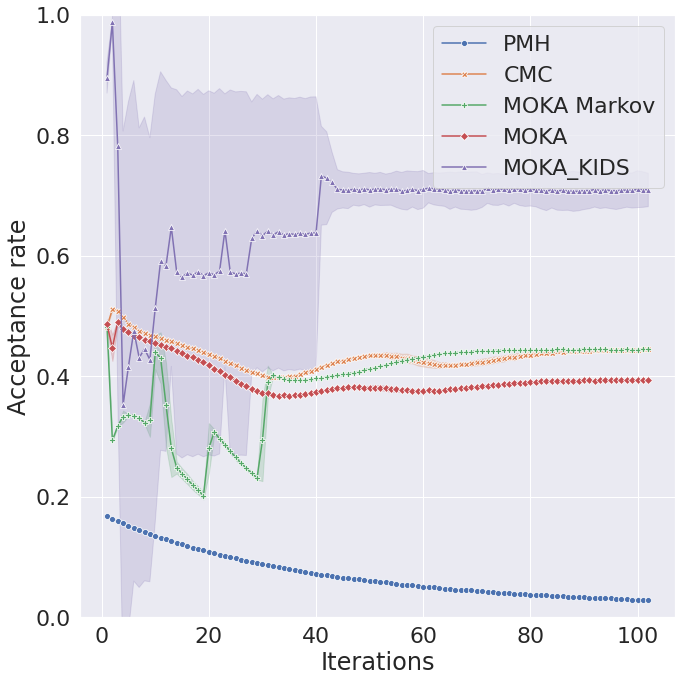}}
    \subfloat[MoKA Markov]{\includegraphics[width=0.24\textwidth]{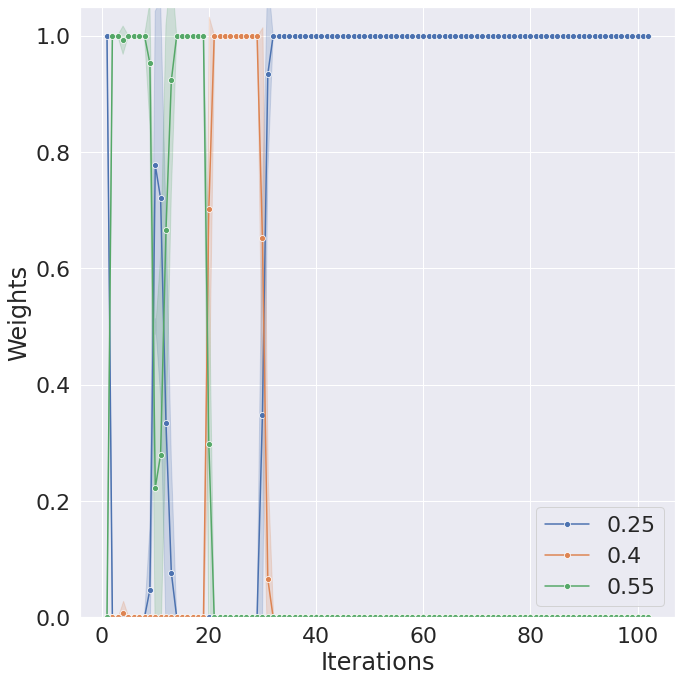}}
    \subfloat[MoKA]{\includegraphics[width=0.24\textwidth]{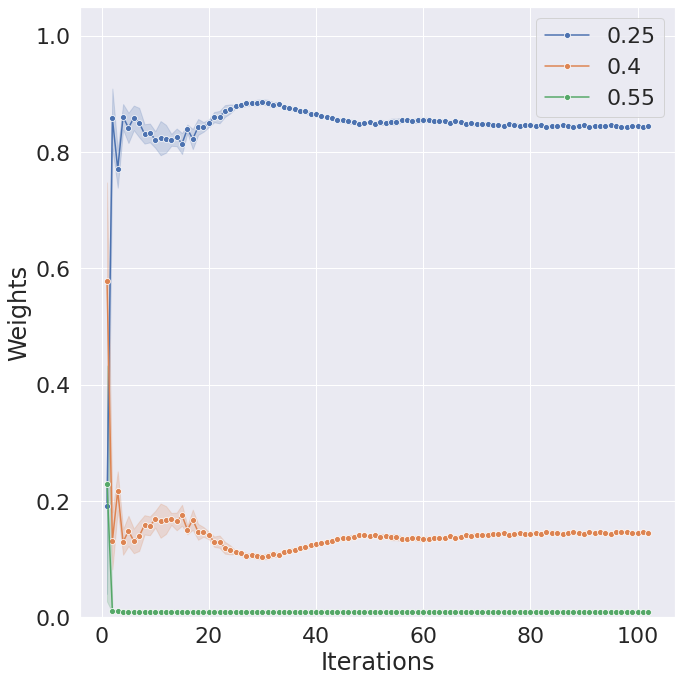}}
    \subfloat[MoKA KIDS]{\includegraphics[width=0.24\textwidth]{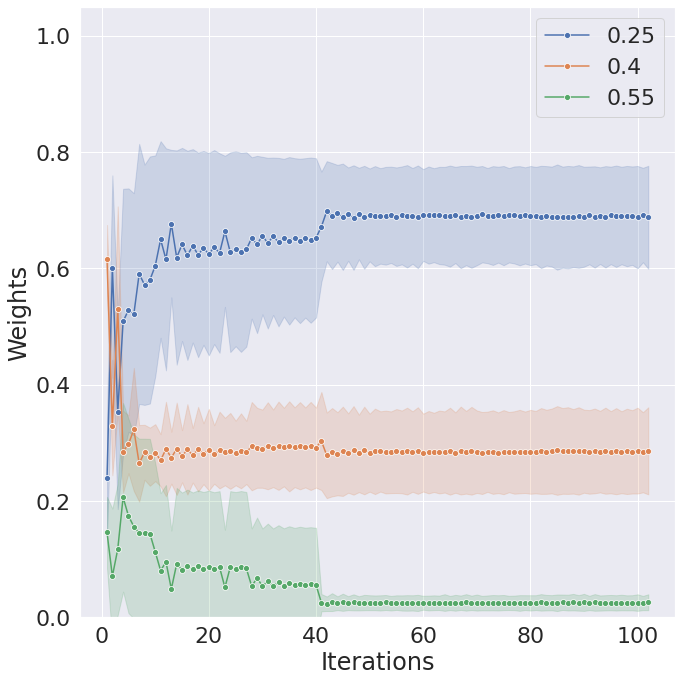}}
    \caption{Mean acceptance rate and weights of the kernels over 100 iterations and standard deviations over 10 independent runs for the simple example ((a) to (d)) and for the difficult example ((e) to (h)).}
    \label{fig:cappeexample_acceptweights}
\end{figure}

\subsection{Tuning the hyper-parameters}\label{sec:tuningCMC}

As any of the other sampling methods discussed, the collective algorithms do require some appropriate tuning of the hyper-parameters. The simplest Algorithm \ref{algo:coker} has two hyper-parameters, the number of particles and the size $R$ of the interaction kernel $K$, which we take equal to the indicator of the ball centered at zero with radius $R$. As for the number of particles, it should simply be taken as large as possible, up to the computational capability. Note that with our implementation, one can easily handle samples of size $N\sim 10^5$ with a computation time of just a few seconds in the previous examples. The size of the interaction kernel is the most important parameter to tune. Although the convergence of the associated nonlinear Markov chain is ensured regardless of the size of $K$, since the algorithm uses a mean-field approximation with a finite number of particles, it is important to make sure that the particles interact in a mean-field regime. A simple criteria is simply to count the average number of neighbouring particles within the interaction radius around the proposals $(Y^i_t)^{}_i$. This can be computed at no additional cost as it is simply a multiple of $K\star\hat{\mu}^N_t(Y^i_t)$. As a rule of thumb, particles need to interact with at least a few tens of neighbours. In our observations, the better results are obtained with $100\sim200$ neighbours in average and is not improved with more neighbours. In particular increasing the interaction kernel to a very large size will naturally increase the number of neighbours but it may also deteriorate the performances as the acceptance rate will drop (which means that more iterations will be needed to converge). Both the number of neighbours and the acceptance rate can be easily monitored to ensure good convergence properties. This is illustrated in Figure \ref{fig:severalsizesK} on the previous 12-dimensional simple example \eqref{eq:cappe_simple}. Note that this only parameter to tune is in contrast with the other methods, in particular SAIS and SMC where the performances can also be strongly affected by the tempering scheme, the mixing weights sequence, the choice of the threshold value of the ESS, the number of mutation steps, the choice of the mutation kernel etc. All these hyper-parameters are subtle to tune in practice. Finally we also mention that although we did not explore this direction and leave it for future work, one may also consider adaptive versions of the Algorithm \ref{algo:coker} (and its variants) where the size of the kernel is chosen in order to ensure a good balance between a good acceptance rate and a reasonable average number of neighbours. However, the convergence properties would likely be theoretically difficult to analyse as the number of neighbours is, by definition, not a mean-field quantity and the limit $N\to+\infty$ would therefore not make sense. 

\begin{figure}[ht]
    \centering
    \includegraphics[scale=0.3]{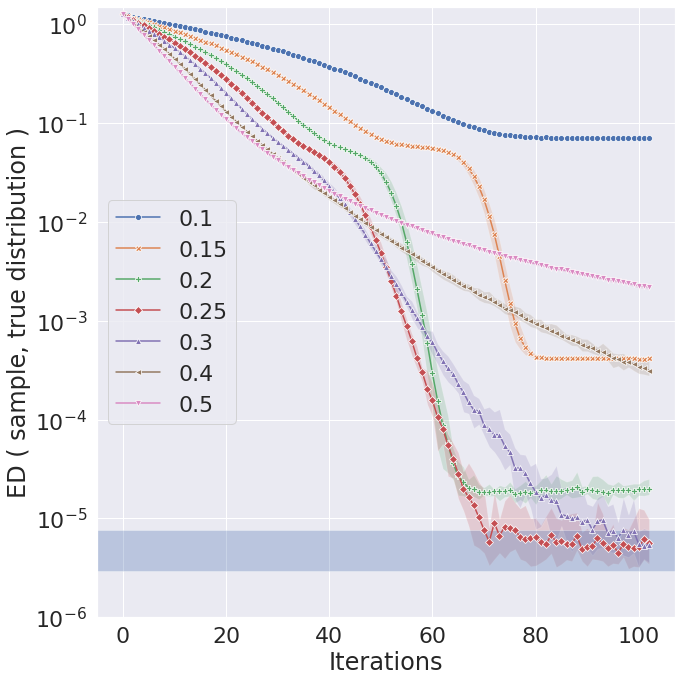}
    \caption{CMC for several sizes of the interaction kernel. Mean energy distance over 10 independent runs and min-max envelope}
    \label{fig:severalsizesK}
\end{figure}

\begin{table}[ht]
    \centering
    \begin{tabular}{cccccccc}
         Size of the interaction kernel &  0.1 & 0.15 & 0.2 & 0.25 & 0.3 & 0.4 & 0.5\\
         Mean error & 7e-02 & 4.1e-04 & 1.9e-05 & \textbf{5.4e-06} & 7.3e-06 & 4e-04 & 2.5e-03\\
         Mean number of neighbors & 2 & 6 & 36 & 194 & 637 & 2137 & 3460 \\
         Mean acceptance rate & 0.72 & 0.63 & 0.57 & 0.44 & 0.30 & 0.11 & 0.05
    \end{tabular}
    \caption{Mean error, mean number of neighbors and mean acceptance rate computed over the ten last iterations of ten independent runs for each size of the interaction kernel.}
\end{table}

Finally, let us note that it is always safe to choose a large interaction kernel: in the maximal case where the kernel embraces all the domain, then the Vanilla CMC Algorithm \ref{algo:coker} simply reduces to a standard random-walk MH algorithm with a proposal uniformly sampled in the domain. Although it may converge too slowly to be of practical use (it keeps a quadratic complexity in $N$), it is theoretically ensured to converge. On the other hand choosing a too small interaction kernel may lead to bad behaviours. In particular, we have observed that when a particle interact with only a few neighbours, approximately less than $10\sim20$, then the system may locally over-concentrate (i.e. there are localized high-density regions) and fail to converge towards the target distribution. This can be seen on Figure \ref{fig:severalsizesK} where the smallest kernels lead to a plateauing behaviour far from the minimal error. This bad behaviour can nevertheless be easily monitored: in practice and in particular for the adaptive MoKA samplers, it is recommended to pay attention to the average number of neighbours and to either eliminate the interaction kernels which lead to a too small number of neighbours or to increase the number of particles. Since the volume of a ball of radius $r$ behaves as $r^d$ in dimension $d$, an unreasonably large number of particles or size of interaction kernel would be needed in very high-dimensional setting. This is why our collective algorithms cannot be used when $d$ is too large. In our examples, we have observed good results up to dimension 13 (with $10^5$ particles).

\subsection{More complex example}\label{sec:manygaussians}

\subsubsection{Target distribution}

In the classical test case introduced in \cite{cappe2008adaptive,delyon2019adaptive} and considered in Section \ref{sec:cappe}, the target distribution is made of only two equally weighted Gaussians regardless of the dimension. Even for moderately large dimensions (around $d=12$), good results can be obtained even with the simplest Metropolis-Hastings algorithm (with a very long computation time), although the symmetry of the target and the initialisation can be misleading. In this section we revisit this example and introduce the following more complex test case in dimension $d$ : 
\begin{equation}\label{eq:target_many_gaussians}\pi = \sum_{i=1}^d w_i \mathcal{N}(m+\mu_i,\sigma^2_d I_d) + \sum_{j=1}^d \overline{w}_j \mathcal{N}(m - \mu_j,\sigma^2_d I_d),\end{equation}
where $m = (0.5,\ldots,0.5)^\mathrm{T}$ is the center of the domain, $\mu_i = \frac{a}{2}e_i$ with $a>0$ and $e_i$ is the $i$-th vector of the canonical basis, $\sigma_d>0$ is a fixed standard deviation and $w_i,\overline{w}_i\in[0,1]$ are normalized weights. This distribution is made of $2d$ Gaussian distributions, 2 per dimension, and as in \cite{cappe2008adaptive, delyon2019adaptive}, the distance between the centers does not depend on the dimension (it is equal to $a$ between the two distributions on the same dimension and to $a/\sqrt{2}$ between two distributions on two different dimensions). We choose the following parameters: 
\[a = 0.7,\quad \sigma_d = \sqrt{\frac{0.03}{4d}},  \quad w_i = 0.25/d, \quad \overline{w_i} = 0.75/d.\]
With this choice, the $2d$ peaks are well separated and their variance decreases with the dimension. In addition half of the peaks carries the three quarters of the total weight of the distribution. This test case thus combines most of the difficulties: this is a multimodal distribution, there is a strong exploration-exploitation dilemma due to the size and the isolation of the peaks and all the modes do not have the same weight.

\subsubsection{Initialization} 

Our samplers are initialized with all the particles in a corner \eqref{eq:cornerinitial}, thus far away from all the peaks and closer to the modes with a lighter weight. For the competitors, this initialization leads to bad results and often a degeneracy of the importance weights. The behaviour of PMH does not really depend on the initialization since the large proposal distribution allows move all across the domain. Thus for PMH, SAIS and SMC the particles are initially sampled uniformly in the whole domain. This gives more chance to the importance sampling based methods since there are particles close to all the modes initially. Moreover, SMC may be quite affected by a bad initialisation due to the tempering scheme which forces the particles to stay close to their initial state during the first iterations. We have also checked that, for our samplers, the initialization \eqref{eq:cornerinitial} or a uniform initialization in the whole domain both lead to the same behaviour (with a faster convergence in the latter case).  

\subsubsection{Parameters} 

In order to tune the the proposal distributions of our samplers, we introduce three different natural scales: 
\[\sigma_\mathrm{small} = 3.5\sigma_d, \quad\sigma_\mathrm{medium} = 10\sigma_d, \quad \sigma_\mathrm{large} = 20\sigma_d.\]
The small scale corresponds to roughly speaking, the size of a peak but it is very small compared to the distance between the peaks. The medium scale is larger than the typical size of a peak and allows moves between neighbouring peaks. The large scale distribution allows moves between far away peaks. The interaction kernel $K$ in the Vanilla CMC Algorithm \ref{algo:coker} is taken equal to the indicator of a ball of radius $\sigma_\mathrm{medium}$. The different variants of the adaptive MoKA algorithm use the three sizes. For SAIS and SMC, we have tested all the kernel sizes between 0.05 and 0.8 with a step of 0.05 and we took the sizes which led to the best results. For PMH the best outcomes are obtained with a large kernel size although the different sizes essentially lead to equally bad results. 

Regarding the other parameters, the ESS threshold in SMC is set to $0.75N$, there are 150 MH steps between two resampling steps and a linear tempering scheme with 80 steps is used from the initial distribution to the target distribution. For SAIS, the sequence of kernel bandwiths and mixture weights are taken according to the recommendations of the authors (see Appendix \ref{sec:linksis}). 

In all samplers, the number of particles is $N=10^5$ and the outcomes are shown for a fixed allowed GPU computation time of 180 seconds. 

\subsubsection{Results}\label{sec:results_many_gaussians}

The results shown in Table \ref{tab:results_many_gaussians} are for dimensions ranging from $d=4$ to $d=13$ and computed over 50 independent runs for all samplers and each dimension. For better clarity, we have classified the results in four possible outcomes, similarly to \cite{cappe2008adaptive}, and defined as follows. 

\begin{itemize}
    \item \textbf{Disastrous.} The median Energy Distance between the final sample and a true sample is larger than $E_0/10$ where $E_0$ is the initial Energy Distance between a uniform sample in the domain and a true sample. 
    \item \textbf{Excellent.} The median Energy Distance between the final sample and a true sample falls below the upper limit of the $90\%$ prediction interval between two true samples. 
    \item \textbf{Good.} The median Energy Distance between the final sample and a true sample falls between the upper limit of the $90\%$ prediction interval and the average in log scale between this upper limit and $E_0/10$. 
    \item \textbf{Mediocre.} The median Energy Distance between the final sample and a true sample falls between $E_0/10$ and the average in log scale between the upper limit of the $90\%$ prediction interval and $E_0/10$. 
\end{itemize}

The convergence plots of the best runs in dimensions 7 and 12 are shown in Figure \ref{fig:bestrunsmanygaussians}.

\begin{figure}[ht]
    \centering
    \subfloat[$d=7$]{\includegraphics[scale=0.22]{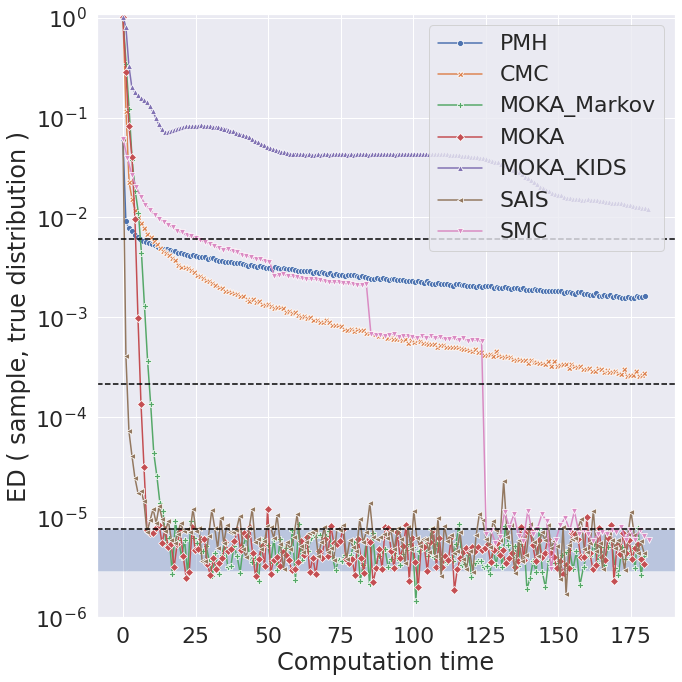}}
    \subfloat[$d=12$]{\includegraphics[scale=0.22]{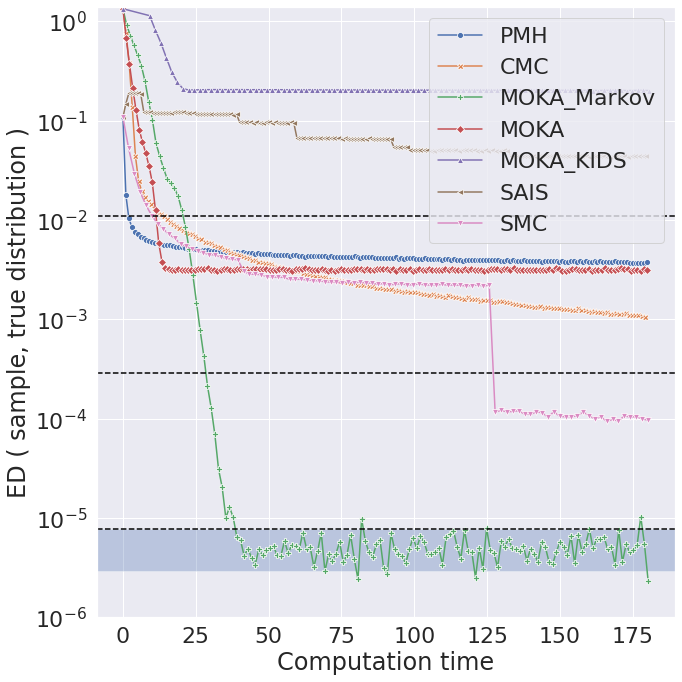}}
    \caption{The best runs among 50 independent runs for the different algorithms in dimension $d=7$ and $d=12$. The three black dashed lines divide the $y$-axis into four regions for the four possible outcomes, from top to bottom: disastrous, mediocre, good, excellent. The best run is the one which leads to the smallest final energy distance or, if several runs of the same algorithm lead to an excellent outcome, the best run is the fastest one. The blue shaded region is the $90\%$ prediction interval of the energy distance between two true samples.}
    \label{fig:bestrunsmanygaussians}
\end{figure}

\begin{table}[!ht]
    \centering
\begin{tabular}{|c||>{\centering\arraybackslash}p{1.15cm}|>{\centering\arraybackslash}p{1.15cm}|>{\centering\arraybackslash}p{1.15cm}|>{\centering\arraybackslash}p{1.15cm}|>{\centering\arraybackslash}p{1.15cm}|>{\centering\arraybackslash}p{1.2cm}|>{\centering\arraybackslash}p{1.15cm}|}
\hline
{} &                                                PMH &                                                CMC &                                       MoKA Markov &                                               MoKA &                                         MoKA KIDS &                                               SAIS &                                               SMC \\
\hline\hline
$d=4$  &    \cellcolor{green} E 4.4E-06 {\tiny (98,2,0,0) } &   \cellcolor{green} E 4.4E-06 {\tiny (100,0,0,0) } &  \cellcolor{green} E 4.6E-06 {\tiny (100,0,0,0) } &   \cellcolor{green} E 4.7E-06 {\tiny (100,0,0,0) } &  \cellcolor{green} E 4.4E-06 {\tiny (100,0,0,0) } &   \cellcolor{green} E 4.8E-06 {\tiny (100,0,0,0) } &  \cellcolor{green} E 4.5E-06 {\tiny (100,0,0,0) } \\\hline
$d=5$  &    \cellcolor{lime} G 3.3E-05 {\tiny (0,100,0,0) } &    \cellcolor{lime} G 2.0E-05 {\tiny (0,100,0,0) } &  \cellcolor{green} E 4.6E-06 {\tiny (100,0,0,0) } &   \cellcolor{green} E 4.5E-06 {\tiny (100,0,0,0) } &   \cellcolor{green} E 4.9E-06 {\tiny (96,2,2,0) } &   \cellcolor{green} E 5.0E-06 {\tiny (100,0,0,0) } &   \cellcolor{green} E 5.3E-06 {\tiny (94,6,0,0) } \\\hline
$d=6$  &  \cellcolor{yellow} M 5.9E-04 {\tiny (0,0,100,0) } &    \cellcolor{lime} G 1.1E-04 {\tiny (0,100,0,0) } &  \cellcolor{green} E 4.7E-06 {\tiny (100,0,0,0) } &   \cellcolor{green} E 4.7E-06 {\tiny (100,0,0,0) } &  \cellcolor{lime} G 4.1E-05 {\tiny (16,46,32,6) } &    \cellcolor{green} E 5.4E-06 {\tiny (96,4,0,0) } &   \cellcolor{lime} G 1.0E-05 {\tiny (24,70,6,0) } \\\hline
$d=7$  &  \cellcolor{yellow} M 1.8E-03 {\tiny (0,0,100,0) } &   \cellcolor{yellow} M 3.1E-04 {\tiny (0,4,96,0) } &  \cellcolor{green} E 4.9E-06 {\tiny (100,0,0,0) } &   \cellcolor{green} E 4.8E-06 {\tiny (100,0,0,0) } &    \cellcolor{red} D 3.7E-02 {\tiny (0,0,0,100) } &   \cellcolor{green} E 6.5E-06 {\tiny (70,30,0,0) } &   \cellcolor{lime} G 1.5E-05 {\tiny (14,80,2,4) } \\\hline
$d=8$  &  \cellcolor{yellow} M 2.5E-03 {\tiny (0,0,100,0) } &  \cellcolor{yellow} M 5.1E-04 {\tiny (0,0,100,0) } &  \cellcolor{green} E 4.8E-06 {\tiny (100,0,0,0) } &   \cellcolor{green} E 5.0E-06 {\tiny (100,0,0,0) } &    \cellcolor{red} D 9.4E-02 {\tiny (0,0,0,100) } &  \cellcolor{yellow} M 1.0E-03 {\tiny (0,12,82,6) } &    \cellcolor{lime} G 6.5E-05 {\tiny (0,96,2,2) } \\\hline
$d=9$  &  \cellcolor{yellow} M 3.2E-03 {\tiny (0,0,100,0) } &  \cellcolor{yellow} M 7.0E-04 {\tiny (0,0,100,0) } &  \cellcolor{green} E 4.8E-06 {\tiny (100,0,0,0) } &   \cellcolor{green} E 4.6E-06 {\tiny (100,0,0,0) } &    \cellcolor{red} D 1.9E-01 {\tiny (0,0,0,100) } &     \cellcolor{red} D 1.6E-02 {\tiny (0,0,20,80) } &    \cellcolor{lime} G 2.8E-05 {\tiny (0,92,2,6) } \\\hline
$d=10$ &  \cellcolor{yellow} M 3.5E-03 {\tiny (0,0,100,0) } &  \cellcolor{yellow} M 9.2E-04 {\tiny (0,0,100,0) } &  \cellcolor{green} E 4.9E-06 {\tiny (100,0,0,0) } &   \cellcolor{green} E 5.9E-06 {\tiny (54,0,46,0) } &    \cellcolor{red} D 2.1E-01 {\tiny (0,0,0,100) } &     \cellcolor{red} D 8.8E-02 {\tiny (0,0,0,100) } &   \cellcolor{lime} G 5.4E-05 {\tiny (0,100,0,0) } \\\hline
$d=11$ &  \cellcolor{yellow} M 3.7E-03 {\tiny (0,0,100,0) } &  \cellcolor{yellow} M 1.1E-03 {\tiny (0,0,100,0) } &  \cellcolor{green} E 5.1E-06 {\tiny (100,0,0,0) } &   \cellcolor{yellow} M 4.7E-03 {\tiny (2,0,90,8) } &    \cellcolor{red} D 2.6E-01 {\tiny (0,0,0,100) } &     \cellcolor{red} D 1.1E-01 {\tiny (0,0,0,100) } &    \cellcolor{lime} G 6.8E-05 {\tiny (0,96,0,4) } \\\hline
$d=12$ &  \cellcolor{yellow} M 3.8E-03 {\tiny (0,0,100,0) } &  \cellcolor{yellow} M 1.3E-03 {\tiny (0,0,100,0) } &  \cellcolor{green} E 5.1E-06 {\tiny (100,0,0,0) } &  \cellcolor{yellow} M 1.0E-02 {\tiny (0,0,64,36) } &    \cellcolor{red} D 2.9E-01 {\tiny (0,0,0,100) } &     \cellcolor{red} D 1.5E-01 {\tiny (0,0,0,100) } &   \cellcolor{lime} G 1.2E-04 {\tiny (0,100,0,0) } \\\hline
$d=13$ &  \cellcolor{yellow} M 3.9E-03 {\tiny (0,0,100,0) } &  \cellcolor{yellow} M 1.5E-03 {\tiny (0,0,100,0) } &   \cellcolor{green} E 5.3E-06 {\tiny (98,0,2,0) } &     \cellcolor{red} D 1.3E-02 {\tiny (0,0,40,60) } &    \cellcolor{red} D 5.0E-01 {\tiny (0,0,0,100) } &     \cellcolor{red} D 2.1E-01 {\tiny (0,0,0,100) } &    \cellcolor{lime} G 1.1E-04 {\tiny (0,98,2,0) } \\\hline
\end{tabular}

    \caption{Detailed outcomes of the different algorithms for the target distribution \eqref{eq:target_many_gaussians} with the dimension ranging from 4 to 13. The number in each cell corresponds to the median computed over 50 independent runs of the Energy Distance between the final sample and a true sample. The outcome is classified as Excellent (E) (green), Good (G) (light green), Mediocre (M) (yellow) or Disastrous (D) (red) as defined in Section \ref{sec:results_many_gaussians}. The four percentages at the bottom of each cell indicates, from left to right, the percentages of Excellent, Good, Mediocre and Disastrous results for 50 independent runs.}
    \label{tab:results_many_gaussians}
\end{table}

\subsubsection{Discussion}

For small dimensions up to dimension 6, all samplers (including PMH) have a similar behaviour and manage to reach the smallest possible distance or at least have a good outcome. From dimension 6 for PMH, dimension 7 for CMC and MoKA KIDS and dimension 9 for SAIS, the outcome is not satisfactory anymore and even become disastrous for MoKA KIDS and SAIS. The only algorithms which have an excellent outcome up to dimension 10 are MoKA Markov and MoKA. The former remains an excellent sampler up to dimension~13. SMC also manages to have a good outcome up to dimension~13 although it sometimes, rarely, leads to mediocre or disastrous results. It is probable that with a better tempering scheme or an adaptive version, SMC would be able to reach the smallest possible distance although this may require a fine tuning of these hyper parameters. In comparison, MoKA Markov seems simpler to tune: the choice of the kernel sizes that we propose consistently leads to excellent results up to dimension~13. Without this adaptive behaviour, the fixed kernel size in CMC is initially well-adapted to exploration but it is too large compared to the size of the modes which causes a very slow convergence and leads to mediocre results (the final median error remains significantly better that PMH for the same computation time though). The adaptive procedure in MoKA is based on maximizing the acceptance and seems less efficient that the optimization procedure in MoKA Markov. It can be explained by the fact that the non-Markovian choice of the kernels tends to favour the smallest interaction kernel. Although this kernel is also ultimately chosen by MoKA Markov, choosing this kernel ``too early'' does not introduce enough exploration. Note that with a uniform initialisation over the whole domain, we have observed (not shown) that this flaw disappears and MoKA and MoKA Markov have a similar behaviour. MoKA KIDS and SAIS both lead to disastrous results when the dimension increases. It is due to the early degeneracy of the weights of a small number of particles which prevents a proper mixing of the population.


\section{Conclusion}

Nonlinear MCMC samplers are appealing. They generalise more traditional methods and overcome many of their flaws. Getting back to the historical development of mathematical kinetic theory, we can advantageously simulate nonlinear Markov processes using systems of interacting particles. This versatility enables the development of a wide variety of algorithms that can tackle difficult sampling problems while remaining in a traditional Markovian framework. Although the implementation may, at first sight, seem computationally demanding, we have shown that modern GPU hardware can now enable the use of interacting particles for Monte Carlo sampling at scale.

Alongside its variants, the CMC algorithm can be implemented efficiently and leads to striking reductions in global convergence times. It relies on pairwise interactions to best leverage the information that is present in any given sample swarm, and thus make the most of each evaluation of the target distribution. CMC avoids the mixing issues of classical ``one particle'' methods such as Metropolis-Hastings, with a notable improvement of the convergence and mixing speed. In particular when dealing with multimodal distributions, where the relative weight of each mode is difficult to estimate. In practice, we thus expect that the benefits of this improved ``sample efficiency'' will outweigh the (small) computational overhead of our method for most applications.

We note that the present contribution shares similarities with some well-known and recent nonlinear samplers, that are often based on non-Markovian importance sampling techniques. In the future, the joint development of Markovian and non-Markovian methods is likely to benefit both approaches: we may for instance improve the importance weights in SAIS-like methods as in the KIDS algorithm, or construct better proposal distributions in CMC which incorporate knowledge of (part of) the past. The theoretical study of such hybrid methods would however be challenging and require the development of new analytical tools.

Finally, one may think of extending the theoretical framework introduced here to other MCMC samplers, such as \emph{nonlinear} PDMP samplers or \emph{nonlinear} Langevin dynamics. This could open new problems in nonlinear analysis and statistics, both on the theoretical and computational sides.

\section*{Acknowledgments}

The authors wish to thank Pierre Degond, Robin Ryder and Christian Robert for their support and useful advice. A.D. acknowledges the hospitality of the CEREMADE, Universit\'e Paris-Dauphine where part of this research was carried out. This research was conducted while A.D. was supported by an EPSRC-Roth scholarship co-funded by the Engineering and Physical Sciences Research Council and the Department of Mathematics at Imperial College London. G.C. acknowledges the hospitality of the Mathematics Department at Imperial College London where part of this research was carried out.

\bibliographystyle{imsart-number}

\bibliography{biblioCMCv2}

\newpage

\begin{appendix}

\stitle{Supplementary proofs and numerical results}
\sdescription{Proof of the propagation of chaos, an application of the results to the classical Metropolis-Hastings case, numerical experiments on another target and normalizing constant estimation.}

\section{Proof of Theorem \ref{thm:pathwiseestimate}}\label{sec:proofpoc}

Let us start with the following lemma. 

\begin{lemma}\label{lemma:lipschitzratio}
Let $\Theta$ be a proposal distribution which satisfies Assumptions \ref{assum:bound}, \ref{assum:lipschitz} and \ref{assum:nonexpansive}. Then the map 
\[\alpha:\pb(E)\times E^2\to\R,\,\,(\mu,x,y)\mapsto \alpha_\mu(x,y):=\frac{\Theta_\mu(x|y)\pi(y)}{\Theta_\mu(y|x)\pi(x)},\]
is Lipschitz in the Wasserstein-1 distance, in the sense that there exists a constant $L_{\Theta}>0$ (which depends also on $\pi$) such that for all $\mu,\nu\in\pb(E)$ and $x,x',y,y'\in E$, it holds that: 
\[|\alpha_\mu(x,y)-\alpha_\nu(x',y')|\leq L_\Theta\big(W^1(\mu,\nu)+|x-x'|+|y-y'|\big).\]
\end{lemma}

\begin{proof}
By the triangle inequality, it holds that:
\begin{equation*}
    |\alpha_\mu(x,y)-\alpha_\nu(x',y')| \leq\frac{\pi(y)}{\pi(x)}\left|\frac{\Theta_\mu(x|y)}{\Theta_\mu(y|x)}-\frac{\Theta_\nu(x'|y')}{\Theta_\nu(y'|x')}\right|+\frac{\Theta_\nu(x'|y')}{\Theta_\nu(y'|x')}\left|\frac{\pi(y)}{\pi(x)}-\frac{\pi(y')}{\pi(x')}\right|.
\end{equation*}

We bound each of the two terms on the right-hand side: 
\begin{align*}
    \frac{\pi(y)}{\pi(x)}\left|\frac{\Theta_\mu(x|y)}{\Theta_\mu(y|x)}-\frac{\Theta_\nu(x'|y')}{\Theta_\nu(y'|x')}\right|&\leq\frac{\pi(y)}{\pi(x)\Theta_\mu(y|x)}|\Theta_\mu(x|y)-\Theta_\nu(x'|y')|\\
    &\qquad+\frac{\pi(y)\Theta_\nu(x'|y')}{\pi(x)\Theta_\mu(x|y)\Theta_\mu(x'|y')}|\Theta_\mu(y|x)-\Theta_\nu(y'|x')|\\
    &\leq \frac{LM_0}{m_0\kappa_-}\left(1+\frac{\kappa_+}{\kappa_-}\right)\Big(W^1(\mu,\nu)+|x-x'|+|y-y'|\Big).
\end{align*}
and
\begin{align*}
    \frac{\Theta_\nu(x'|y')}{\Theta_\nu(y'|x')}\left|\frac{\pi(y)}{\pi(x)}-\frac{\pi(y')}{\pi(x')}\right|&\leq \frac{\Theta_\nu(x'|y')}{\pi(x)\Theta_\nu(y'|x')}|\pi(y)-\pi(y')|\\
    &\qquad\qquad+\frac{\Theta_\nu(x'|y')\pi(y')}{\Theta_\nu(y'|x')\pi(x)\pi(x')}|\pi(x)-\pi(x')|\\
    &\leq \frac{\|\pi\|_\mathrm{Lip}\kappa_+}{ m_0\kappa_-}\left(|y-y'|+\frac{M_0}{m_0}|x-x'|\right),
\end{align*}
where $\|\pi\|_\mathrm{Lip}$ denotes the Lipschitz norm of $\pi$. Gathering everything gives the result with 
\[L_\Theta = \frac{M_0}{m_0}\left(\frac{L}{\kappa_-}\left(1+\frac{\kappa_+}{\kappa_-}\right)+\frac{\|\pi\|_\mathrm{Lip}\kappa_+}{ m_0\kappa_-}\right).\]

\end{proof}

\begin{proof}[Proof (of Theorem \ref{thm:pathwiseestimate})]

The strategy of the proof of Theorem \ref{thm:poc} will be based on coupling arguments inspired by \cite{sznitman1991topics} and adapted from \cite{diez2019propagation}. We start by the following trajectorial representation of the nonlinear Markov chain $(\overline{X}_t)_t$ defined by the transition kernel~\eqref{eq:cmckernel}. 

\begin{definition}[Nonlinear process]\label{def:nonlinearprocess}
Let $\overline{X}_0\sim \mu_0$ be an initial state where $\mu_0\in\mathcal{P}(E)$. The state $\overline{X}_t$ at time $t\in\N$, $t\geq1$, is constructed from $\overline{X}_{t-1}$ and the law of $\overline{X}_{t-1}$ denoted by $\mu_{t-1}\in\mathcal{P}(E)$ as follows. 
\begin{enumerate}
\item Take a proposal a random variable $\overline{Y}_t\sim \Theta_{\mu_{t-1}}(\cdot | \overline{X}_{t-1})$ 
\item Compute the ratio 
\[\alpha_{\mu_{t-1}}(\overline{X}_{t-1},\overline{Y}_{t}):=\frac{\Theta_{\mu_{t-1}}(\overline{X}_{t-1}|\overline{Y}_{t})\pi(\overline{Y}_{t})}{\Theta_{\mu_{t-1}}(\overline{Y}_{t}|\overline{X}_{t-1})\pi(\overline{X}_{t-1})}.\]
\item Take $\overline{U}_{t}\sim \mathcal{U}([0,1])$ and if $\overline{U}_t\leq h\big({\alpha}_{\mu_{t-1}}(\overline{X}_{t-1},\overline{Y}_{t})\big)$, then accept the proposal, else reject it:
\[\overline{X}_t = \overline{X}_{t-1}\1_{\overline{U}_t\geq h({\alpha}_{\mu_{t-1}}(\overline{X}_{t-1},\overline{Y}_{t}))}+\overline{Y}_t\1_{\overline{U}_t\leq h({\alpha}_{\mu_{t-1}}(\overline{X}_{t-1},\overline{Y}_{t}))}.\]
\end{enumerate}
\end{definition}

From now on we consider $N$ independent copies $(\overline{X}^i_t)^{}_t$, $i\in\{1,\ldots,N\}$, of the nonlinear process defined by Definition \ref{def:nonlinearprocess}. We then construct a coupled particle process $(X^i_t)^{}_t$ such that for all $i\in\{1,\ldots,N\}$, initially $X^i_0=\overline{X}^i_0\sim \mu_0$ and for each time $t\in\N$ we take:
\begin{enumerate}
\item the same jump decision random variables $U^i_t=\overline{U}^i_t\sim\mathcal{U}([0,1])$,
\item optimal proposals of the form $Y^i_t=s(\overline{Y}^i_t)$ where $s$ is an optimal transport map between $\Theta_{\mu_{t-1}}(\cdot|\overline{X}^i_{t-1})$ and $\Theta_{\hat{\mu}^N_{t-1}}(\cdot|{X}^i_{t-1})$. Since these two probability measures are absolutely continuous with respect to the Lebesgue measure, the existence of such optimal transport map (Monge problem) is proved for instance in \cite{champion2011monge} or \cite{caffarelli2002constructing}. By definition, the pathwise error between the proposals can thus be controlled by 
\begin{align*}\E\big[|Y^i_t-\overline{Y}^i_t| \big| \mathcal{F}_{t-1}\big] &= W^1\left(\Theta_{\hat{\mu}^N_{t-1}}(\cdot|{X}^i_{t-1}),\Theta_{\mu_{t-1}}(\cdot|\overline{X}^i_{t-1})\right)\\
&\leq W^1\left(\Theta_{\hat{\mu}^N_{t-1}}(\cdot|{X}^i_{t-1}),\Theta_{\bar{\mu}^N_{t-1}}(\cdot|\overline{X}^i_{t-1})\right)\\
&\qquad\qquad+W^1\left(\Theta_{\bar{\mu}^N_{t-1}}(\cdot|\overline{X}^i_{t-1}),\Theta_{\mu_{t-1}}(\cdot|\overline{X}^i_{t-1})\right)
\end{align*}
where $\bar{\mu}^N_t=\frac{1}{N}\sum_{i=1}^N \delta_{\overline{X}^i_t}$ and $\mathcal{F}_t$ is the $\sigma$-algebra generated by the processes up to time $t\in\N$. We conclude that:
\begin{equation}\label{optimalproposal}
\E\big[|Y^i_t-\overline{Y}^i_t| \big| \mathcal{F}_{t-1}\big]\leq W^1(\hat{\mu}^N_{t-1},\bar{\mu}^N_{t-1})+|X^i_{t-1}-\overline{X}^i_{t-1}|+\epsilon_t^N
\end{equation}
where the error term $\epsilon_t^N$ only depends on (the laws of) the $N$ independent nonlinear processes $(\overline{X}^i_t)_t$:
\[\epsilon_t^N:= W^1(\bar{\mu}^N_{t-1},\mu_{t-1}).\]
\end{enumerate}

Let $t\in\N$, $t\geq1$. It holds that: 
\begin{align*}
\E\big[|X^i_t-\overline{X}^i_t| \big| \mathcal{F}_{t-1}\big] &=\E\big[|Y^i_t-\overline{Y}^i_t|\1_{U^i_t\leq\min(h^i_t,\overline{h}^i_t)} \big| \mathcal{F}_{t-1}\big]\\[0.15cm]
&\qquad+ |X^i_{t-1}-\overline{X}^i_{t-1}|\PP\big(U^i_t\geq\max(h^i_t,\overline{h}^i_t)\big|\mathcal{F}_{t-1}\big)\\[0.15cm]
&\qquad\qquad+\E\big[|X^i_t-\overline{Y}^i_t|\1_{h^i_t\leq U^i_t\leq \overline{h}^i_t} \big| \mathcal{F}_{t-1}\big]\\[0.15cm]
&\qquad\qquad\qquad+\E\big[|Y^i_t-\overline{X}^i_t|\1_{\overline{h}^i_t\leq U^i_t\leq {h}^i_t} \big| \mathcal{F}_{t-1}\big]
\end{align*}
where we write for short: 
\[h^i_t\equiv h\big(\alpha_{\hat{\mu}^N_t}(X^i_t,Y^i_t)\big)\quad\text{and}\quad\overline{h}^i_t\equiv h\big(\alpha_{\mu_t}(\overline{X}^i_t,\overline{Y}^i_t)\big).\]

we deduce that: 
\begin{multline}\label{EXXbar}
\E\big[|X^i_t-\overline{X}^i_t| \big| \mathcal{F}_{t-1}\big]\leq W^1(\hat{\mu}^N_{t-1},\bar{\mu}^N_{t-1})+2|X^i_{t-1}-\overline{X}^i_{t-1}|\\
+2M_0(\PP(h^i_t\leq U^i_t\leq \overline{h}^i_t \big| \mathcal{F}_{t-1})+\PP(\overline{h}^i_t\leq U^i_t\leq {h}^i_t \big| \mathcal{F}_{t-1}))+\epsilon_t^N
\end{multline}
The last two probabilities are bounded by $\E\big[|h^i_t-\overline{h}^i_t|\big|\mathcal{F}_{t-1}\big]$. Assuming that ${h}$ is $L_h$-Lipschitz for a constant $L_h>0$, it holds that: 
 \[|h^i_t-\overline{h}^i_t|\leq L_h\Big|{\alpha}_{\hat{\mu}_{t-1}}(X^i_{t-1},Y^i_{t})-{\alpha}_{\mu_{t-1}}(\overline{X}^i_{t-1},\overline{Y}^i_{t})\Big|.\]
 Let $\bar{\mu}^N_t$ be the empirical measure of the $N$ nonlinear Markov processes $\overline{X}^i_t$ at time $t$. It holds that: 
 \begin{multline*}
 |h^i_t-\overline{h}^i_t|\leq L_h\Big|{\alpha}_{\hat{\mu}_{t-1}}(X^i_{t-1},Y^i_{t})-{\alpha}_{\bar{\mu}_{t-1}}(\overline{X}^i_{t-1},\overline{Y}^i_{t})\Big|\\
 +L_h\Big|{\alpha}_{\bar{\mu}_{t-1}}(\overline{X}^i_{t-1},\overline{Y}^i_{t})-{\alpha}_{\mu_{t-1}}(\overline{X}^i_{t-1},\overline{Y}^i_{t})\Big|\end{multline*}
 Using Lemma \ref{lemma:lipschitzratio} we get:
 \[|h^i_t-\overline{h}^i_t|\leq C_{\Theta}\left(W^1(\hat{\mu}^N_{t-1},\bar{\mu}^N_{t-1})+|X^i_{t-1}-\overline{X}^i_{t-1}|+|Y^i_t-\overline{Y}^i_t|+\epsilon^N_t\right),\]
 with $C_\Theta := L_hL_\Theta$. Therefore: 
\[\E\big[|X^i_t-\overline{X}^i_t| \big| \mathcal{F}_{t-1}\big]\leq \left(1+C_{\Theta}\right)\big(|X^i_{t-1}-\overline{X}^i_{t-1}|+W^1(\hat{\mu}^N_{t-1},\bar{\mu}^N_{t-1})\big)+(1+2C_{\Theta})\epsilon_t^N.\]

Let us define: 
\[S_t :=\frac{1}{N}\sum_{i=1}^N |X^i_t-\overline{X}^i_{t}|.\]
Summing the previous expression for $i$ from 1 to $N$ and dividing by $N$ gives the following estimate for $S_t$:

\begin{equation}\label{eq:estf}\E[S_t|\mathcal{F}_{t-1}]\leq2\left(1+C_{\Theta}\right)S_{t-1}+(1+2C_{\Theta})\epsilon_t^N\end{equation}
where we have used the fact that 
\[W^1(\hat{\mu}^N_{t-1},\bar{\mu}^N_{t-1})\leq S_{t-1}.\]
Taking the expectation in \eqref{eq:estf}, we deduce that: 
\begin{equation}\label{gronwallestimate}
\E[S_t]\leq\left(1+C_{\Theta}\right)\E[S_{t-1}] + C_{\Theta}\E[\epsilon_t^N]
\end{equation}
where the value of the constant $C_\Theta$ has been updated by $C_{\Theta}\leftarrow1+2C_{\Theta}$.

The error term can be controlled uniformly on $t$ using \cite[Theorem 5.8]{carmona2018probabilistic} or \cite{fournier2015rate}. In particular, since $\pi$ is a smooth probability density function on a compact set, it has finite moments of all orders and therefore it follows from \cite[Theorem 1]{fournier2015rate} that:
\begin{equation}\label{errorestimate}\forall t\in\N,\,\,\E[\epsilon_t^N]\leq \beta(N)\end{equation}
where $\beta(N)$ is defined by \eqref{eq:glivenkocantelli}.

One can easily prove by induction that: 

\begin{equation}
\E[S_t] \leq C_{\Theta}\beta(N)\sum_{s=0}^{t-1}\e^{C_{\Theta}s}\leq\beta(N)\frac{C_{\Theta}}{\e^{C_{\Theta}}-1}\e^{tC_{\Theta}}
\end{equation}

By symmetry of the processes, all the quantities $\E[|X^i_t-\overline{X}^i_t|]$
are equal and their common value is $\E[S_t]$. The result follows.
\end{proof}

\begin{rem}[Moderate interaction, part 2]\label{rem:moderateinteraction2} The result of Theorem \ref{thm:pathwiseestimate} provides an explicit convergence rate in terms of $N$. This could be used to understand more precisely the moderate interaction assumption mentioned in Remark~\ref{rem:moderateinteraction} in order to justify the study of the degenerate proposal distribution $\Theta_\mu(\dd y|x) = \mu(\dd y)$ in the nonlinear Markov chain with transition kernel \eqref{eq:cmckernel}. In the case $\Theta_\mu(\dd y|x)=K\star\mu(y)\dd y$, one can take at the particle level (\emph{i.e} in Algorithm \ref{algo:cmc}) an interaction kernel $K\equiv K^N$ which depends on $N$. Typically, one can consider $K^N$ equal to the density of a centered Gaussian law with covariance $\sigma_N I_d$, $\sigma_N>0$ or the normalized indicator of the ball centered at zero and with radius $\sigma_N$. The goal is to take the size $\sigma_N\to0$ as $N\to+\infty$ (and thus $K^N\to\delta_0$). As a consequence the constant $C_\Theta\equiv C_\Theta^N$ in Theorem \ref{thm:pathwiseestimate} would depend on $N$. Since we have a precise control on $\beta(N)$ we can choose $\sigma_N$ such that the following convergence still holds for all $t\leq T$ smaller than a fixed finite value~$T\in(0,+\infty)$: 
\[\beta(N)\e^{tC_\Theta^N}\underset{N\to+\infty}{\longrightarrow}0.\]
In particular, $\sigma_N$ should not converges to zero too fast, justifying the moderate interaction terminology introduced in \cite{Oelschl_ger_1985}. In the limit $N\to+\infty$ we then obtain that the emprirical measure $\hat{\mu}^N_t$ converges towards the $t$-th iterate of the transition operator \eqref{eq:transitionoperator} with the degenerate choice of proposal distribution $\Theta_f(\dd y|x)=f(y)\dd y$ which makes sense as soon as $f\in\pdf_0(E)$. We refer the reader to \cite{jourdain1998propagation} and \cite{diez2019propagation} for two examples of propagation of chaos results under a moderate interaction assumption. Note that this result is mainly of theoretical interest as it does not give sharp estimates on how slow $\sigma_N$ should decrease to zero. Moreover, similarly to Theorem \ref{thm:convergencecmc}, this result is only valid when $N\to+\infty$ on a finite bounded time interval. 

\end{rem}

\begin{rem}[About the assumptions]\label{rem:aboutassumptions} In order to prove a propagation of chaos property, it is usually assumed that the parameters of the problem are Lipschitz \cite{sznitman1991topics,meleard1996asymptotic}. This corresponds to the two Lipschitz assumptions \ref{assum:lipschitz} and \ref{assum:nonexpansive}. Propagation of chaos in non-Lipschitz settings is a more difficult problem (see for instance \cite{jabin2016mean} for a recent result). 

Assumption \ref{assum:pi} and Assumption \ref{assum:bound} should be understood as technical assumptions. In the proof of Theorem \ref{thm:poc}, we use the fact that the acceptance ratio is Lipschitz (Lemma \ref{lemma:lipschitzratio}) which follows directly from Assumptions \ref{assum:pi}, \ref{assum:bound} and \ref{assum:lipschitz}. However, we could relax the compactness assumption \ref{assum:pi} and keep the same Lipschitz property by replacing Assumptions \ref{assum:pi}, \ref{assum:bound} and \ref{assum:lipschitz} by the following assumption. 
\begin{assumption} The target distribution $\pi$ does not vanish on $E$ and the map 
\[\pb(E)\times E^2\to \R,\quad(\mu,x,y)\mapsto g_\mu(y|x):=\frac{\Theta_\mu(y|x)}{\pi(y)}\]
satisfies the two following properties.
\begin{itemize}
\item \textbf{(Boundedness).} There exists two constants $\kappa_->0$ and $\kappa_+>0$ such that
\[\forall (x,y)\in E^2,\quad\kappa_-\leq g(y|x)\leq \kappa_+.\]
\item \textbf{(Lipschitz).} There exists a constant $L>0$ such that
\begin{multline*}\forall (\mu,x,y),(\nu,x',y')\in\pb(E)\times E^2,\\ |g_\mu(y|x)-g_\nu(y'|x')|\leq \Big(W^1(\mu,\nu)+|x-x'|+|y-y'|\Big).\end{multline*}
\end{itemize}
\end{assumption}
In practice, this would necessitate a precise control of the tails of $\pi$ and of the proposal distribution. It seems easier for us to check the compactness and boundedness assumptions \ref{assum:pi}, \ref{assum:bound} and \ref{assum:lipschitz} (possibly up to truncating the support of $\pi$ and replacing it by a compact set). 
\end{rem}

\section{Continuous time version}\label{appendix:continuoustime}

As mentioned in Section \ref{sec:convergence}, entropy methods are widely used in a continuous-time context, in particular for the long-time analysis of nonlinear PDEs \cite{jungel2016entropy}. In this section, we define a continuous-time version of Algorithm \ref{algo:cmc}, we show that the mean-field limit towards the solution of a nonlinear PDE and we prove the exponential long-time convergence of its solution towards $\pi$. This last step also motivates the discrete-time analysis in Section~\ref{sec:convergence}. Finally, the application of these ideas to the classical linear Metropolis-Hastings case leads (formally) to a convergence result obtained earlier in \cite{diaconis2011geometric} with other techniques (Section \ref{sec:metropolis}).

\subsection{Continuous-time particle system}

The $N$-particle system constructed by Algorithm \ref{algo:cmc} is a standard Markov chain which can be turned into a continuous-time Markov process by subordinating it to a Poisson process, as explained in \cite[Chapter 8, Definition 2.2]{bremaud}. Namely, let $(N_t)_{t\geq0}$ be a Poisson process independent of all the other random variables and with constant parameter equal to 1. The continuous-time particle system is defined at time $t\geq0$ by $\tilde{X}^i_t := X^i_{N_t}$ for $i\in\{1,\ldots,N\}$, where $X^i_{N_t}$ is the $i$-th particle constructed by Algorithm \ref{algo:cmc} at iteration $N_t$. For notational simplicity, in the remaining of this section, we drop the tilde notation and simply write $X^i_t\equiv\tilde{X}^i_t$ with $t\in[0,+\infty)$. 

The first step is the mean-field limit analog of Theorem \ref{thm:pathwiseestimate}. 

\begin{theorem}\label{thm:EW1tcontinuous}
Let the particles $X^i_0\sim f_0$ be initially i.i.d. with common law $f_0\in\pdf_0(E)$. Under the assumptions of Theorem \ref{thm:pathwiseestimate}, the (random) continuous-time empirical measure $\hat{\mu}^N_t = \frac{1}{N}\sum_{i=1}^N \delta_{X^i_t}$ at any time $t\geq0$ satisfies
\[\E W^1(\hat{\mu}^N_t, f_t) \leq C_1\beta(N)\e^{tC_2},\]
where $C_1,C_2>0$ are two absolute constants, $\beta(N)$ is defined by \eqref{eq:glivenkocantelli} and $f_t$ is the solution of the nonlinear PDE
\begin{equation}\label{eq:nonlinearpde_appendix}\partial_t f_t = \mathcal{T}[f_t]-f_t,\end{equation}
with initial condition $f_0$. 
\end{theorem}

\begin{proof}
It is enough to prove the analog of \eqref{eq:couplingbound} in the time continuous framework. The end of the proof will then follows similarly as in Corollary \ref{thm:poc}. The nonlinear process $(\overline{X}_t)_t$ with law $f_t$ is defined as follows. 
\begin{enumerate}
    \item Let $(T_n)_{n\in\N}$ the increasing sequence of jump times of the Poisson process $(N_t)_t$ with $T_0=0$ and initially $\overline{X}_0 \sim f_0$. 
    \item Between two jump times, at any $t\in[T_n,T_{n+1})$, $\overline{X}_t=\overline{X}_{T_n}$. 
    \item At each jump time $T_n$, we sample a proposal random variable $\overline{Y}_{T_n}\sim \Theta_{f_{T_n}}(\cdot|\overline{X}_{T_n})$ and accept or reject it as in Definition \ref{def:nonlinearprocess}. 
\end{enumerate}
Then we construct two coupled systems of $N$ independent nonlinear processes $(\overline{X}^i_t)^{}_t$ and $N$ particles as in the proof of Theorem \ref{thm:pathwiseestimate}. We define similarly $S_t := \frac{1}{N}\sum_{i=1}^N \E|X^i_t-\overline{X}^i_t|$. As in the discrete time case \eqref{eq:estf}, at any jump time $T_n$, it holds that
\[\E[S_{T_n}|\mathcal{F}_{T_n^-}]\leq2\left(1+C_{\Theta}\right)S_{T_{n-1}}+C_{\Theta}\beta(N).\]
Taking the conditional expectation with respect to $\mathcal{G}_n=\sigma(T_1,T_2-T_1,\ldots,T_n-T_{n-1})$, we obtain: 
\[\E[S_{T_n}|\mathcal{G}_n]\leq2\left(1+C_{\Theta}\right)\E[S_{T_{n-1}}|\mathcal{G}_{n-1}]+C_{\Theta}\beta(N).\]
And thus, for all $n\in\N$, 
\[\E[S_{T_n}|\mathcal{G}_n]\leq \beta(N)\frac{C_{\Theta}}{\e^{C_{\Theta}}-1}\e^{C_{\Theta} n}.\]
Since $N_t$ follows a Poisson law with parameter $t$, it holds that: 
\begin{align*}
    \E[S_{T_n}\1_{N_t=n}]&=\E[\E[S_{T_n}\1_{N_t=n}|\mathcal{G}_{n+1}]]\\
    & = \E[\1_{N_t=n}\E[S_{T_n}|\mathcal{G}_n]]\\
    &\leq \beta(N)\frac{C_{\Theta}}{\e^{C_{\Theta}}-1}\e^{C_{\Theta} n}\PP(N_t = n),
\end{align*}
where the second inequality comes from the fact that the event $\{N_t = n\}$ is $\mathcal{G}_{n+1}$ measurable and the fact that $S_{T_n}$ is independent from $T_{n+1}-T_n$. As a consequence, since $S_t=S_{T_{N_t}}$ and $\PP(N_t = n)=\e^{-t}t^n/n!$, we conclude that:
\[\E[S_t]=\E[S_{T_{N_t}}]=\sum_{n=0}^{+\infty}\E[S_{T_n}\1_{N_t=n}] \leq \beta(N)\frac{C_{\Theta}}{\e^{C_{\Theta}}-1}\exp(t(\e^{C_{\Theta}}-1)).\]
\end{proof}

\subsection{Convergence of the nonlinear process}\label{appendix:convergencentropy}

The analog of Theorem \ref{thm:convergencenonlineardiscrete} in the continuous-time setting is the following theorem. 

\begin{theorem}[Convergence of the Nonlinear Process]\label{thm:convergencetopi}
Let $\Theta$ a proposal distribution which satisfies Assumption \ref{assum:monotonicity}. Let $f_0\in\pdf_0(E)$ and let $f_t$ be the solution at time $t\in[0,+\infty)$ of the nonlinear PDE \eqref{eq:nonlinearpde_appendix}. Then for all $t\geq0$, it holds that:
\[\|f_t-\pi\|_{\mathrm{TV}}\leq C_0\e^{-\lambda t},\]
where $C_0>0$ depends only on $f_0$ and $\pi$ and where
\begin{equation}\label{eq:lambda}\lambda:=c^-\left(\inf_x \frac{f_0(x)}{\pi(x)}\right)h(1)>0.\end{equation}
\end{theorem}

Note that in the continuous-time case, Assumption \ref{assum:hbounded} is not required.

We recall that given a convex function $\phi:[0,+\infty)\to[0,+\infty)$ such that $\phi(1)=0$, the relative entropy $\mathcal{H}[f|\pi]$ and dissipation $\mathcal{D}[f|\pi]$ of a probability density $f\in\pdf(E)$ with respect to $\pi$ are defined respectively by
\begin{equation*}\mathcal{H}[f|\pi] := \int_E \pi(x) \phi{\left(\frac{f(x)}{\pi(x)}\right)} \dd x,\quad \mathcal{D}[f|\pi] := -\int_E \phi'{\left(\frac{f(x)}{\pi(x)}\right)}\big(\mathcal{T}[f](x)-f(x)\big)\dd x.\end{equation*}

In the continuous time framework when $f_t$ solves \eqref{eq:nonlinearpde_appendix}, the entropy and dissipation are simply linked by the relation 
\begin{equation}\label{eq:entropydissipation}\frac{\dd}{\dd t}\mathcal{H}[f_t|\pi] = -\mathcal{D}[f_t|\pi]\leq 0,\end{equation}
and therefore we conclude immediately that the entropy is non increasing (which is Lemma \ref{lemma:discreteentropynonincreasing} in the discrete setting). A consequence of this fact is an alternative proof of Lemma \ref{lemma:infsupdiscrete}. 

\begin{lemma}\label{lemma:infsup_appendix} Let $f_t$ be the solution of the integro-differential equation \eqref{eq:nonlinearpde_appendix} with initial condition $f_0\in\pdf_0(E)$. Then 
\[\inf_{x\in E} \frac{f_t(x)}{\pi(x)}\geq \inf_{x\in E} \frac{f_0(x)}{\pi(x)},\quad \sup_{x\in E} \frac{f_t(x)}{\pi(x)}\leq \sup_{x\in E} \frac{f_0(x)}{\pi(x)}.\]
\end{lemma}

\begin{proof} Let us denote
\[m:=\inf_{x\in E} \frac{f_0(x)}{\pi(x)}\quad\text{and}\quad M:=\sup_{x\in E} \frac{f_0(x)}{\pi(x)}.\]
Let us take $\phi:[0,+\infty)\to[0,+\infty)$ a convex function such that $\phi\equiv0$ on the segment $[m,M]$ and $\phi>0$ elsewhere. Note that since $f_0$ and $\pi$ are both probability densities, it holds that $m<1$ and $M>1$ and thus $\phi(1)=0$. The entropy-dissipation relation \eqref{eq:entropydissipation} gives: 
\[\frac{\dd}{\dd t} \mathcal{H}[f_t|\pi] \leq 0\]
and therefore for all $t\geq0$,
\[\mathcal{H}[f_t|\pi]\leq \mathcal{H}[f_0|\pi] =0\]
by definition of $\phi$. As a consequence and since $\phi\geq0$ and $\pi>0$ on $E$, it holds that for all $t\geq0$ and all $x\in E$, 
\[\phi\left(\frac{f_t(x)}{\pi(x)}\right)=0,\]
which implies that 
\[\forall x\in E,\quad m\leq \frac{f_t(x)}{\pi(x)}\leq M.\]
\end{proof}

In order to prove Theorem \ref{thm:convergencetopi}, we follow the classical steps which are detailed for instance in \cite[Section 1.3]{jungel2016entropy} and can be applied to various linear and nonlinear jump and diffusion processes. 

\begin{enumerate}
    \item Compute the dissipation $\mathcal{D}[f_t|\pi]=-\frac{\dd}{\dd t} \mathcal{H}[f_t|\pi]$.
    \item Prove that the dissipation can be bounded from below by a multiple of the entropy: for a constant $\lambda>0$,
    \[\mathcal{D}[f_t|\pi]\geq \lambda \mathcal{H}[f_t|\pi].\]
    \item Apply Gronwall lemma to the relation $\frac{\dd}{\dd t} \mathcal{H}[f_t|\pi]\leq \lambda \mathcal{H}[f_t|\pi]$ to obtain the exponential decay of the entropy: \[\mathcal{H}[f_t|\pi]\leq \mathcal{H}_0 \e^{-\lambda t}.\]
    \item Show that the entropy controls the TV distance and conclude that for some constants $c, C_0>0$: 
    \[\|f_t-\pi\|_{\mathrm{TV}}\leq c \mathcal{H}[f_t|\pi]\leq C_0 \e^{-\lambda t}.\]
\end{enumerate}

\begin{proof}[Proof (of Theorem \ref{thm:convergencetopi})]
Let $\phi(s)=\frac{1}{2}(s-1)^2$. Most of the computations are the same as in the discrete-time case. From the entropy-dissipation relation \eqref{eq:entropydissipation}, Lemma~\ref{lemma:infsup_appendix} and Lemma \ref{lemma:boundbelowW}, it follows that
\begin{multline*}\frac{\dd}{\dd t}\mathcal{H}[f_t|\pi] = -\mathcal{D}[f_t|\pi]\\ \leq -\frac{c^-(m)h(1)}{2}\iint_{E\times E}\pi(x)\pi(y)\left|\frac{f_t(x)}{\pi(x)}-\frac{f_t(y)}{\pi(y)}\right|^2\dd x\dd y=-2c^-(m)h(1)\mathcal{H}[f_t|\pi],
\end{multline*}
where $m:=\inf_{x\in E}f_0(x)/\pi(x)$. Using Gronwall's inequality we then deduce that: 
\[\mathcal{H}[f_t|\pi]\leq \mathcal{H}[f_0|\pi]\e^{-2c^-(m)h(1) t}.\]
The conclusion follows from the Cauchy-Schwarz inequality by writing
\[\|f_t-\pi\|_{\mathrm{TV}}=\int_E |f_t(x)-\pi(x)|\dd x = \int_E \sqrt{\pi}\sqrt{\pi}\left|\frac{f_t(x)}{\pi(x)}-1\right|\dd x \leq \sqrt{2\mathcal{H}[f_t|\pi]},\]
where we have used the fact that the TV norm is equal to the $L^1$ norm of the probability density functions. 
\end{proof}

\begin{rem} Another natural choice for $\phi$ would be $\phi(s)=s\log s-s+1$. The relative entropy is in this case equal to the Kullback-Leibler divergence. However, the dissipation term becomes in this case: 
\begin{multline*}\mathcal{D}[f_t|\pi] = \frac{1}{2}\iint_{E\times E} W_{f_t}(x\to y)\pi(x)\left(\frac{f_t(x)}{\pi(x)}-\frac{f_t(y)}{\pi(y)}\right)\\
\times\left(\log\left(\frac{f_t(x)}{\pi(x)}\right)-\log\left(\frac{f_t(y)}{\pi(y)}\right)\right)\dd x\dd y,\end{multline*}
and it is not clear that it can be bounded from below by the relative entropy. Note that this dissipation functional is very similar to the one obtained in the study of the Boltzmann equation (in this context, the Kullback-Leibler divergence is also called the Boltzmann entropy). The long-time asymptotics of this equation is a long-standing problem and the specific question of whether the dissipation controls the entropy is the object of a famous conjecture by Cercignani \cite{desvillettes2001celebrating,villani2003cercignani}. In our case, we know that the Kullback-Leibler divergence is decreasing with time but all this suggests that its exponential decay could be harder to obtain or could hold only in specific cases.
\end{rem}

Putting together Theorem \ref{thm:EW1tcontinuous} and Theorem \ref{thm:convergencetopi} leads to the following corollary which is the continuous-time analog of Theorem \ref{thm:convergencecmc}. 

\begin{corollary} Under the assumptions of Theorem \ref{thm:EW1tcontinuous} and Theorem \ref{thm:convergencetopi}, for any $t>0$ and $N>1$ it holds that
\[\E[W^1(\hat{\mu}^N_t,\pi)] \leq C_1 \beta(N)\e^{C_2t} + C_3 \e^{-\lambda t},\]
where $\beta(N)$ is given by \eqref{eq:glivenkocantelli}, $\lambda>0$ is given by \eqref{eq:lambda} and $C_1,C_2,C_3>0$ are absolute constants.
\end{corollary}

\begin{proof}
By the triangle inequality, it holds that
\[\E[W^1(\hat{\mu}^N_t,\pi)] \leq \E[W^1(\hat{\mu}^N_t,f_t)] + \E[W^1(f_t,\pi)].\]
The first term on the right-hand side is bounded by $C_1 \beta(N)\e^{C_2t}$ by Theorem \ref{thm:EW1tcontinuous}. For the second term on the right-hand side, we first note that on the compact set $E$, the total variation norm controls the Wasserstein-1 distance \cite[Theorem 6.15]{villani_optimal_2009}. The conclusion therefore follows from Theorem \ref{thm:convergencetopi}. 
\end{proof}

\subsection{Links between the discrete- and continuous-time versions}\label{appendix:linksdiscretecontinuous}

The discrete-time counterpart of the entropy-dissipation relation \eqref{eq:entropydissipation} is the relation \eqref{eq:entropydissipationtaylor}. The difference $\mathcal{H}[\mu_{t+1}|\pi]-\mathcal{H}[\mu_t|\pi]$ is the discrete analog of a time derivation. The main difference with the continuous-time entropy-dissipation relation is the additional non negative term on the right-hand side of \eqref{eq:entropydissipationtaylor}. The role of the technical Assumption \ref{assum:hbounded} is to ensure that this non negative term remains smaller than the dissipation in order to close the argument as in the continuous-time case. Since this term does not appear in the continuous-time setting, Assumption \ref{assum:hbounded} is not required to prove Theorem \ref{thm:convergencetopi}. 

Assumption \ref{assum:hbounded} is actually better understood when the discrete-time relation \eqref{eq:recmu} is seen as an explicit Euler discretization scheme \eqref{eq:expliciteuler} of the nonlinear PDE \eqref{eq:nonlinearpde_appendix}. More precisely, the numerical scheme 
\eqref{eq:expliciteuler} can be re-written
\begin{equation}\label{eq:discretescheme_appendix}\mu_{t+1} = \mu_t+\Delta t Q[\mu_t],\end{equation}
where
\[Q[\mu](\dd x) := \int_E \pi(x)W_\mu(x\to y){\left(\frac{\mu(\dd y)}{\pi(y)}\dd x - \frac{\mu(\dd x)}{\pi(x)}\dd y\right)}.\]
Note that $\mathcal{T}[\mu]=\mu+Q[\mu]$. In order to check that this numerical scheme preserves the continuous-time entropy-dissipation relation \eqref{eq:entropydissipation}, let us write for a general function the second-order Taylor expansion: 
\[\mathcal{H}[\mu_{t+1}|\pi] \simeq \mathcal{H}[\mu_t|\pi] - \Delta t\mathcal{D}[\mu_t|\pi]+\frac{\Delta t^2}{2}\int_E \pi(x)\phi''{\left(\frac{\mu_t(x)}{\pi(x)}\right)}{\left|\frac{Q[\mu_t](x)}{\pi(x)}\right|}^2\dd x.\]
Since $\phi$ is a convex function, the second-order term of the Taylor expansion is non negative but it is dominated by the non positive first-order term (equal to minus the dissipation) for $\Delta t$ small enough. For this reason, the explicit Euler discretization scheme does not unconditionally preserve the entropy structure of the continuous-time PDE. Since $W_\mu(x\to y)=\Theta_\mu(y|x)h(\alpha_\mu(x,y))$, the operator $Q[\mu]$ is proportional to $h$ and it is equivalent to assume $\Delta t<1$ or $h<1$. With the notations of Section \ref{sec:exponentialdecaydiscrete}, the time-step is $\Delta t\equiv\eta$ and Assumption \ref{assum:hbounded} can therefore be interpreted as a ``numerical'' condition to preserve the entropy structure of the discrete scheme. We did not manage to prove the exponential convergence of the discrete scheme when $\Delta t \equiv \eta=1$. However, it is still possible to prove the convergence of the discrete scheme without rate and for a weaker topology using Theorem \ref{thm:convergencetopi} and a compactness argument. 

\begin{corollary}[Convergence of the discrete scheme]\label{coro:continuousimpliesdiscrete}
Let $(\mu_t)_{t\in\N}$ be the sequence of probability laws defined by the recurrence relation \eqref{eq:recmu} (i.e. by the discrete scheme \eqref{eq:discretescheme_appendix} with $\Delta t=1$). Let $\mu_0\in\pdf_0(E)$ be the initial condition and let $\Theta$ satisfy Assumption \ref{assum:monotonicity}. Then it holds that $\mu_t\to\pi$ as $t\to+\infty$ for the weak convergence of probability measures.
\end{corollary}

\begin{proof}
The sequence $(\mu_t)_{t\in\N}$ is tight because $E$ is compact so it admits a converging subsequence. The limit of any converging subsequence is a fixed point of the operator~$\mathcal{T}$. Since the convergence result stated in Theorem \ref{thm:convergencetopi} does not depend on the initial condition in $\pdf_0(E)$, it implies that $\pi$ is the unique fixed point of the operator $\mathcal{T}$ in $\pdf_0(E)$ and therefore, all the converging subsequences of $(\mu_t)_t$ converges towards $\pi$, which implies the convergence of the whole sequence. 
\end{proof}

\subsection{The Metropolis-Hastings case}\label{sec:metropolis}
The following theorem revisits the main result of \cite{diaconis2011geometric} regarding the convergence rate of the Metropolis-Hastings algorithm, here formally proved with the entropy techniques introduced in Section \ref{appendix:convergencentropy} and in the continuous-time setting for simplicity. 

\begin{theorem}[Formal]\label{thm:convergencemetropolis} Let us consider the linear case outlined in Section \ref{sec:pmh} with $q(y|x)=K_\sigma(x-y)$, where $K_\sigma$ is a fixed symmetric random walk kernel of size $\sigma>0$ (typically a Gaussian kernel with standard deviation $\sigma$). Assume that, as $\sigma\to0$, the random-walk kernel $K_\sigma$ satisfies for any smooth function $\varphi\in~C^\infty_0(\R^d)$: 
\begin{equation}\label{eq:taylor}\int_E\varphi(x)K_\sigma(x)\dd x = \varphi(0)+\frac{1}{2}\sigma^2\Delta\varphi(0) +o(\sigma^2).\end{equation}
Assume that $\pi$ and $E$ are such that the following Poincar\'{e} inequality holds:
\begin{equation}\label{eq:poincare}
    \int_E u(x)^2\pi(x) \dd x-\left(\int_E u(x)\pi(x)\dd x\right)^2\leq \frac{1}{\lambda_P}\int_E |\nabla u|^2(x)\pi(x)\dd x
\end{equation}
for all functions $u$ in the weighted Sobolev space $H^1_\pi(E)$ and for a constant $\lambda_P>0$.
Then, as $\sigma\to 0$, it holds that: 
\[\|f_t-\pi\|_{\mathrm{TV}}\leq C_0 \e^{-\frac{1}{2}\sigma^2(\lambda_P+o(1)) t},\]
for a constant $C_0>0$ which depends only on $f_0$ and $\pi$. 
\end{theorem}

\begin{proof} The formal Taylor expansion \eqref{eq:taylor} as $\sigma\to0$ applied to \eqref{eq:entropydissipation} leads to
\begin{equation}\label{eq:dissipationmetropolis}\mathcal{D}_\phi[f_t|\pi]=\frac{\sigma^2}{2}\int_{E}\pi(x)\phi''\left(\frac{f_t(x)}{\pi(x)}\right)\left|\nabla_x\left(\frac{f_t}{\pi}\right)\right|^2\dd x+o(\sigma^2).\end{equation}
Taking $\phi(s)=\frac{1}{2}(s-1)^2$ gives: 
\[\mathcal{D}_\phi[f_t|\pi]=\frac{\sigma^2}{2}\int_E \pi\left|\nabla\left(\frac{f_t}{\pi}\right)\right|^2\dd x+o(\sigma^2).\]
Using the Poincar\'{e} inequality with the function $u=f_t/\pi$, we obtain: 
\[\mathcal{D}_\phi[f_t|\pi]\geq \frac{\sigma^2\lambda_P}{2}\left(\int_E \left(\frac{f_t}{\pi}\right)^2\pi\dd x-1\right)+o(\sigma^2) = \sigma^2\lambda_P \mathcal{H}_\phi[f_t|\pi]+o(\sigma^2).\]
From the entropy-dissipation relation \eqref{eq:entropydissipation} and Gronwall lemma, we deduce that
\[\mathcal{H}_\phi[f_t|\pi]\leq C_0 \e^{-\sigma^2(\lambda_P+o(1)) t}.\]
The conclusion follows from the Cauchy-Schwarz inequality as in the proof of Theorem~\ref{thm:convergencetopi}.
\end{proof}

A similar result has been obtained rigorously in \cite{diaconis2011geometric} using linear spectral theory. In particular, the fact that a Poincar\'{e} inequality holds depends on the regularity of the boundary of $E$. In our case, the argument could be made rigorous by studying in details the wellposedness of \eqref{eq:nonlinearpde_appendix} in Sobolev spaces. The argument could also lead to a more detailed analysis of the convergence rate of the Metropolis-Hastings algorithm in other metrics. In particular, the same formal argument holds when we take $\phi(s)=s\log(s)-s+1$ in \eqref{eq:dissipationmetropolis}. In this case, the relative entropy is the Kullback-Leibler divergence and its time derivative is controlled by the following dissipation: 
\[\mathcal{D}_\phi[f_t|\pi]=\frac{\sigma^2}{2}\int_E \pi\left|\nabla\sqrt{\frac{f_t}{\pi}}\right|^2\dd x+o(\sigma^2).\]
In order to apply Gronwall lemma and obtain the exponential decay of the Kullback-Leibler divergence, we need the following convex Sobolev inequality: 
\[\int_E u(x)^2\log\left(\frac{u(x)^2}{\|u\|_{L^2_\pi}^2}\right)\pi(x)\dd x\leq \frac{1}{\lambda_S}\int_E |\nabla u(x)|^2\pi(x)\dd x,\]
for all $u\in H^1_\pi$ and for a constant $\lambda_S>0$. The conclusion follows by applying this inequality to $u=\sqrt{f_t/\pi}$ and as before from Gronwall lemma and from the classical Csisz\'{a}r-Kullback-Pinsker inequality \cite[Theorem A.2]{jungel2016entropy} which shows that the Total Variation norm is controlled by the Kullback-Leibler divergence. When the above convex Sobolev inequilty holds, then the Poincar\'{e} inequality \eqref{eq:poincare} also holds with $\lambda_P\geq2\lambda_S$. More details on convex Sobolev inequalities can be found in \cite{bartier2006convex} or\cite[Section 2.2]{jungel2016entropy}. Their application to the rigorous computation of (optimal) convergence rates for the Metropolis-Hastings algorithm is left for future work.

\section{Related Works}\label{sec:related}

\subsection{Another Nonlinear MCMC sampler}\label{sec:anothernonlinear}

A nonlinear kernel which does not fit into the ``collective proposal'' category has been introduced in \cite{andrieu2011nonlinear} and is defined by: 
\[K_\mu(x,\dd y)=(1-\varepsilon)K^{MH}(x,\dd y)+\varepsilon Q_\mu(x,\dd y),\]
where $K^{MH}$ is the Metropolis-Hastings kernel and 
\[Q_\mu(x,\dd y) = \left(1-\int_E\alpha(x,u)\mu(\dd u)\right)\delta_x(\dd y)+\alpha_\eta(x,y)\mu(\dd y).\]
The function $\alpha$ is defined by: $\alpha_\eta(x,y) = \eta(x)\pi(y)/(\eta(y)\pi(x))$,
that is, $\alpha_\eta(x,y)$ is the Metropolis ratio associated to \textbf{another} distribution $\eta\in\pdf_0(E)$. In \cite{andrieu2011nonlinear}, the authors investigated the case $\eta=\pi^{\tilde{\alpha}}$ for $\tilde{\alpha}\in(0,1)$. This kernel satisfies:
\[\iint_{E\times E} \phi(y) K_\eta(x,\dd y)\pi(\dd x)=\int_E \phi(y)\pi(\dd y).\]
The sampling procedure is therefore quite different as it requires an auxiliary chain to build samples from $\eta$ first in order to construct a sample from the desired nonlinear kernel. More precisely, the authors propose the following iterative procedure to construct a couple of Markov chains $(X_t,Y_t)$: 
\[(X_{t+1},Y_{t+1}) \sim \Big((1-\varepsilon)K^{MH}(X_t,\dd x_{t+1})+\varepsilon Q_{\hat{\mu}_t^Y}(X_t,\dd x_{t+1})\Big)P(Y_t,\dd y_{t+1}),\]
where $P$ is a (linear) Markov transition kernel with invariant distribution $\eta$ and $(Y_t)_t$ is a Markov chain with transition kernel $P$. The empirical measure of this chain is denoted by:
\[\hat{\mu}^Y_t = \frac{1}{t+1}\sum_{s=0}^t \delta_{Y_s}.\]
The final MCMC approximation of an observable $\varphi$ is given in this case by:
\begin{equation}\label{eq:timeaverage}\int_E \varphi(x)\pi(\dd x) \simeq \frac{1}{t+1}\sum_{s=0}^t \varphi(X_s).\end{equation}
In this empirical sum, the successive iterations of the single chain $(X_t)_t$ are used. In the collective proposal framework introduced in Section \ref{sec:framework}, the algorithm produces $N$ (asymptotically) independent copies of a nonlinear chain $(X^i_t)_t$, $i\in\{1,\ldots,N\}$ and we have at our disposal a \emph{sequence} of MCMC approximations of the form: 
\begin{equation}\label{eq:ensembleaverage}\int_E \varphi(x)\pi(\dd x) \simeq \frac{1}{N}\sum_{i=1}^N \varphi(X_t^i),\end{equation}
as $t\to+\infty$. We can therefore interpret the sum \eqref{eq:timeaverage} as a \emph{time average} and the sum \eqref{eq:ensembleaverage} as an \emph{ensemble average}. 

\subsection{Links with Importance Sampling Based Methods}\label{sec:linksis}

Even though CMC does not use importance weights, it shares some similarities with importance sampling methods, in particular SMC \cite{del2006sequential} and PMC methods \cite{cappe2004population}. We will discuss the links with SMC in Section \ref{sec:smc} and focus here on PMC.

According to \cite{cappe2004population}, in PMC methods, without the importance correction, a regular acceptance step --- as in Metropolis-Hastings --- for each mutation would lead to a simple parallel implementation of $N$ Metropolis-Hastings algorithm. Under the same parallel, we can compare the mutation step in PMC with the proposal step of CMC.

In the first implementation of PMC, at each mutation step, each particle $X_t^i$ is updated independently from the others, according to a kernel $q_{it}( \cdot)$ (that can depend on $t$ and $i$), the new particle $X_{t+1}^i$ is then associated with a weight proportional to the ratio $\pi(X_{t+1}^{i})/q_{it}( X_{t+1}^{i} )$. In PMC, a mutation therefore occurs according to $q_{it}$, that is only depends on the position of the ancestor particle. In CMC, the mutation occurs according to $\Theta_{\hat{\mu}^N_t}( \cdot \mid X_t^i)$, the update thus depends on the position of all the particles as $\Theta$ depends on the empirical measure of the system $\hat{\mu}^N_t$. This additional dependency is particularly emphasized in the case of Algorithm \ref{algo:coker}. This corresponds to the Rao-Blackwellised version of PMC, in which we integrate over the position of all the particles.

Recently, \cite{delyon2019adaptive}
also proposed to Rao-Blackwellise the mutation kernel in PMC while keeping an importance sampling framework. The resulting algorithm is non-Markovian and does not conserve the number of particles. At each iteration a batch of particles is added to the system according to the previous estimation of the target density. The number of particles $N$ grows with the number of iterations. Once a particle is added, its position does not change, only its weight is updated along the iterations as explained in the following Algorithm \ref{algo:sais}. 

\begin{algorithm}[ht]
    \SetAlgoLined
    \KwIn{A target distributions $\pi$, \\
    an initial distribution $q_0$, \\
    a number of particles $N$,\\
    a sequence of transition kernels $K_{t}$\\
    a sequence of mixing weights $(\lambda_t)_t$\\
    a number of iterations $T$}\vspace{.1cm}
    \KwOut{A weighted sample $\big(X^i_t, w^i_t\big)_{1 \leq i \leq N ; \ 0 \leq t \leq T}$}\vspace{.1cm}
    
    Sample independently $N$ particles $X^i_0 \sim q_0$, $i\in\{1,\ldots,N\}$ \; 
    Compute the importance weights $w^i_0 = \pi(X^i_0)/q_0(X^i_0)$ \;

    \For{$t=1$ \textup{\textbf{to}} $T$}{\vspace{.05cm}
    
    Sample independently $N$ particles from $X^i_t \sim q_t := (1-\lambda_t)f_t + \lambda_t q_0$, $i\in\{1,\ldots,N\}$ where 
    \[f_t(x) = \sum_{s=0}^{t-1}\sum_{i=1}^N \tilde{w}^i_s K_t(x - X^i_s), \]
    and the weights $\tilde{w}^i_s = w^i_s/\sum_{s=0}^{t-1}\sum_{i=1}^N w^i_s$ are normalized. 
    
    Compute the new importance weights $w^i_t = \pi(X^i_t)/q_t(X^i_t)$ ;
    
    }
\caption{Safe and Adaptive Importance Sampling (SAIS)}
\label{algo:sais}
\end{algorithm}

The main hyper parameters to tune are the sequence of kernels $(K_t)_t$ and mixing weights $(\lambda_t)_t$. Typically the sequence of kernels is taken equal to $K_{t}(x)\equiv K_{h_t}(x) := K_0(x/h_t)/h_t^d$ where $K_0$ is a fixed kernel and $(h_t)_t$ is a sequence of bandwidths. In \cite{delyon2019adaptive}, the authors recommend to use
\begin{align*}
    h_t &= h_0 \left(1+\frac{2Nt}{T_0}\right)^{-1/(4+d)},\\
    \lambda_t &= \left\{ 
    \begin{array}{rcl}
    0.25 \left(1+\frac{2Nt}{T_0}\right)^{-1/(4+d)} & \text{if} & t\geq T_0, \\
    0.5 & \text{if} & T_0/2<t<T_0, \\ 
    1 & \text{if} & t\leq T_0/2.
    \end{array}
    \right.
\end{align*}
where $T_0 \sim T/10$ is the duration of the burn-in phase and $h_0$ is the initial bandwidth size which becomes the only parameter to choose. During the burn-in phase the authors also recommend to raise the importance weights to the power $3/4$ in order to prevent an early degeneracy of a small number of weights. We will always use these recommendations. 

This method shares some similarities with ours, an important difference being that an old particle cannot be improved through time and a ``bad'' particle will indefinitely remain in the system at the same place, while its weight decreases through time, eventually increasing the computation cost. A second difference is that it is necessary to tune the sequence of bandwidths and mixture weights. 
\medskip 

Importance sampling based methods output unbiased estimators. For CMC, as stated before, except for the Metropolis-Hastings proposal, each one of the methods previously described is biased. Indeed, for a fixed number $N$ of particles, the algorithm does not converge to the target distribution. For a large number of particles, the algorithm provides however a good approximation of the target density, according to Theorem \ref{thm:convergencecmc}. 
In addition, as a byproduct, we can re-use this approximation to provide an \emph{unbiased} estimator by simply using the (sequence of) collective proposal distributions as importance distributions in any importance sampling based sampler. A more thorough study is left for future work, but some elements can be found in Appendix \ref{sec:cmcis}.

\subsection{Adaptive Sequential Monte Carlo methods}\label{sec:smc}

A classical version of the Sequential Monte Carlo method well suited for sampling problems is given by Algorithm \ref{algo:smc}. 

\begin{algorithm}[ht]
    \SetAlgoLined
    \KwIn{A sequence of target distributions $\pi_0,\ldots,\pi_T$, \\
    a number of particles $N$,\\
    a sequence of transition kernels $K_{t,\xi}$ parametrized by $\xi\in E^N$, \\
    a threshold $ESS_\mathrm{min}$, \\ 
    a given allowed number $S$ of mutation steps}\vspace{.1cm}
    \KwOut{A weighted sample $\big(X^i_t, w^i_t\big)_{1 \leq i \leq N ; \ 1 \leq t \leq T}$}\vspace{.1cm}
    \textit{\textbf{(Initialization)}}\;\vspace{.1cm}
    \For{$i=1$ \textup{\textbf{to}} $N$}{\vspace{.05cm}
        Sample $X^i_0\sim \pi_0$. \;
        
        Set the weights $w^i_0 = 1/N$ \;}

    \For {$t=1$ \textup{\textbf{to}} $T$}{\vspace{.1cm}
    
    \textit{\textbf{(Correction)} \;}\vspace{.1cm}
    
    Update the weights $w^i_{t} =  w^i_{t-1}\frac{\pi_{t}(X^i_{t-1})}{\pi_{t-1}(X^i_{t-1})}$ for $i\in\{1,\ldots,N\}$\;
    Normalize the weights \; \vspace{.1cm}
    
    \textit{\textbf{(Resampling)} \;}\vspace{.1cm}
    
    Compute $ESS = (\sum_i (w^i_{t-1})^2 )^{-1}$ \;
    
    \If{ $ESS< ESS_\mathrm{min}$}{\vspace{.05cm}
    Resample the particles $X^i_{t-1} \sim \sum_j w^j_{t} \delta_{X^j_{t-1}}$ for all $i\in\{1,\ldots,N\}$\;
    
    Set $w^i_{t} = 1/N$ for all $i\in\{1,\ldots,N\}$
    }
    
    \vspace{.1cm}
    
    \textit{\textbf{(Mutation)} \;}\vspace{.1cm}
    
    Set $\mathcal{X}^N_{t-1} = (X^1_{t-1},\ldots, X^N_{t-1}) \in E^N$ \;
    \For{$i=1$ \textup{\textbf{to}} $N$}{\vspace{.05cm}
        Set $X^i_{0,t} = X^i_{t-1}$ \;
        \For{$s=1$ \textup{\textbf{to}} $S$}{\vspace{.05cm}
            Sample $X^i_{s,t}\sim K_{t,\mathcal{X}^N_{t-1}}(X^i_{s-1,t},\dd x)$ \;
        }
        Set $X^i_{t} = X^i_{S,t}$  \;

    }}
    \vspace*{.1cm}
\caption{Sequential Monte Carlo (SMC)}
\label{algo:smc}
\end{algorithm}

In this version, the sequence of target distributions is typically given by a tempering scheme between $\pi_0$, which is easy to sample from, and the final distribution $\pi_T = \pi$, the target distribution that we wish to sample from. The SMC algorithm alternates between two steps. During the mutation step, the particles evolve independently using a $\pi_t$-invariant random-walk kernel. During this step, the weights of the particles are updated using a classical importance sampling procedure. Then, during the resampling step, the particles with a small weight are eliminated and replaced by particles with a higher weight. The SMC algorithm is said to be adaptive when the transition kernel $K_{t,\mathcal{X}^n_{t-1}}$ at iteration $t$ depends on the population of particles at the previous iteration. 

This framework is general enough to encompass our main Algorithm \ref{algo:cmc}, by simply taking un-weighted particles (or $ESS_\mathrm{min}=N$), a constant sequence of target distributions and $K_{t,\mathcal{X}^N_{t-1}}$ equal to \eqref{eq:cmckernel} with $\mu$ equal to the empirical measure of the particle system. However our setting has several differences with the classical SMC methods advocated in the literature. 

In all the adaptive methods proposed in the literature and that we are aware of, the choice of the random-walk kernel does not depend on the whole distribution of the particles but rather on summary statistics, that is on real-valued quantities of the form $\langle \Psi_t, \hat{\mu}^N_t\rangle$ for a well-chosen test function $\Psi_t$. A typical choice is to consider the covariance. Considering only summary statistics has several advantages, in particular the algorithm keeps a linear footprint in $N$ and can be easily parallelized. Then, the mutation step is typically a standard MH algorithm whose flaws are expectedly corrected by the resampling procedure. However, the full knowledge of the particle swarm is not exploited and choosing the right summary statistics and the right mutation kernel are difficult tasks which may cause instability in practice. One goal in this article is precisely to avoid resampling and tempering by considering better mutation kernels thanks to wish we can safely rely on ergodicity only. Our main algorithm can be seen as the most general case where the mutation kernel depends on the whole distribution of the particle system and thus fully exploit the knowledge of the particle swarm. Even if the complexity of the algorithms that we propose becomes quadratic in $N$, we show that at least equally good results can be achieved.

On the theoretical side, this choice of mutation kernel is much more difficult to study and requires specific tools. On the contrary, due to their linear nature, the convergence properties of adaptive SMC methods with a mutation kernel which depends on summary statistics are typically studied with completely different tools. One important difference to note is that, unlike most of the particle-based algorithms in the literature, the target distribution $\pi^{\otimes N}$ is \emph{not} an invariant distribution of the particle system. This issue, which may seem critical at first sight, is actually solved asymptotically when $N\to+\infty$ thanks to a propagation of chaos result. This enlarge the class of mutation kernels.

Finally, we mention that although we did not investigate this direction, CMC within SMC would be a very natural idea.  

\section{Importance Sampling plug-in of CMC}\label{sec:cmcis}

\label{appendix:CMCIS}

Importance sampling methods are unbiased, by definition,  while CMC presents a (usually small) bias depending on the number of particles. To remove this bias, while keeping the original form of the algorithm, we present in Algorithm~\ref{algo:cmcIS} a small plug in, which output an unbiased sample without additional cost.

\begin{center}
\begin{minipage}{1\linewidth}
\centering
\begin{algorithm}[H]
    \SetAlgoLined
    \KwIn{An initial population of particles $(X^1_0,\ldots,X^N_0)\in E^N$, a maximum time $T\in\N$, a proposal distribution $\Theta$ and an acceptance function $h$}
    \KwOut{$T$ estimators $\hat{\varphi}_t$ of $E_\pi[\varphi(X)]$}

    \For {$t=0$ to $T-1$}{
    \For{$i=1$ to $N$}{

    Draw $Y^i_t\sim \Theta_{\hat{\mu}_t^N}(\cdot|X^i_t)$ a proposal for the new state of particle $i$\;
    
        Compute $\alpha_{\hat{\mu}^N_t}(X^i_t,Y^i_t) = \frac{\Theta_{\hat{\mu}^N_t}(X^i_t|Y^i_t)\pi(Y^i_t)}{\Theta_{\hat{\mu}^N_t}(Y^i_t|X^i_t)\pi(X^i_t)}$\;
        Store $W_t^i = \pi(Y^i_t)/\Theta_{\hat{\mu}^N_t}(Y^i_t|X^i_t) $ and $Y_t^i$\;

        Draw $U^i_t\sim\mathcal{U}([0,1])$\;
        
        \eIf{$U^i_t\leq h\big(\alpha_{\hat{\mu}^N_t}(X^i_t,Y^i_t)\big)$}{
            
            Set $X^i_{t+1}=Y^i_t$\;
            }{
            Set $X^i_{t+1}=X^i_t$ \;
        }
    }}
    Set $\hat\varphi_t = \sum_{i} W_t^i\varphi(Y_t^i). $
\caption{Collective Monte Carlo with IS output}
\label{algo:cmcIS}
\end{algorithm}
\end{minipage}
\end{center}

At each step, we can store $Y_t^i$ the proposed state for each particle, and the numerator of the acceptance ratio $W_t^i = \pi(Y^i_t)/\Theta_{\hat{\mu}^N_t}(Y^i_t|X^i_t) $. An estimator of $\E_\pi[\varphi(X)] $ is then $\sum_i W_t^i \varphi(Y_t^i) $. This term corresponds exactly to the importance ratio with $\pi$ as target and $\Theta_{\hat{\mu}^N_t}$ as proposal distribution. This is similar in a sense to PMC, but the main difference is that \emph{we do not propagate $Y_t^i$} unless it is accepted, and thus the weight cannot degenerate.

This addition does not interfere with the choice of the interaction $\Theta$, and ensures that the resulting estimator is unbiased. We do not have to use all the points $(Y_t^i)$ as we can only use some of the $t$. Notice that each of the $(Y^i_t)^{}_i$ are independent by construction, but that $(Y^i_t)^{}_{t,i}$ are not independent.

To measure the quality of a numerical method, it is common to use the Effective Sample Size (ESS) \cite{robert2013monte} that represents the equivalent number of i.i.d. samples needed to build an estimator of $\E_\pi[X]$ with the same variance --- notice that this quantity is mostly informative. In Importance sampling, the ESS is commonly estimated by $ESS(t) = (\sum_i (W_t^i)^2)^{-1}$, Interestingly enough, Theorem \ref{thm:convergencecmc} shows that precisely, as the number of particles increases,
\[\E[(\sum_i (W_t^i)^2)^{-1}]/N \rightarrow_{N \to \infty} 1,\]
as it tends to an i.i.d. sample distributed according to $\pi$. The precise convergence speed would depend on the constants appearing in the theorems. 

As usually in importance sampling, this allows to estimate at no cost the normalizing constant --- or in a Bayesian framework, the marginal likelihood of a mode : $\int_x \pi(x) \mathrm{d}x$, where we recall that $\pi$ is non-normalised. Classical Metropolis-Hastings methods cannot estimate this quantity, while importance sampling methods (SMC, PMC) can provide an estimator whose variance may be too large to be usable. Figure \ref{fig:normalizingconstant} shows the result for the experiments of Section \ref{sec:experiments}. These are only results on two ``simple'' targets, further numerical simulations would be needed to confirm the efficiency of our method to estimate this quantity.

\begin{figure}[ht]
    \centering
    \includegraphics[scale=.25]{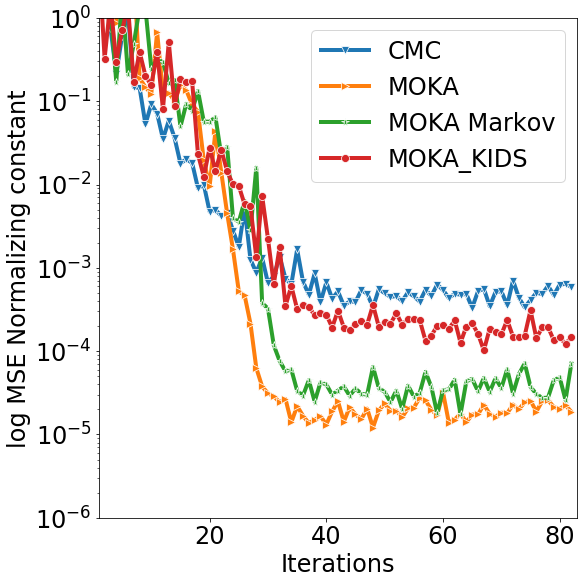}
    \includegraphics[scale=.25]{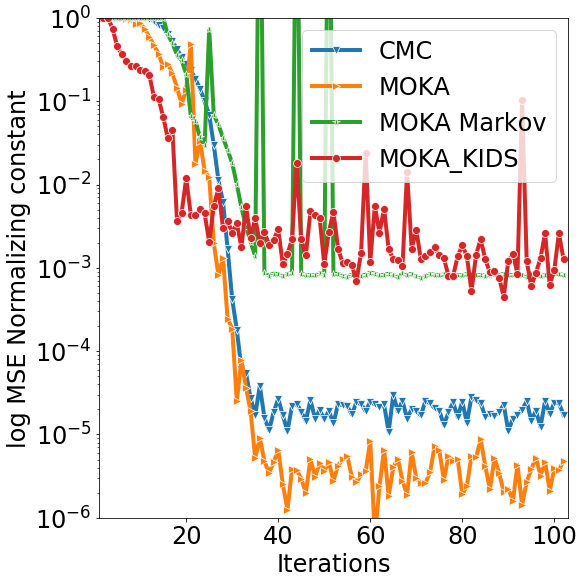}
    \caption{Log-MSE of the normalizing constant computed for the banana-shaped distribution (left) and the 8-dimensional Gaussian mixture (right).}
    \label{fig:normalizingconstant}
\end{figure}

\section{Numerical experiments on a Cauchy mixture}\label{appendix:Cauchy}

\begin{figure}[ht]
    \centering
    \includegraphics[scale=.3]{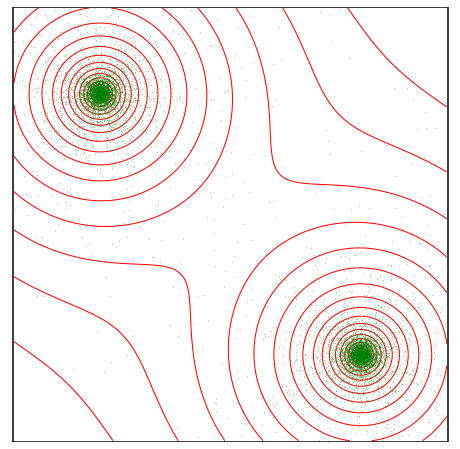}
    \caption{True sample and levels of the mixture of Cauchy target distribution}
    \label{fig:cauchytrue}
\end{figure}

Our last target is a simple mixture of two two-dimensional Cauchy distributions with means $(0.2,0.8)$ and $(0.8,0.2)$ and the same scale parameter $0.01$. We represent in Figure~\ref{fig:cauchytrue} a sample and the level sets of the target distribution. This distribution has ``heavy tails'' which should reduce the efficiency of our uniform proposals on balls. For this example, we choose $0.1$ for the radius of the unique proposals, and $0.01, 0.05, 0.1,0.3$ for the MoKAs. The results presented in Figure \ref{fig:CauchyED} show that none of the method reaches the minimal value possible for the Energy distance. Furthermore, MoKA-KIDS dramatically fails, this is probably because of the underlying assumptions of the deconvolution algorithm we use. This last example confirms the reliability of MoKA-Markov as an efficient and secure method. The variance of the energy distance is given in Tabular \ref{tab:tablevarED_cauchy}.

\begin{figure}[ht]
    \centering
    \includegraphics[scale=.4]{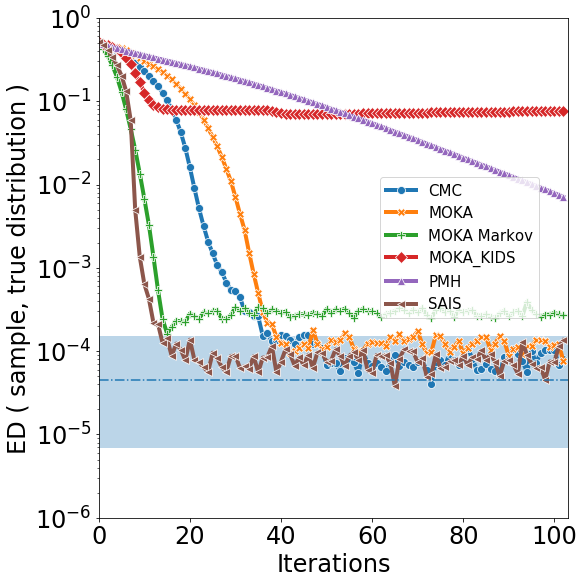}
    \caption{Mean energy distance to a true sample on 10 repetitions of the algorithm, Cauchy mixture. The dotted line represents the mean distance between two iid exact samples, computed over $100$ independent realisations, and the coloured area is the corresponding $90\%$ prediction interval.}
    \label{fig:CauchyED}
\end{figure}

\begin{table}[ht]
    \centering
\begin{tabular}{ccccccc}
& CMC & MoKA & MoKA-Markov & MoKA-KIDS & PMH & SAIS \\
Cauchy mixture & 8.01e-05 & \textbf{6.59e-05} & 6.86e-05 & 7.40e-02 & 3.31e-4 & 8.66e-05 \\
\end{tabular}
    \caption{Variance of the Energy distance at the last iteration for the Cauchy target.}
    \label{tab:tablevarED_cauchy}
\end{table}

\end{appendix}

\end{document}